\documentclass[final,onefignum,onetabnum]{siamart190516}

\usepackage{lipsum}
\usepackage{amsfonts}
\usepackage{graphicx}
\usepackage{epstopdf}
\usepackage{algorithmic}
\ifpdf
  \DeclareGraphicsExtensions{.eps,.pdf,.png,.jpg}
\else
  \DeclareGraphicsExtensions{.eps}
\fi

\newsiamremark{remark}{Remark}
\newsiamremark{hypothesis}{Hypothesis}
\crefname{hypothesis}{Hypothesis}{Hypotheses}
\newsiamthm{claim}{Claim}

\headers{Entropy Stable \MakeLowercase{$p$}-Nonconforming Discretizations} 
{D.\ C.\ Del Rey Fern\'andez {\em et~al.}}

\title{Entropy Stable \MakeLowercase{$p$}-Nonconforming Discretizations with the Summation-by-Parts Property for the Compressible Euler equations
\thanks{Associated with this paper is a companion NASA technical report~\cite{Fernandez2018_TM} that contians additional details omitted in this paper for brevity.}}

\author{David C.\ Del Rey Fern\'andez\thanks{National Institute of Aerospace and 
        Computational AeroSciences Branch, NASA Langley Research Center, Hampton, VA, United States, 
                (\email{dcdelrey@gmail.com}).}
  \and Mark H.\ Carpenter\thanks{Computational AeroSciences Branch, NASA Langley Research Center, Hampton, VA, United States, 
                (\email{mark.h.carpenter@nasa.gov}).}
  \and Lisandro Dalcin\thanks{King Abdullah University of Science and Technology (KAUST), 
  Computer Electrical and Mathematical Science and Engineering Division (CEMSE), 
  Extreme Computing Research Center (ECRC), Thuwal, Saudi Arabia,  
                (\email{dalcinl@gmail.com}, \email{diego.rojasblanco@kaust.edu.sa}, \email{stefano.zampini@kaust.edu.sa}, \email{matteo.parsani@kaust.edu.sa}).}
    \and Lucas Fredrich\thanks{Mathematical Institute, University of Cologne, North Rhine-Westphalia, Germany,
                (\email{lfriedri@math.uni-koeln.de}, \email{ggassner@math.uni-koeln.de}).}
    \and Diego Rojas\footnotemark[4]
    \and Andrew R.\ Winters\thanks{Department of Mathematics (MAI), Link\"{o}ping University, Sweden,
                (\email{andrew.ross.winters@liu.se}).}
    \and Gregor J.\ Gassner\footnotemark[5]
    \and Stefano Zampini\footnotemark[4]
    \and Matteo Parsani\footnotemark[4]
}

\usepackage{amsopn}
\DeclareMathOperator{\diag}{diag}
\usepackage{xspace}
\usepackage{amsmath}
\usepackage{amsfonts}
\usepackage{amssymb}
\usepackage{bm}
\usepackage{subfigure}
\usepackage{xspace}
\usepackage{afterpage}
\usepackage{color}
\usepackage{soul}
\usepackage{cancel}
\usepackage{gensymb}
\usepackage{tikz}
\usetikzlibrary{matrix,backgrounds,shapes}
\usetikzlibrary{positioning}
\usetikzlibrary{fit}
\usetikzlibrary{plotmarks}
\usetikzlibrary{decorations.pathreplacing}
\usepackage{pgfplots}
\usetikzlibrary{shapes}
\usetikzlibrary{positioning,arrows}
\usetikzlibrary{fadings,shapes.arrows,shadows}  
\usetikzlibrary{arrows.meta}
\definecolor{darkorange}{rgb}{1.0, 0.55, 0.0}
\definecolor{do}{rgb}{1.0, 0.55, 0.0}
\definecolor{Royalblue}{rgb}{0.254,0.41,0.88}
\usepackage{xcolor}

\definecolor{darkorange}{rgb}{1.0, 0.55, 0.0}
\definecolor{royalblue}{rgb}{0.254,0.41,0.88}

\newcommand{\Tr}{\ensuremath{^{\mr{T}}}}

\newcommand{\mr}[1]{\ensuremath{\mathrm{#1}}}
\newcommand{\fnc}[1]{\ensuremath{\mathcal{#1}}}
\newcommand{\bfnc}[1]{\ensuremath{\bm{\mathcal{#1}}}}
\newcommand{\mat}[1]{\ensuremath{\mathsf{#1}}}
\newcommand{\etal}[0]{{\em et~al.\@}\xspace}
\newcommand{\eg}[0]{{e.g.\@}\xspace}
\newcommand{\ie}[0]{{i.e.\@}\xspace}
\newcommand{\etc}[0]{{etc.\@}\xspace}
\newcommand{\Theorem}[0]{Theorem}
\newcommand{\Eq}[0]{Eq.}
\newcommand{\Th}[0]{\ensuremath{^{\mathrm{th}}}}
\newtheorem{assume}{Assumption}
\newtheorem{thrm}{Theorem}

\newcommand{\pL}[0]{\ensuremath{p_{\mathrm{L}}}}
\newcommand{\pH}[0]{\ensuremath{p_{\mathrm{H}}}}
\newcommand{\qL}[0]{\ensuremath{\bm{q}_{\mathrm{L}}}}
\newcommand{\qH}[0]{\ensuremath{\bm{q}_{\mathrm{H}}}}

\newcommand{\rmL}[0]{\mathrm{L}}
\newcommand{\rmH}[0]{\mathrm{H}}
\newcommand{\rmR}[0]{\mathrm{R}}

\newcommand{\xm}[1]{\ensuremath{x_{#1}}}
\newcommand{\xil}[1]{\ensuremath{\xi_{#1}}}
\newcommand{\alphal}[1]{\ensuremath{\alpha_{#1}}}
\newcommand{\betal}[1]{\ensuremath{\beta_{#1}}}
\newcommand{\Nl}[1]{\ensuremath{N_{#1}}}
\newcommand{\bxil}[1]{\ensuremath{\bm{\xi}_{#1}}}
\newcommand{\bxili}[2]{\ensuremath{\bm{\xi}_{#1}^{(#2)}}}

\newcommand{\bmui}[1]{\ensuremath{\bm{u}^{(#1)}}}

\newcommand{\Q}[0]{\ensuremath{\bm{\fnc{Q}}}}
\newcommand{\Jk}[0]{\ensuremath{\fnc{J}_{\kappa}}}
\newcommand{\Qk}[0]{\ensuremath{\bm{\fnc{Q}_{\kappa}}}}
\newcommand{\Jdxildxm}[2]{\ensuremath{\Jk\frac{\partial\xil{#1}}{\partial\xm{#2}}}}
\newcommand{\Fxm}[1]{\ensuremath{\bm{\fnc{F}}_{\xm{#1}}}}

\newcommand{\fxm}[1]{\ensuremath{\fnc{F}_{\xm{#1}}}}
\newcommand{\Um}[1]{\ensuremath{\fnc{U}_{#1}}}

\newcommand{\nxm}[1]{\ensuremath{n_{\xm{#1}}}}
\newcommand{\nxil}[1]{\ensuremath{n_{\xil{#1}}}}
\newcommand{\GB}[0]{\ensuremath{\bm{\fnc{G}}^{(B)}}}
\newcommand{\Gzero}[0]{\ensuremath{\bm{\fnc{G}}^{(0)}}}

\newcommand{\DxiloneD}[1]{\ensuremath{\mat{D}_{\xil{#1}}^{(1D)}}}
\newcommand{\PxiloneD}[1]{\ensuremath{\mat{P}_{\xil{#1}}^{(1D)}}}
\newcommand{\QxiloneD}[1]{\ensuremath{\mat{Q}_{\xil{#1}}^{(1D)}}}
\newcommand{\SxiloneD}[1]{\ensuremath{\mat{S}_{\xil{#1}}^{(1D)}}}
\newcommand{\ExiloneD}[1]{\ensuremath{\mat{E}_{\xil{#1}}^{(1D)}}}
\newcommand{\txilalpha}[1]{\ensuremath{\bm{t}_{\alphal}}}
\newcommand{\txilbeta}[1]{\ensuremath{\bm{t}_{\betal}}}
\newcommand{\DoneD}[0]{\mat{D}^{(1D)}}
\newcommand{\PoneD}[0]{\mat{P}^{(1D)}}

\newcommand{\Imat}[1]{\ensuremath{\mat{I}_{#1}}}
\newcommand{\M}[0]{\ensuremath{\mat{P}}}
\newcommand{\barM}[0]{\ensuremath{\overline{\mat{P}}}}
\newcommand{\Dxm}[1]{\ensuremath{\mat{D}_{\xm{#1}}}}
\newcommand{\Dxil}[1]{\ensuremath{\mat{D}_{\xil{#1}}}}
\newcommand{\Qxil}[1]{\ensuremath{\mat{Q}_{\xil{#1}}}}
\newcommand{\Sxil}[1]{\ensuremath{\mat{S}_{\xil{#1}}}}
\newcommand{\Exil}[1]{\ensuremath{\mat{E}_{\xil{#1}}}}
\newcommand{\Ealphal}[1]{\ensuremath{\mat{E}_{\alphal{#1}}}}
\newcommand{\Ebetal}[1]{\ensuremath{\mat{E}_{\betal{#1}}}}
\newcommand{\barEalphal}[1]{\ensuremath{\bar{\mat{E}}_{\alphal{#1}}}}
\newcommand{\barEbetal}[1]{\ensuremath{\bar{\mat{E}}_{\betal{#1}}}}
\newcommand{\Ralphal}[1]{\ensuremath{\mat{R}_{\alphal{#1}}}}
\newcommand{\Rbetal}[1]{\ensuremath{\mat{R}_{\betal{#1}}}}
\newcommand{\barRalphal}[1]{\ensuremath{\overline{\mat{R}}_{\alphal{#1}}}}
\newcommand{\barRbetal}[1]{\ensuremath{\overline{\mat{R}}_{\betal{#1}}}}
\newcommand{\Porthol}[1]{\ensuremath{\mat{P}_{\perp\xil{#1}}}}
\newcommand{\barPorthol}[1]{\ensuremath{\overline{\mat{P}}_{\perp\xil{#1}}}}

\newcommand{\barDxil}[1]{\ensuremath{\overline{\mat{D}}_{\xil{#1}}}}
\newcommand{\barQxil}[1]{\ensuremath{\overline{\mat{Q}}_{\xil{#1}}}}
\newcommand{\barSxil}[1]{\ensuremath{\overline{\mat{S}}_{\xil{#1}}}}
\newcommand{\barExil}[1]{\ensuremath{\overline{\mat{E}}_{\xil{#1}}}}
\newcommand{\Ck}[0]{\ensuremath{C_{\kappa}}}
\newcommand{\matFxm}[3]{\ensuremath{\mat{F}_{\xm{#1}}\left(#2,#3\right)}}
\newcommand{\matFxmtilde}[3]{\ensuremath{\tilde{\mat{F}}_{\xm{#1}}\left(#2,#3\right)}}
\newcommand{\matFxmscai}[4]{\ensuremath{\mat{F}_{\xm{#1}}^{[#4]}\left(#2,#3\right)}}

\newcommand{\matFxmi}[4]{\ensuremath{\mat{F}_{\xm{#1}}^{#4}\left(#2,#3\right)}}
\newcommand{\qk}[1]{\ensuremath{\bm{q}_{#1}}}
\newcommand{\qki}[2]{\ensuremath{\bm{q}_{#1}^{#2}}}
\newcommand{\vki}[2]{\ensuremath{\bm{v}_{#1}^{#2}}}
\newcommand{\tildeqi}[1]{\ensuremath{\tilde{\bm{q}}^{#1}}}
\newcommand{\wki}[2]{\ensuremath{\bm{w}_{#1}^{#2}}}
\newcommand{\tildewi}[1]{\ensuremath{\tilde{\bm{w}}^{#1}}}
\newcommand{\vk}[1]{\ensuremath{\bm{v}_{#1}}}

\newcommand{\fxmsc}[3]{\ensuremath{\bm{f}_{\xm{#1}}^{sc}\left(#2,#3\right)}}
\newcommand{\fxmsctilde}[3]{\ensuremath{\tilde{\bm{f}}_{\xm{#1}}^{sc}\left(#2,#3\right)}}
\newcommand{\ones}[1]{\ensuremath{\bm{1}_{#1}}}
\newcommand{\barones}[1]{\ensuremath{\overline{\bm{1}}_{#1}}}
\newcommand{\wk}[1]{\ensuremath{\bm{w}_{#1}}}
\newcommand{\psixmk}[2]{\ensuremath{\bm{\psi}_{\xm{#1}}^{#2}}}
\newcommand{\psixmki}[3]{\ensuremath{\left(\bm{\psi}_{\xm{#1}}^{#2}\right)^{(#3)}}}
\newcommand{\psixmtildei}[2]{\ensuremath{\tilde{\bm{\psi}}_{\xm{#1}}^{(#2)}}}
\newcommand{\psixmtilde}[1]{\ensuremath{\left(\tilde{\bm{\psi}}_{\xm{#1}}\right)}}
\newcommand{\matJk}[1]{\ensuremath{{\color{orange}\mat{J}}_{#1}}}
\newcommand{\barmatJk}[1]{\ensuremath{{\color{orange}\overline{\mat{J}}}_{#1}}}
\newcommand{\matAlmk}[3]{\ensuremath{\left[{\color{blue}\fnc{J}\frac{\partial\xil{#1}}{\partial\xm{#2}}}\right]_{#3}}}
\newcommand{\barmatAlmk}[3]{\ensuremath{\overline{\left[{\color{blue}\fnc{J}\frac{\partial\xil{#1}}{\partial\xm{#2}}}\right]}_{#3}}}
\newcommand{\matAlmkLH}[5]{\ensuremath{\left[{\color{#5}\fnc{J}\frac{\partial\xil{#1}}{\partial\xm{#2}}}\right]_{#3}^{#4}}}
\newcommand{\EHtoLm}[1]{\ensuremath{\mat{E}^{\rmH\rm{to}\rmL}\left(\matAlmkLH{1}{#1}{}{}{red}\right)}}
\newcommand{\ELtoHm}[1]{\ensuremath{\mat{E}^{\rmL\rm{to}\rmH}\left(\matAlmkLH{1}{#1}{}{}{red}\right)}}
\newcommand{\barEHtoLm}[1]{\ensuremath{\overline{\mat{E}}^{\rmH\rm{to}\rmL}\left(\matAlmkLH{1}{#1}{}{}{red}\right)}}
\newcommand{\barELtoHm}[1]{\ensuremath{\overline{\mat{E}}^{\rmL\rm{to}\rmH}\left(\matAlmkLH{1}{#1}{}{}{red}\right)}}
\newcommand{\ILtoH}[0]{\ensuremath{\mat{I}_{\rmL\mr{to}\rmH}}}
\newcommand{\IHtoL}[0]{\ensuremath{\mat{I}_{\rmH\mr{to}\rmL}}}
\newcommand{\ILtoHoneD}[0]{\ensuremath{\mat{I}_{\rmL\mr{to}\rmH}^{(1D)}}}
\newcommand{\IHtoLoneD}[0]{\ensuremath{\mat{I}_{\rmH\mr{to}\rmL}^{(1D)}}}
\newcommand{\RL}[0]{\ensuremath{\mat{R}_{\rmL}}}
\newcommand{\RH}[0]{\ensuremath{\mat{R}_{\rmH}}}

\newcommand{\Efm}[4]{\ensuremath{\mat{E}^{#1\mathrm{to}#2}\left(\left[{\color{red}\fnc{J}\frac{\partial\xil{#3}}{\partial\xm{#4}}}\right]\right)}}
\newcommand{\barEfm}[4]{\ensuremath{\overline{\mat{E}}^{\mathrm{#1to#2}}\left(\overline{\left[{\color{red}\fnc{J}\frac{\partial\xil{#3}}{\partial\xm{#4}}}\right]}\right)}}
\newcommand{\Ok}[0]{\ensuremath{\Omega_{\kappa}}}
\newcommand{\pOk}[0]{\ensuremath{\partial\Omega_{\kappa}}}
\newcommand{\Ohat}[0]{\ensuremath{\hat{\Omega}}}
\newcommand{\Ohatk}[0]{\ensuremath{\hat{\Omega}_{\kappa}}}
\newcommand{\Ghatk}[0]{\ensuremath{\hat{\Gamma}_{\kappa}}}
\newcommand{\Ghat}[0]{\ensuremath{\hat{\Gamma}}}
\newcommand{\pOhatk}[0]{\ensuremath{\partial\hat{\Omega}_{\kappa}}}

\newcommand{\Y}[0]{\ensuremath{\mat{Y}}}
\newcommand{\dissL}[0]{\ensuremath{\bm{diss}_{\mathrm{L}}}}
\newcommand{\dissH}[0]{\ensuremath{\bm{diss}_{\mathrm{H}}}}
\newcommand{\dfdwL}[0]{\ensuremath{\left|\frac{\partial\bfnc{F}_{\xil{1}}}{\partial\bfnc{W}}\right|_{\mathrm{L}}}}
\newcommand{\dfdwH}[0]{\ensuremath{\left|\frac{\partial\bfnc{F}_{\xil{1}}}{\partial\bfnc{W}}\right|_{\mathrm{H}}}}

\newcommand{\uk}[0]{\ensuremath{\bm{u}_{\kappa}}}
\newcommand{\ur}[0]{\ensuremath{\bm{u}_{r}}}
\newcommand{\uki}[1]{\ensuremath{\bm{u}_{\kappa}^{#1}}}
\newcommand{\uri}[1]{\ensuremath{\bm{u}_{r}^{#1}}}
\newcommand{\uL}[0]{\ensuremath{\bm{u}_{\rmL}}}
\newcommand{\uH}[0]{\ensuremath{\bm{u}_{\rmH}}}
\newcommand{\qbark}[0]{\ensuremath{\bar{\bm{q}}_{\kappa}}}
\newcommand{\vbark}[0]{\ensuremath{\bar{\bm{v}}_{\kappa}}}
\newcommand{\gfk}[2]{\ensuremath{\bm{g}_{#1}^{#2}}}

\newcommand{\gL}[0]{\ensuremath{\bm{g}_{\mathrm{L}}}}
\newcommand{\gH}[0]{\ensuremath{\bm{g}_{\mathrm{H}}}}

\newcommand{\Qc}[0]{\ensuremath{\bfnc{Q}_{c}}}
\newcommand{\qh}[0]{\ensuremath{\bm{q}_{h}}}
\newcommand{\vh}[0]{\ensuremath{\bm{v}_{h}}}
\newcommand{\Ballk}[0]{\ensuremath{B_{\kappa}}}

\newcommand{\Jtilde}[0]{\ensuremath{{\color{orange}\mat{J}}}}
\newcommand{\utilde}[0]{\ensuremath{\widetilde{\bm{u}}}}
\newcommand{\qtilde}[0]{\ensuremath{\widetilde{\bm{q}}}}
\newcommand{\qtildei}[1]{\ensuremath{\widetilde{\bm{q}}^{#1}}}
\newcommand{\wtilde}[0]{\ensuremath{\widetilde{\bm{w}}}}
\newcommand{\stilde}[0]{\ensuremath{\widetilde{\bm{s}}}}
\newcommand{\wtildei}[1]{\ensuremath{\widetilde{\bm{w}}^{(#1)}}}
\newcommand{\psitildem}[1]{\ensuremath{\widetilde{\bm{\psi}_{#1}}}}
\newcommand{\Mtilde}[0]{\ensuremath{\widetilde{\mat{P}}}}
\newcommand{\DtildeH}[1]{\ensuremath{\widetilde{\mat{D}}_{\xil{#1}}}}
\newcommand{\Dtildel}[1]{\ensuremath{\widetilde{\mat{D}}_{\xil{#1}}}}
\newcommand{\Qtildel}[1]{\ensuremath{\widetilde{\mat{Q}}_{\xil{#1}}}}
\newcommand{\Stildel}[1]{\ensuremath{\widetilde{\mat{S}}_{\xil{#1}}}}
\newcommand{\Etildel}[1]{\ensuremath{\widetilde{\mat{E}}_{\xil{#1}}}}
\newcommand{\tildebarmatAlm}[2]{\ensuremath{\left[{\color{blue}\fnc{J}\frac{\partial\xil{#1}}{\partial \xm{#2}}}\right]}}
\newcommand{\tildeone}[0]{\ensuremath{\tilde{\bm{1}}}}

\newcommand{\eonel}[1]{\ensuremath{\bm{e}_{1_{#1}}}}
\newcommand{\eNl}[1]{\ensuremath{\bm{e}_{\Nl{#1}}}}
\newcommand{\bmxi}[1]{\ensuremath{\bm{\xi}^{(#1)}}}
\newcommand{\U}[0]{\ensuremath{\fnc{U}}}
\newcommand{\barmatAlmkLH}[5]{\ensuremath{\overline{\left[{\color{#5}\fnc{J}\frac{\partial\xil{#1}}{\partial\xm{#2}}}\right]}_{#3}^{#4}}}
\newcommand{\QoneDL}[0]{\ensuremath{\mat{Q}^{(1D)}_{\rmL}}}
\newcommand{\EoneDL}[0]{\ensuremath{\mat{E}^{(1D)}_{\rmL}}}
\newcommand{\MM}[0]{\ensuremath{\mat{M}}}
\newcommand{\OhatL}[0]{\ensuremath{\hat{\Omega}_{\rmL}}}
\newcommand{\GammaL}[0]{\ensuremath{\hat{\Gamma}_{\rmL}}}

\newcommand{\QL}[0]{\ensuremath{\bfnc{Q}_{\rmL}}}
\newcommand{\QR}[0]{\ensuremath{\bfnc{Q}_{\rmR}}}
\newcommand{\matAk}[1]{\ensuremath{\left[\fnc{A}\right]_{#1}}}
\newcommand{\barmatAk}[1]{\ensuremath{\overline{\left[\fnc{A}\right]}_{#1}}}
\newcommand{\xR}[0]{\ensuremath{x_{\rm{R}}}}
\newcommand{\xL}[0]{\ensuremath{x_{\rm{L}}}}

\newcommand{\pP}[0]{\ensuremath{p_{\M}}}
\newcommand{\Dhatm}[1]{\hat{\mat{D}}_{#1}}
\newcommand{\Qhatm}[1]{\hat{\mat{Q}}_{#1}}
\newcommand{\Ehatm}[1]{\hat{\mat{E}}_{#1}}
\newcommand{\barDhatm}[1]{\overline{\hat{\mat{D}}}_{#1}}
\newcommand{\barQhatm}[1]{\overline{\hat{\mat{Q}}}_{#1}}
\newcommand{\barEhatm}[1]{\overline{\hat{\mat{E}}}_{#1}}
\newcommand{\Vol}{\ensuremath{\rm{Vol}}}

\newcommand{\Dtildelm}[2]{\ensuremath{\widetilde{\mat{D}}_{#1,#2}}}

\ifpdf
\hypersetup{
  pdftitle={Entropy Stable p-Nonconforming Discretizations with the Summation-by-Parts Property for the 
  Compressible Euler equations}
  pdfauthor={D.\ C.\ Del Rey Fern\'andez, M.\ H.\ Carpenter, L.\ Dalcin, L.\ Fredrich, D.\ Rojas, A.\ R.\ Winters, G.\ J.\ Gassner, S.\ Zampini, M.\ Parsani}
}
\fi

\usepackage{titlesec}
\titlespacing*{\section}
{0pt}{5pt}{5pt}
\titlespacing*{\subsection}
{0pt}{5pt}{5pt}
\usepackage{etoolbox}
\newtoggle{proofs}
\togglefalse{proofs}

\begin{document}
\setlength{\belowdisplayskip}{3pt} \setlength{\belowdisplayshortskip}{3pt}
\setlength{\abovedisplayskip}{3pt} \setlength{\abovedisplayshortskip}{3pt}

\maketitle

\begin{abstract}
The entropy conservative/stable algorithm of Friedrich~\etal (2018) for hyperbolic conservation laws 
on nonconforming $p$-refined/coarsened Cartesian grids, is extended to curvilinear grids for the 
compressible Euler equations. The primary focus is on constructing appropriate coupling procedures across 
the curvilinear nonconforming interfaces. A simple and flexible approach is proposed that uses 
interpolation operators from one element to the other. On the element faces, the analytic metrics are used to 
construct coupling terms, while metric terms in the volume are approximated to satisfy a discretization 
of the geometric conservation laws. The resulting scheme is entropy conservative/stable, elementwise 
conservative, and freestream preserving. The accuracy and stability properties of the resulting 
numerical algorithm are shown to be comparable to those of the original conforming scheme 
($\sim p+1$ convergence) in the context of the isentropic Euler vortex and 
the inviscid Taylor--Green vortex problems on manufactured high order grids.
\end{abstract}


\section{Introduction}

High-order accurate finite element methods (FEM) provide an efficient approach to achieve high solution 
accuracy. Their computational kernels are arithmetically dense making them compatible with the current
and future highly concurrent hardware. In addition, they are amenable to $h$-, $p$-, and $r$-refinement algorithms, thus facilitating 
the redistribution of the degrees of freedom to better resolve multiscale physics.

High-order accurate methods are known to be more efficient than low-order methods for linear wave 
propagation~\cite{Kreiss1972,Swartz1974}.  Although high order discretizations have a long history of 
development, their application to nonlinear partial differential equations (PDEs) 
for practical applications has been limited by robustness issues. Thus, nominally second-order
accurate discretization operators are typically used in commercial and industrial software.

The summation-by-parts (SBP) framework provides a systematic and discretization-agnostic methodology for 
constructing arbitrarily high-order accurate and provably stable numerical methods for linear and variable coefficient 
linear problems (see, for instance, the survey papers~\cite{Fernandez2014,Svard2014}). SBP operators can be 
viewed as matrix difference operators that are mimetic of integration by parts in that they have a telescoping property. Discrete 
stability over the whole domain is achieved by combining the SBP mechanics with suitable inter-element 
coupling procedures and boundary conditions 
(e.g., the simultaneous approximation terms (SATs) \cite{Carpenter1994,Nordstrom1999}).

For nonlinear problems, progress towards provably stable algorithms applicable to practical problems 
has been much slower. However, certain class of nonlinear conservation laws come endowed 
with a complementary inequality statement (equality for smooth solutions) on the mathematical entropy 
(see, for instance, \cite{dafermos-book-2010}). Therefore, it is desirable for the numerical method to 
satisfy a corresponding discrete analogue. This can then be used to prove 
nonlinear stability of the numerical scheme \cite{Tadmor1987entropy,Tadmor2003}.
Along these lines, a productive trajectory was initiated by Tadmor~\cite{Tadmor1987entropy} who 
constructed entropy conservative low-order finite volume schemes that achieve entropy 
conservation by using two-point flux functions that when contracted with the entropy variables 
result in a telescoping entropy flux. Entropy stability was ensured by adding appropriate dissipation. 
Through the telescoping property, the continuous $L^{2}$ entropy stability analysis is 
mimicked by the semi-discrete stability analysis (for a review of these ideas see 
Tadmor~\cite{Tadmor2003}). 
Tadmor's basic idea has led to the construction of a number of high order and
low order entropy stable schemes (see, for instance, \cite{Fjordholm2012,Ray2016}). An alternative approach, developed by Olsson and Oliger~\cite{Olsson1994}, 
Gerritsen and Olsson~\cite{Margot1996} 
and Yee~\etal~\cite{Yee2000} (see also~\cite{Sandham2002,Bjorn2018}), relies on choosing entropy functions 
that result in a homogeneity property on the Euler fluxes. By using this property, splittings of the Euler fluxes are 
constructed such that when contracted with the entropy variables result in stability estimates analogous in form 
to energy estimates obtained for linear PDEs. Thus, discretizing the resulting split form using SBP 
operators, the nonlinear stability analysis performed at the continuous level is mimicked at the 
semi-discrete level.

A complementary extension of Tadmor's ideas to finite domains was initiated by Fisher and coworkers who 
combined the SBP framework, using classical finite difference SBP operators, with Tadmor's two-point flux~\cite{Fisher2012phd,Fisher2013,FisherCarpenter2013JCPb}. 
The resulting schemes follow the continuous entropy analysis and can be shown to be entropy 
conservative and be made entropy stable by adding appropriate interface dissipation. This nonlinearly stable approach inherits 
all the mechanics of linear SBP schemes for the imposition of boundary conditions and inter-element 
coupling and therefore gives a systematic methodology for discretizing problems on complex 
geometries~\cite{Carpenter2014,Parsani2015b,Parsani2015}. Moreover, by constructing schemes that are 
discretely mimetic of the continuous stability analysis, the need to assume exact integration in the 
stability proofs is eliminated (see for example the work of Hughes~\etal~\cite{Hughes1986}). These ideas 
have been extended to include collocated spectral elements~\cite{Carpenter2014}, fully- and semi-staggered 
spectral elements~\cite{Parsani2016,Carpenter2016}, Cartesian, semi-staggered, nonconforming 
$p$-refinement~\cite{Carpenter2016}, WENO spectral collocation~\cite{Yamaleev2017}, multidimensional SBP 
operators~\cite{Crean2018,Chen2017}, multidimensional staggered SBP operators~\cite{Fernandez2019_staggered}, modal 
decoupled SBP operators~\cite{Chan2018}, and fully discrete explicit entropy stable schemes~\cite{Friedrich2019,ranocha2019relaxation}, 
as well as to a number of PDEs besides the compressible Euler and Navier--Stokes equations (for example the 
magnetohydrodynamics~\cite{Winters2017} and the shallow 
water~\cite{Winters2015} equations).

A necessary constraint for entropy stability on curvilinear meshes is satisfaction of 
discrete geometric conservation law (GCL) 
conditions~\cite{Fisher2012phd,Fisher2013,Carpenter2014,Parsani2016}. Well documented procedures exist for 
generating discrete transformation metrics on conforming meshes that satisfy the GCL 
conditions~\cite{Thomas1979,Vinokur2002a}.  These procedures extend immediately to the nonconforming case 
provided that the polynomial order of the method is sufficient to analytically resolve the element surface 
transformation metrics; a condition that is naturally satisfied if the polynomial orders of the geometry, 
$p_g$ and discretization, $p$, are related by the inequality $p_g \le \frac{p+1}{2}$.  
While it is possible to generate body fitted meshes with full geometric surface order ($p_g = p$) and 
reduced off body order $p_g \le \frac{p+1}{2}$, this forces undue complexity on the already 
complex task of grid generation. Enforcing the $p_g = p$ grid constraint on near-body elements, while
avoiding grid line cross-over or negative Jacobians, increases in complexity with mesh aspect ratio, and
is virtually impossible to achieve in the high Reynolds number limit of realistic aerodynamic configurations.

Herein, the primary objective is to construct entropy stable discretizations suitable for the mechanics of 
high order FEM $p$-adaptivity, applicable for general meshes containing hexahedral polynomial elements of 
full geometric order (i.e., $p_g \le p$).  Initial nonconforming efforts focused on Cartesian,
semi-staggered spectral collocation operators~\cite{Carpenter2016}, but identifying a curvilinear extension 
has proven difficult.  Thus, this work extends to curvilinear coordinates the Cartesian grid work previously reported 
by Friedrich~\etal~\cite{Friedrich2018}, and primarily focuses on constructing appropriate coupling 
procedures across curvilinear nonconforming interfaces.

Many novel contributions are included in this work.  They are summarized as follows:
\begin{itemize}
\item A general, yet simple entropy stable nonconforming algorithm is proposed in curvilinear coordinates 
for the compressible Euler equations that
\begin{itemize}
\item Encapsulates and generalizes several approaches for coupling curvilinear nonconforming interfaces
\item Formalizes necessary conditions for entropy conservation at curvilinear nonconforming interfaces
\item Elegantly handles various mesh generation strategies including elements of full geometric order ($p_g \le p$)
\item Extends the metric approximation approach of Crean~\etal~\cite{Crean2018} to curvilinear nonconforming interfaces 
\item Satisfies the discrete GCL conditions, therefore ensuring freestream preservation
\item Exploits the generality of distinct surface and volume metrics that couple through the GCL constraint
\end{itemize}
\item Numerical evidence is provided that supports the assertion that the conforming~\cite{Carpenter2014,Parsani2016}
and nonconforming algorithms exhibit similar 
1) nonlinear robustness properties and 2) $L^{2}$-norm convergence rates, (i.e., nominally $p+1$ for polynomials of degree $p$).
\end{itemize}

The paper is organized as follows: Section~\ref{sec:notation} delineates the notation used herein. The nonconforming algorithm is presented in the simple context of the 
convection equation in Section~\ref{sec:discretizationconvection}. This is followed by an introduction 
to the construction of nonlinearly stable schemes by examining the discretization of the Burgers' equation (Section~\ref{sec:Burgers}). 
The nonconforming algorithm presented in Section~\ref{sec:discretizationconvection} and the mechanics presented in 
Section~\ref{sec:Burgers} are combined in Section~\ref{sec:Euler} to construct an entropy conservative/stable nonconforming 
discretization for the compressible Euler equations. Numerical experiments are are given in Section~\eqref{sec:num}, while 
conclusions are drawn in Section~\ref{sec:Conclusion}.
\section{Notation}\label{sec:notation}
PDEs are discretized on cubes having Cartesian computational coordinates denoted by 
the triple $(\xil{1},\xil{2},\xil{3})$, where the physical coordinates are denoted by the triple 
$(\xm{1},\xm{2},\xm{3})$. Vectors are represented by lowercase bold font, for example $\bm{u}$, 
while matrices are represented using sans-serif font, for example, $\mat{B}$. Continuous 
functions on a space-time domain are denoted by capital letters in script font.  For example, 
\begin{equation*}
\fnc{U}\left(\xil{1},\xil{2},\xil{3},t\right)\in L^{2}\left(\left[\alphal{1},\betal{1}\right]\times
\left[\alphal{2},\betal{2}\right]\times\left[\alphal{3},\betal{3}\right]\times\left[0,T\right]\right)
\end{equation*}
represents a square integrable function, where $t$ is the temporal coordinate. The restriction of such 
functions onto a set of mesh nodes is denoted by lower case bold font. For example, the restriction of 
$\fnc{U}$ onto a grid of $\Nl{1}\times\Nl{2}\times\Nl{3}$ nodes is given by the vector
\begin{equation*}
\bm{u} = \left[\fnc{U}\left(\bxili{}{1},t\right),\dots,\fnc{U}\left(\bxili{}{N},t\right)\right]\Tr,
\end{equation*}
where, $N$ is the total number of nodes ($N\equiv\Nl{1}\Nl{2}\Nl{3}$) square brackets ($[]$) are used 
to delineate vectors and matrices as well as ranges for variables (the context will make clear which meaning is being used). Moreover, $\bm{\xi}$ is a vector of vectors 
constructed from the three vectors $\bxil{1}$, $\bxil{2}$, and $\bxil{3}$, which are 
vectors of size $\Nl{1}$, $\Nl{2}$, and $\Nl{3}$ and contain the coordinates of the mesh in 
the three computational directions, respectively. Finally, $\bxil{}$ is constructed as 
\begin{equation*}
\bxil{}(3(i-1)+1:3i)\equiv  \bxili{}{i}
\equiv\left[\bxil{1}(i),\bxil{2}(i),\bxil{3}(i)\right]\Tr,
\end{equation*}
where the notation $\bm{u}(i)$ means the $i\Th$ entry of the vector $\bm{u}$ and $\bm{u}(i:j)$ is the subvector 
constructed from $\bm{u}$ using the $i\Th$ through $j\Th$ entries (\ie, Matlab notation is used).

 Oftentimes, monomials are discussed and the following notation is used:
\begin{equation*}
\bxil{l}^{j} \equiv \left[\left(\bxil{l}(1)\right)^{j},\dots,\left(\bxil{l}(\Nl{l})\right)^{j}\right]\Tr,
\end{equation*}
and the 
convention that $\bxil{l}^{j}=\bm{0}$ for $j<0$ is used.

Herein, one-dimensional SBP operators are used to discretize derivatives. 
The definition of a one-dimensional SBP operator in the $\xil{l}$ direction, $l=1,2,3$, 
is~\cite{DCDRF2014,Fernandez2014,Svard2014}
\begin{definition}\label{SBP}
\textbf{Summation-by-parts operator for the first derivative}: A matrix operator with constant coefficients, 
$\DxiloneD{l}\in\mathbb{R}^{\Nl{l}\times\Nl{l}}$, is a linear SBP operator of degree $p$ approximating the derivative 
$\frac{\partial}{\partial \xil{l}}$ on the domain $\xil{l}\in\left[\alphal{l},\betal{l}\right]$ with nodal 
distribution $\bxil{l}$ having $\Nl{l}$ nodes, if 
\begin{enumerate}
\item $\DxiloneD{l}\bxil{l}^{j}=j\bxil{l}^{j-1}$, $j=0,1,\dots,p$;
\item $\DxiloneD{l}\equiv\left(\PxiloneD{l}\right)^{-1}\QxiloneD{l}$, where the norm matrix, 
$\PxiloneD{l}$, is symmetric positive definite;
\item $\QxiloneD{l}\equiv\left(\SxiloneD{l}+\frac{1}{2}\ExiloneD{l}\right)$, 
$\SxiloneD{l}=-\left(\SxiloneD{l}\right)\Tr$, $\ExiloneD{l}=\left(\ExiloneD{l}\right)\Tr$,  
$\ExiloneD{l} = \diag\left(-1,0,\dots,0,1\right)=\eNl{l}\eNl{l}\Tr-\eonel{l}\eonel{l}\Tr$, 
$\eonel{l}\equiv\left[1,0,\dots,0\right]\Tr$, and $\eNl{l}\equiv\left[0,0,\dots,1\right]\Tr$. 
\end{enumerate}
Thus, a degree $p$ SBP operator is 
one that differentiates exactly monomials up to degree $p$.
\end{definition}

In this work, one-dimensional SBP operators are extended to multiple dimensions 
using tensor products ($\otimes$).  The tensor product between the matrices $\mat{A}$ and $\mat{B}$ 
is given as $\mat{A}\otimes\mat{B}$. When referencing individual entries in a matrix the notation $\mat{A}(i,j)$ 
is used, which means 
the $i\Th$ $j\Th$ entry in the matrix $\mat{A}$.
The Hadamard product ($\circ$) is also used in the construction of entropy conservative/stable discretizations. For example, 
the Hadamard product between the matrices $\mat{A}$ and $\mat{B}$ is given by $\mat{A}\circ\mat{B}$. 

The focus in this paper is exclusively on diagonal-norm SBP operators. Moreover, the same 
one-dimensional SBP operator are used in each direction, each operating on $N$ nodes. 
Specifically, diagonal-norm SBP operators constructed on the Legendre--Gauss--Lobatto (LGL) 
nodes are used, \ie, a discontinuous Galerkin collocated spectral element approach is utilized.

The physical domain $\Omega\subset\mathbb{R}^{3}$, 
with boundary $\Gamma\equiv\partial\Omega$ is partitioned into $K$ non-overlapping 
hexahedral elements. The domain of the $\kappa^{\text{th}}$ element is denoted by $\Ok$ and has 
boundary $\pOk$. Numerically, PDEs are solved in computational 
coordinates, 
where each $\Ok$ is locally transformed to $\Ohatk$, with boundary $\Ghat\equiv\pOhatk$, under the 
following assumption:

\begin{assume}\label{assume:curv}
Each element in physical space is transformed using 
a local and invertible curvilinear coordinate transformation that is compatible at 
shared interfaces, meaning that points in computational space on either side of a 
shared interface  mapped to the same physical location and therefore map back 
to the analogous location in computational space; this is the standard assumption 
that the curvilinear coordinate transformation is water tight.
\end{assume}
\section{A $p$-nonconforming algorithm: Linear convection}\label{sec:discretizationconvection}

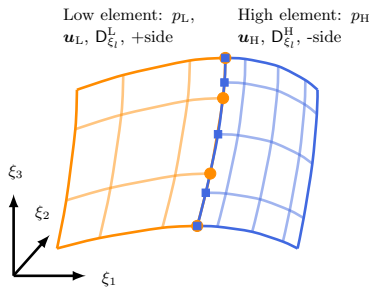
\begin{figure}[!t]
\centering
\begin{tikzpicture}[scale =0.7]
\begin{axis}
[
axis line style={draw=none},
tick style={draw=none},
ticks=none,
ymin=0.25,ymax=0.8,xmin=0.25,xmax=1.0
]
\addplot[mark=none,line width = 0.5mm, color = darkorange, smooth] table {X11.txt};
\addplot[mark=none,line width = 0.5mm, color = darkorange, draw opacity = 0.5, smooth] table {X12.txt};
\addplot[mark=none,line width = 0.5mm, color = darkorange, draw opacity = 0.5, smooth] table {X13.txt};
\addplot[mark=none,line width = 0.5mm, color = darkorange, smooth] table {X14.txt};

\addplot[mark=none,line width = 0.5mm, color = darkorange, smooth] table {Y11.txt};
\addplot[mark=none,line width = 0.5mm, color = darkorange, draw opacity = 0.5, smooth] table {Y12.txt};
\addplot[mark=none,line width = 0.5mm, color = darkorange, draw opacity = 0.5, smooth] table {Y13.txt};
\addplot[mark=none,line width = 0.5mm, color = darkorange, smooth] table {Y14.txt};

\addplot[mark=*,color = darkorange,smooth, mark options={solid}, only marks,mark size=3pt] table {Y14.txt};

\addplot[mark=none,line width = 0.5mm, color = royalblue, smooth] table {X21.txt};
\addplot[mark=none,line width = 0.5mm, color = royalblue, draw opacity = 0.5, smooth] table {X22.txt};
\addplot[mark=none,line width = 0.5mm, color = royalblue, draw opacity = 0.5, smooth] table {X23.txt};
\addplot[mark=none,line width = 0.5mm, color = royalblue, draw opacity = 0.5, smooth] table {X24.txt};
\addplot[mark=none,line width = 0.5mm, color = royalblue, smooth] table {X25.txt};

\addplot[mark=none,line width = 0.5mm, color = royalblue, smooth] table {Y21.txt};
\addplot[mark=none,line width = 0.5mm, color = royalblue, draw opacity = 0.5, smooth] table {Y22.txt};
\addplot[mark=none,line width = 0.5mm, color = royalblue, draw opacity = 0.5, smooth] table {Y23.txt};
\addplot[mark=none,line width = 0.5mm, color = royalblue, draw opacity = 0.5, smooth] table {Y24.txt};
\addplot[mark=none,line width = 0.5mm, color = royalblue, smooth] table {Y25.txt};

\addplot[mark=square*,color = royalblue, mark options={solid}, smooth, only marks] table {Y21.txt};

\node (elemL) at (250,505) [text width=2.5cm]{\small Low element: $\pL$, $\uL$, $\Dxil{l}^{\rmL}$, +side};
\node (elemH) at (640,505) [text width=3cm]{\small High element: $\pH$, $\uH$, $\Dxil{l}^{\rmH}$, -side};
          \draw[-latex, line width = 0.5mm,] (10,50,0) -- (10,200,0);
          \draw[-latex, line width = 0.5mm,] (10,50,0) -- (160,50,0);
          \draw[-latex, line width = 0.5mm,] (10,50.5,0) -- (85,125,-2);
          \draw (180,45,0) node[anchor=west] {\small$\xil{1}$};
          \draw (15,215,0) node[anchor=south] {\small$\xil{3}$};
          \draw (100,140,-2) node[anchor=south east] {\small$\xil{2}$};
%
%
%
%
\end{axis}
\end{tikzpicture}
\caption{Nonconforming elements.}
\label{fig:non}
\end{figure}
The focus herein is on nonconformity arising from curvilinearly mapped elements 
that have conforming interfaces but nonconforming nodal distributions, as would result 
from $p$-refinement interfaces (see Fig.~\ref{fig:non}). The construction of entropy conservative/stable 
discretizations for the compressible Euler equations on Cartesian grids is detailed in Friedrich \etal~\cite{Friedrich2018}. 
Extending the Cartesian work involves developing a $p$-refinement, curvilinear, interface coupling technique
that maintains 1) accuracy, 2) discrete entropy conservation/stability, and 3) elementwise conservation.

\subsection{Scalar convection: continuous and semi-discrete analysis}\label{sec:ICASE_Eqn}
Many of the salient difficulties encountered while constructing stable and conservative nonconforming discretizations 
of the Euler equations, are also present in discretization of linear convection.  
Consider the constant coefficient linear convection in Cartesian coordinates.  
\begin{equation}\label{eq:cartconvection}
  \frac{\partial\U}{\partial t}+\sum\limits_{m=1}^{3}\frac{\partial\left(a_{m}\U\right)}{\partial\xm{m}}=0,
\end{equation}
where $(a_{m}\U)$ are the fluxes and $a_{m}$ are the (constant) components of the convection speed. The stability of~\eqref{eq:cartconvection} can be determined 
using the energy method which proceeds by multiplying~\eqref{eq:cartconvection} by the solution ($\U$), 
and after using the chain rule results in 
\begin{equation}\label{eq:cartconvectionenergy1}
  \frac{1}{2}\frac{\partial\U^{2}}{\partial t}+
  \frac{1}{2}\sum\limits_{m=1}^{3}\frac{\partial\left(a_{m}\U^{2}\right)}{\partial\xm{m}}=0.
\end{equation}
Integrating over the domain, $\Omega$, using integration by parts, and Leibniz rule gives
\begin{equation}\label{eq:cartconvectionenergy2}
  \frac{\mr{d}}{\mr{d}t}\int_{\Omega}\frac{\U^{2}}{2}\mr{d}\Omega+
  \frac{1}{2}\sum\limits_{m=1}^{3}\oint_{\Gamma}\left(a_{m}\U^{2}\right)\nxm{m}\mr{d}\Gamma=0,
\end{equation}
where $\nxm{m}$ is the $m\Th$ component of the outward facing unit normal. What \Eq~\eqref{eq:cartconvectionenergy2} 
demonstrates is that the time rate of change of the norm of the solution, 
$\|\U\|^{2}\equiv\int_{\Omega}\U^{2}\mr{d}\Omega$, depends solely on surface flux integrals. 
Imposing appropriate boundary conditions in \Eq~\eqref{eq:cartconvectionenergy2} leads to an energy estimate 
on the solution and therefore a proof of stability. The discretizations that are developed in this paper 
mimic the above energy analysis in a one-to-one fashion and similarly lead to stability statements on the 
semi-discrete equations.  

The proposed algorithm is constructed from differentiation matrices that 
live in computational space and for this purpose, \Eq~\eqref{eq:cartconvection} is transformed using the 
curvilinear coordinate transformation $\xm{m}=\xm{m}\left(\xil{1},\xil{2},\xil{3}\right)$. Thus, 
after expanding the derivatives as 
\begin{equation*}
  \frac{\partial}{\partial\xm{m}}=
  \sum\limits_{l=1}^{3}\frac{\partial\xil{l}}{\partial\xm{m}}\frac{\partial}{\partial\xil{l}},
\end{equation*} 
and multiplying by the metric Jacobian ($\Jk$), \eqref{eq:cartconvection} becomes 
 \begin{equation}\label{eq:convectionchain}
  \Jk\frac{\partial\U}{\partial t}+\sum\limits_{l,m=1}^{3}\Jk\frac{\partial\xil{l}}{\partial\xm{m}}
  \frac{\partial \left(a_{m}\U\right)}{\partial\xil{l}}=0.
\end{equation}

Herein, we refer to Eq.~\eqref{eq:convectionchain} as the chain rule form of Eq.~\eqref{eq:cartconvection} .
Bringing the metric terms, $\Jdxildxm{l}{m}$, inside the derivative and using the product rule gives
 \begin{equation}\label{eq:convectionstrong1}
  \Jk\frac{\partial\U}{\partial t}+\sum\limits_{l,m=1}^{3}
  \frac{\partial}{\partial\xil{l}}\left(\Jdxildxm{l}{m}a_{m}\U\right)
-\sum\limits_{l,m=1}^{3}a_{m}\U\frac{\partial}{\partial\xil{l}}\left(\Jdxildxm{l}{m}\right)
 =0.
\end{equation}
The last term on the left-hand side of~\eqref{eq:convectionstrong1} is zero via the GCL relations
\begin{equation}\label{eq:GCL}
    \sum\limits_{l=1}^{3}\frac{\partial}{\partial\xil{l}}\left(\Jdxildxm{l}{m}\right)=0,\quad m=1,2,3,
\end{equation}
leading to the strong conservation form of the convection equation in curvilinear coordinates:
 \begin{equation}\label{eq:convectionstrong}
  \Jk\frac{\partial\U}{\partial t}+\sum\limits_{l,m=1}^{3}
  \frac{\partial}{\partial\xil{l}}\left(\Jdxildxm{l}{m}a_{m}\U\right)=0.
\end{equation}

Now, consider discretizing~\eqref{eq:convectionstrong} by using 
the following differentiation matrices:
\begin{equation*}
\Dxil{1}\equiv\mat{D}^{(1D)}\otimes\Imat{N}\otimes\Imat{N},\quad
\Dxil{2}\equiv\Imat{N}\otimes\mat{D}^{(1D)}\otimes\Imat{N},\quad
\Dxil{3}\equiv\Imat{N}\otimes\Imat{N}\otimes\mat{D}^{(N)},
\end{equation*} 
where $\Imat{N}$ is an $N\times N$ identity matrix. The diagonal 
matrices containing the metric Jacobian and metric terms along their diagonals, respectively, are defined as follows:
\begin{equation*}
  \begin{split}
  &\matJk{\kappa}\equiv\diag\left(\Jk(\bmxi{1}),\dots,\Jk(\bmxi{\Nl{\kappa}})\right),\\
  &\matAlmk{l}{m}{\kappa}\equiv\diag\left(\Jdxildxm{l}{m}(\bmxi{1}),\dots,
  \Jdxildxm{l}{m}(\bmxi{\Nl{\kappa}})\right).
  \end{split}
\end{equation*}
where $\Nl{\kappa}\equiv N^{3}$ is the total number of nodes in element $\kappa$.
With this nomenclature, the discretization of~\eqref{eq:convectionstrong} on the $\kappa\Th$ element reads
 \begin{equation}\label{eq:convectionstrongdisc}
   \matJk{\kappa}\frac{\mr{d}\uk}{\mr{d}t}+\sum\limits_{l,m=1}^{3}\Dxil{l}\matAlmk{l}{m}{\kappa}a_{m}\uk=\bm{SAT},
\end{equation}
where $\bm{SAT}$ is the vector of the SATs used to impose both boundary conditions and or inter-element connectivity. 
Unfortunately, the scheme~\eqref{eq:convectionstrongdisc} is not guaranteed to be stable. However, a well-known remedy is to consider a canonical splitting which 
is arrived at by summing one half of~\eqref{eq:convectionchain} 
and one half of~\eqref{eq:convectionstrong1}, giving 
\begin{equation}\label{eq:convectionsplit}
  \begin{split}
  &\Jk\frac{\partial\U}{\partial t}+\frac{1}{2}\sum\limits_{l,m=1}^{3}\left\{
    \frac{\partial}{\partial\xil{l}}\left(\Jdxildxm{l}{m}a_{m}\U\right)+
     \Jdxildxm{l}{m}\frac{\partial}{\partial\xil{l}}\left(a_{m}\U\right)
    \right\}\\&-\frac{1}{2}\sum\limits_{l,m=1}^{3}\left\{
    a_{m}\U\frac{\partial}{\partial\xil{l}}\left(\Jdxildxm{l}{m}\right)\right\}=0,
  \end{split}
\end{equation}
where the last set of terms are zero by the GCL conditions~\eqref{eq:GCL}. Then, a 
stable semi-discrete form is constructed in the same manner as the split form~\eqref{eq:convectionsplit} by 
discretizing~\eqref{eq:convectionchain} and~\eqref{eq:convectionstrong1} using $\Dxil{l}$, $\matJk{\kappa}$, and  
$\matAlmk{l}{m}{\kappa}$, and averaging the results. This procedure yields
\begin{equation}\label{eq:convectionsplitdisc}
  \begin{split}
  &\matJk{\kappa}\frac{\mr{d}\uk}{\mr{d} t}+\frac{1}{2}\sum\limits_{l,m=1}^{3}
  a_{m}\left\{\Dxil{l}\matAlmk{l}{m}{\kappa}+\matAlmk{l}{m}{\kappa}\Dxil{l}\right\}\uk
  \\&-\frac{1}{2}\sum\limits_{l,m=1}^{3}\left\{
    a_{m}\diag\left(\uk\right)\Dxil{l}\matAlmk{l}{m}{\kappa}\ones{\kappa}\right\}=\bm{0},
  \end{split}
\end{equation}
where $\ones{\kappa}$ is a vector of ones of the size of the number of nodes on 
the $\kappa\Th$ element. As in the continuous case, the semi-discrete form has a set of discrete GCL conditions
\begin{equation}\label{eq:discGCLconvection}
\sum\limits_{l=1}^{3}
    \Dxil{l}\matAlmk{l}{m}{\kappa}\ones{\kappa}=\bm{0}, \quad m = 1,2,3,
\end{equation}
that if satisfied, lead to the following telescoping (provably stable) semi-discrete form
\begin{equation}\label{eq:convectionsplitdisctele}
  \matJk{\kappa}\frac{\mr{d}\uk}{\mr{d} t}+\frac{1}{2}\sum\limits_{l,m=1}^{3}
  a_{m}\left\{\Dxil{l}\matAlmk{l}{m}{\kappa}+\matAlmk{l}{m}{\kappa}\Dxil{l}\right\}\uk =0.
\end{equation}
\begin{remark}
  The linear stability of semi-discrete operators for constant coefficient hyperbolic systems, 
  is not preserved by arbitrary design order approximations to the metric terms.  Only approximations to the metric terms that satisfy
   discrete GCL~\eqref{eq:discGCLconvection} conditions lead to stable semi-discrete forms, a point that can
  not be overemphasized in the context of this work. 
\end{remark}
\begin{remark}
  The satisfaction of the discrete GCL conditions~\eqref{eq:discGCLconvection} for tensor-product 
  conforming discretizations has well-known solutions (Vinokur and Yee~\cite{Vinokur2002a} or Thomas and Lombard~\cite{Thomas1979}) 
  which require that the differentiation matrices commute, \ie $\Dxil{1}\Dxil{2}=\Dxil{2}\Dxil{1}$, \etc.
\end{remark}
\begin{remark}
  The discrete GCL conditions can alternatively be derived by inputting a constant solution into~\eqref{eq:convectionsplitdisc}.
\end{remark}

\subsection{Scalar convection and the nonconforming interface}\label{sec:ICASE_Eqn_NonConform}
The analysis and presentation of nonconforming semi-discrete algorithms is simplified by considering a single interface
between two adjoining elements as shown in Figure~\ref{fig:non}.  The elements have
aligned computational coordinates and the shared interface is vertical. The nonconformity is assumed to arise from 
local approximations with differing polynomial degrees. Specifically, the left element has 
polynomial degree $\pL$ (low: subscript/superscript $\rm{L}$) and the right element has polynomial degree $\pH$ 
(high: subscript/superscript $\rm{H}$) where $\pH>\pL$ (see Figure~\ref{fig:non}). 
The first task is to construct SBP operators that span both elements: i.e., macro element operators
composed of elements $\rmL$ and $\rmH$. A naive construction would be the following operators in the three coordinate
directions:
\begin{equation}\label{eq:naive}
    \DtildeH{1}\equiv\left[
  \begin{array}{cc}
    \Dxil{1}^{\rmL}&\\
    &\Dxil{1}^{\rmH}
  \end{array}  
  \right] 
  \quad ; \quad
    \Dtildel{2}\equiv\left[
   \begin{array}{cc}
     \Dxil{2}^{\rmL}\\
     &\Dxil{2}^{\rmH} 
   \end{array}
   \right]
  \quad ; \quad
    \Dtildel{3}\equiv\left[
   \begin{array}{cc}
     \Dxil{3}^{\rmL}\\
     &\Dxil{3}^{\rmH} 
   \end{array}
   \right]
\end{equation}
The $\Dtildel{2}$ and $\Dtildel{3}$ macro element operators are by construction SBP operators; the action of their 
differentiation is parallel to the interface, and they telescope out to the boundaries.  
The $\DtildeH{1}$ is not by construction an SBP operator, despite 
the individual matrices composing $\DtildeH{1}$ being SBP operators.
In addition, and more critically, there is no coupling between the two elements. 
Thus, interface coupling must be introduced between the two elements that render the operators on the macro element as SBP operators.
For this purpose, interpolation operators are needed 
that interpolate information from element $\rmH$ to element $\rmL$ and vice versa.  For simplicity, 
the interpolation operators use only tensor product surface information from the adjoining interface surface.

With this background, general matrix difference operators between the two elements are constructed as 
\begin{equation}
  \Dtildel{l}=\Mtilde^{-1}\Qtildel{l}=\Mtilde^{-1}\left(\Stildel{l}+\frac{1}{2}\Etildel{l}\right).
\end{equation}
Focusing on the direction orthogonal to the interface ($\xi_1$) the relevant matrices are given by 
\begin{equation}\label{eq:marcoD} 
  \begin{split}
  &\Mtilde\equiv\diag\left[
  \begin{array}{cc}  
  \M^{\rmL}\\
  &\M^{\rmH}
  \end{array}
  \right],\\
  &\Stildel{1}\equiv\\
  &
  \left[
  \begin{array}{cc}
    \Sxil{1}^{\rmL}&\frac{1}{2}\left(\eNl{\rmL}\eonel{\rmH}\Tr\otimes\PoneD_{\rmL}\IHtoLoneD\otimes\PoneD_{\rmL}\IHtoLoneD\right)\\
    -\frac{1}{2}\left(\eonel{\rmH}\eNl{\rmL}\Tr\otimes\PoneD_{\rmH}\ILtoHoneD\otimes\PoneD_{\rmH}\ILtoHoneD\right)&\Sxil{1}^{\rmH}
  \end{array}
  \right],\\
  &\Etildel{1}\equiv\\
  &
  \left[
  \begin{array}{cc}
    -\eonel{\rmL}\eonel{\rmL}\Tr\otimes\PoneD_{\rmL}\otimes\PoneD_{\rmL}\\
    &\eNl{\rmH}\eNl{\rmH}\Tr\otimes\PoneD_{\rmH}\otimes\PoneD_{\rmH}
  \end{array}
  \right],
\end{split}
\end{equation}
and $\IHtoLoneD$ and $\ILtoHoneD$ are one-dimensional interpolation operators from the $\rmH$ element
to the $\rmL$ element and vice versa. 

A necessary constraint the SBP formalism places on $\Dtildel{1}$, is skew-symmetry of the 
$\Stildel{1}$ matrices. The block-diagonal matrices in $\Stildel{1}$ are 
already skew-symmetric but the off diagonal blocks are not. 
Thus, it is necessary to satisfy the following
condition:
\begin{equation*}
\left(\eNl{\rmL}\eonel{\rmH}\Tr\otimes\PoneD_{\rmL}\IHtoLoneD\otimes\PoneD_{\rmL}\IHtoLoneD\right) = \left\{\left(\eonel{\rmH}\eNl{\rmL}\Tr\otimes\PoneD_{\rmH}\ILtoHoneD\otimes\PoneD_{\rmH}\ILtoHoneD\right)\right\}\Tr.
\end{equation*}
This implies that the interpolation operators are related to each other as follows:
\begin{equation*}
  \IHtoLoneD=\left(\PoneD_{\rmL}\right)^{-1}\left(\ILtoHoneD\right)\Tr\PoneD_{\rmH}.
\end{equation*}
This property is denoted as the SBP preserving property because it leads to 
a macro element differentiation matrix that is an SBP operator. The optimal interpolation operator, 
$\ILtoHoneD$, is constructed to exactly interpolate polynomial of degree $\pL$ 
and can be easily constructed as follows:
\begin{equation*}
\ILtoHoneD=\left[\bm{\xi}_{\rmH}^{0},\dots,\bm{\xi}_{\rmH}^{\pL}\right]\left[\bm{\xi}_{\rmL}^{0},\dots,\bm{\xi}_{\rmL}^{\pL}\right]^{-1},
\end{equation*}
where $\bm{\xi}_{\rmL}$ and $\bm{\xi}_{\rmH}$ are the one-dimensional nodal distributions in computational space of the two elements. 
The companion interpolation 
operator $\IHtoLoneD$ is sub-optimal by one degree ($\pL-1$), a consequence of satisfying the necessary 
SBP-preserving property (see Friedrich~\etal~\cite{Friedrich2018} for a complete discussion). 

The semi-discrete skew-symmetric split operator given in \Eq~\eqref{eq:convectionsplitdisc}, 
discretized using the macro element operators $\Dtildel{l}$, and metric information $\Jtilde$, $\matAlmk{l}{m}{}$,
leads to the following:
\begin{equation}\label{eq:macro}
  \begin{split}
  &\Jtilde\frac{\mr{d}\utilde}{\mr{d}t}+
  \frac{1}{2}\sum\limits_{l,m=1}^{3}a_{m}\left(\Dtildel{l}\tildebarmatAlm{l}{m}+\tildebarmatAlm{l}{m}\Dtildel{l}\right)\utilde\\
  &-\frac{1}{2}\sum\limits_{l,m=1}^{3}a_{m}\diag\left(\utilde\right)\Dtildel{l}\tildebarmatAlm{l}{m}\tildeone
  =\bm{0},
  \end{split}
\end{equation} 
where
\begin{equation}\label{eq:marcmetrics} 
  \begin{split}
  &\utilde\equiv\left[\uL\Tr,\uH\Tr\right]\Tr,
  \Jtilde\equiv\diag\left[
  \begin{array}{cc}  
  \matJk{\rmL}\\
  &\matJk{\rmH}
  \end{array}
  \right],
  \matAlmk{l}{m}{}\equiv
  \left[
  \begin{array}{cc}
  \matAlmk{l}{m}{\rmL}\\
  &\matAlmk{l}{m}{\rmH}
\end{array}  
  \right].
\end{split}
\end{equation}
As was the case in \Eq~\eqref{eq:convectionsplitdisc}, a necessary condition for stability is that the metric
terms satisfy the following discrete GCL conditions:
\begin{equation}\label{eq:discGCLmacro}
  \sum\limits_{l=1}^{3}\Dtildel{l}\tildebarmatAlm{l}{m}\tildeone=\bm{0}.
\end{equation}
Recognizing that $\Dtildel{1}$ is not a tensor product operator, discrete metrics constructed using 
the analytic formalism of Vinokur and Yee~\cite{Vinokur2002a} or Thomas and Lombard~\cite{Thomas1979} 
will not in general satisfy the discrete GCL condition required in \Eq~\eqref{eq:discGCLmacro}.  
The only viable alternative is to solve for discrete metrics that directly satisfy the GCL constraints.

\begin{remark}
Note that metric terms are assigned colors; e.g.,  
the time-term Jacobian: $\Jtilde$ or the volume metric terms: $\tildebarmatAlm{l}{m}$.  
Metric terms with common colors form a clique and must be formed consistently.  
For example, the time-term Jacobian and the volume metric Jacobian need not be equivalent.  Another
important clique: the surface metrics, will be introduced in the next subsection.
\end{remark}

\subsection{Isolating the metrics}\label{sec:mortar:metrics1}
The system~\eqref{eq:discGCLmacro} is a highly under-determined system which 
couples the approximation of the metrics in both elements.
Worse still, whole clouds of nonconforming elements would in general need to be solved simultaneously, 
making the approach undesirable and even impracticable!
A close examination of the volume terms provides insight on how to overcome this issue:
\begin{equation}\label{eq:combined}
  \begin{split}
&\Mtilde\left(\Dtildel{1}\tildebarmatAlm{1}{m}+\tildebarmatAlm{1}{m}\Dtildel{1}\right)=\\
 &\left[
  \begin{array}{cc}
   \mat{A}_{11}&
\mat{A}_{12}\\
-\mat{A}_{12}\Tr
      &\mat{A}_{22}
  \end{array}
  \right]
  +\left(\Etildel{1}\tildebarmatAlm{1}{m}+\tildebarmatAlm{1}{m}\Etildel{1}\right),\\\\
  &\mat{A}_{11}\equiv \left\{\Sxil{1}^{\rmL}\matAlmk{1}{m}{\rmL}+\matAlmk{1}{m}{\rmL}\Sxil{1}^{\rmL}\right\},\\
  &\mat{A}_{12}=  \frac{1}{2}\left\{\begin{array}{l}
    \colorbox{yellow}{\matAlmk{1}{m}{\rmL}}\left(\eNl{\rmL}\eonel{\rmH}\Tr\otimes\PoneD_{\rmL}\IHtoLoneD\otimes\PoneD_{\rmL}\IHtoLoneD\right)\\
    +\left(\eNl{\rmL}\eonel{\rmH}\Tr\otimes\PoneD_{\rmL}\IHtoLoneD\otimes\PoneD_{\rmL}\IHtoLoneD\right)\colorbox{red}{\matAlmk{1}{m}{\rmH}}\end{array}\right\},\\
    &\mat{A}_{22}\equiv\left\{\Sxil{1}^{\rmH}\matAlmk{1}{m}{\rmH}+\matAlmk{1}{m}{\rmH}\Sxil{1}^{\rmH}\right\}.
  \end{split}
\end{equation}
Replacing the highlighted off-diagonal metric terms in \eqref{eq:combined} with known metrics data,
by construction preserves the skew-symmetry of the operator $\Stildel{l}$, but decouples 
the discrete GCL conditions \eqref{eq:discGCLmacro}.  The highlighted off-diagonal surface metric terms:
$\rmL$ and $\rmH$ become forcing data for the GCL condition.  Note that $\rmL$ and $\rmH$ 
need not be equivalent, provided they are design order close.
\begin{remark}
  More generally, $\mat{A}_{12}$ can be composed of any terms that are design order interpolations from one element to the other and these can be constructed in 
  a very general way, for example, one might consider first interpolating data to an intermediate set of Gauss nodes on which the metric terms 
  are specified. The theoretical results in this paper apply to any such approach that is design order, satisfied the structural requirements presented above, 
  and that uses specified metric information in the coupling terms.
\end{remark}
With this approach, the discrete GCL conditions become (where contributions from the boundary SATs have been ignored)
\begin{equation}\label{eq:discGCLL}
  \begin{split}
  &2\M^{\rmL}\sum\limits_{l=1}^{3}\Dxil{l}^{\rmL}\matAlmk{l}{m}{\rmL}\ones{\rmL}=\\
  &\left\{
  \left(\eNl{\rmL}\eNl{\rmL}\Tr\otimes\PoneD_{\rmL}\otimes\PoneD_{\rmL}\right)\matAlmk{1}{m}{\rmL}
  +\matAlmk{1}{m}{\rmL}\left(\eNl{\rmL}\eNl{\rmL}\Tr\otimes\PoneD_{\rmL}\otimes\PoneD_{\rmL}\right)\right\}\ones{\rmL}\\
  &-\left\{
\left(\eNl{\rmL}\otimes\Imat{\Nl{\rmL}}\otimes\Imat{\Nl{\rmL}}\right)
  \matAlmkLH{1}{m}{\rmL}{\Ghat}{red}\left(\PoneD_{\rmL}\IHtoLoneD\otimes\PoneD_{\rmL}\IHtoLoneD\right)
  \left(\eonel{\rmH}\Tr\otimes\Imat{\Nl{\rmH}}\otimes\Imat{\Nl{\rmH}}\right)\right\}\ones{\rmH}\\
  &-\left\{
\left(\eNl{\rmL}\otimes\Imat{\Nl{\rmL}}\otimes\Imat{\Nl{\rmL}}\right)
  \left(\PoneD_{\rmL}\IHtoLoneD\otimes\PoneD_{\rmL}\IHtoLoneD\right)\matAlmkLH{1}{m}{\rmH}{\Ghat}{red}
  \left(\eonel{\rmH}\Tr\otimes\Imat{\Nl{\rmH}}\otimes\Imat{\Nl{\rmH}}\right)\right\}\ones{\rmH},
  \end{split}
\end{equation}
\begin{equation}\label{eq:discGCLH}
  \begin{split}
  &2\M^{\rmL}\sum\limits_{l=1}^{3}\Dxil{l}^{\rmH}\matAlmk{l}{m}{\rmH}\ones{\rmH}=\\
  &-\left\{
  \left(\eonel{\rmH}\eonel{\rmH}\Tr\otimes\PoneD_{\rmH}\otimes\PoneD_{\rmH}\right)\matAlmk{1}{m}{\rmH}
  +\matAlmk{1}{m}{\rmH}\left(\eonel{\rmL}\eonel{\rmH}\Tr\otimes\PoneD_{\rmH}\otimes\PoneD_{\rmH}\right)\right\}\ones{\rmH}\\
  &+\left\{
\left(\eonel{\rmH}\otimes\Imat{\Nl{\rmH}}\otimes\Imat{\Nl{\rmL}}\right)
  \matAlmkLH{1}{m}{\rmH}{\Ghat}{red}\left(\PoneD_{\rmH}\ILtoHoneD\otimes\PoneD_{\rmH}\ILtoHoneD\right)
  \left(\eNl{\rmL}\Tr\otimes\Imat{\Nl{\rmL}}\otimes\Imat{\Nl{\rmL}}\right)\right\}\ones{\rmL}\\
  &+\left\{
\left(\eonel{\rmH}\otimes\Imat{\Nl{\rmH}}\otimes\Imat{\Nl{\rmH}}\right)
  \left(\PoneD_{\rmH}\ILtoHoneD\otimes\PoneD_{\rmH}\ILtoHoneD\right)\matAlmkLH{1}{m}{\rmL}{\Ghat}{red}
  \left(\eNl{\rmL}\Tr\otimes\Imat{\Nl{\rmL}}\otimes\Imat{\Nl{\rmL}}\right)\right\}\ones{\rmL}.
  \end{split}
\end{equation}
The matrix $\matAlmkLH{1}{m}{\rmL}{\Ghat}{red}$ is of size $\Nl{\rmL}^{2}\times\Nl{\rmL}^{2}$ ($\Nl{\rmL}$ is the number of nodes in each computational direction on element $\rmL$) and its 
diagonal elements are 
approximations to the 
metrics on the surface nodes of element $\rmL$ at the shared interface. An analogous definition holds
for $\matAlmkLH{1}{m}{\rmH}{\Ghat}{red}$. In order to decouple the two systems of equations in~\eqref{eq:discGCLL} and 
\eqref{eq:discGCLH}
the terms in $\matAlmkLH{1}{m}{\rmL}{\Ghat}{red}$ and $\matAlmkLH{1}{m}{\rmH}{\Ghat}{red}$ need to be specified, for example, 
using the analytic metrics, which is the approach used in this paper. For later use, we introduce notation for the macro element $\Dtildelm{l}{m}$ which is 
the macro element operator constructed as descried above for the metric terms $\Jdxildxm{l}{m}$.

The next section details how to construct the metrics so that the remaining discrete GCL conditions are satisfied.
\subsection{Metric solution mechanics}\label{sec:mortar:metrics2}
In this section, to demonstrate the proposed approach for approximating the metric terms, the discrete GCL conditions associated with the 
element $\rmL$~\eqref{eq:discGCLL} are used.

Note that for an arbitrary interior element, the discrete GCL system that needs to be solved includes 
the coupling terms on all six faces, while for the simplified problem considered above only the 
(nonconforming) coupling on the vertical interface appears.

\Eq~\eqref{eq:discGCLL} can be algebraically manipulated into a form that is more convenient for constructing a 
solution procedure for the metric terms. 
This form  is derived by multiplying \Eq~\eqref{eq:discGCLL} by $-1$, using the SBP property 
$\QoneDL = -\left(\QoneDL\right)\Tr+\EoneDL$, and canceling common terms. 
Doing so results in 

\begin{equation}\label{eq:couplingsecondtwo}
  \begin{split}
&2\sum\limits_{l=1}^{3}
\left(\Qxil{l}^{\rmL}\right)\Tr\matAlmk{l}{m}{\rmL}\ones{\rmL}=\\
&-\left\{
\left(\eNl{\rmL}\otimes\Imat{\Nl{\rmL}}\otimes\Imat{\Nl{\rmL}}\right)
  \matAlmkLH{1}{m}{\rmL}{\Ghat}{red}\left(\PoneD_{\rmL}\IHtoLoneD\otimes\PoneD_{\rmL}\IHtoLoneD\right)
  \left(\eonel{\rmH}\Tr\otimes\Imat{\Nl{\rmH}}\otimes\Imat{\Nl{\rmH}}\right)\right\}\ones{\rmH}\\
  &-\left\{
\left(\eNl{\rmL}\otimes\Imat{\Nl{\rmL}}\otimes\Imat{\Nl{\rmL}}\right)
  \left(\PoneD_{\rmL}\IHtoLoneD\otimes\PoneD_{\rmL}\IHtoLoneD\right)\matAlmkLH{1}{m}{\rmH}{\Ghat}{red}
  \left(\eonel{\rmH}\Tr\otimes\Imat{\Nl{\rmH}}\otimes\Imat{\Nl{\rmH}}\right)\right\}\ones{\rmH},\\
 &m = 1,2,3,
  \end{split}
\end{equation}
where 
$\Qxil{1}^{\rmL}\equiv\QoneDL\otimes\PoneD_{\rmL}\otimes\PoneD_{\rmL}$, 
$\Qxil{2}^{\rmL}\equiv\PoneD_{\rmL}\otimes\QoneDL\otimes\PoneD_{\rmL}$, $\Qxil{3}^{\rmL}\equiv\PoneD_{\rmL}\otimes\PoneD_{\rmL}\otimes\QoneDL$. 
Note that the contributions from the $\Exil{l}$ from the left-hand side 
(\ie, coming from the step $\mat{Q}=-\mat{Q}\Tr+\mat{E}$) related to the boundaries of the macro element 
are ignored. This contributions interact with the boundary SATs in the same way as the interface does.

The metric terms in \Eq~\eqref{eq:couplingsecondtwo} are determined by solving a 
strictly convex quadratic optimization problem, following the algorithm given in Crean~\etal~\cite{Crean2018}:

\begin{equation}\label{eq:opt}
  \begin{split}
&\min_{\bm{a}_{m}}\frac{1}{2}\left(\bm{a}_{m}^{\rmL}
-\bm{a}_{m,\text{target}}^{\rmL}\right)\Tr\left(\bm{a}_{m}^{\rmL}-\bm{a}_{m,\text{target}}^{\rmL}\right),\quad
\text{subject to}\quad\mat{M}^{\rmL}\bm{a}_{m}^{\rmL}=\bm{c}_{m}^{\rmL},\\
 &m = 1,2,3,
  \end{split}
\end{equation}
where $\bm{a}_{m}^{\rmL}$ and $\bm{a}_{m,\text{target}}^{\rmL}$  
are the optimized and targeted metric terms, respectively, and

\begin{equation*}
\begin{split}
&\left(\bm{a}_{m}^{\rmL}\right)\Tr\equiv\ones{\rmL}\Tr
\left[
\matAlmk{1}{m}{\rmL},
\matAlmk{2}{m}{\rmL},
\matAlmk{3}{m}{\rmL}
\right],\; \mat{M}^{\rmL}\equiv\left[\left(\Qxil{1}^{\rmL}\right)\Tr,\left(\Qxil{2}^{\rmL}\right)\Tr,\left(\Qxil{3}^{\rmL}\right)\Tr\right],\\
&2\bm{c}_{m}^{\rmL}\equiv\\&-\left\{
\left(\eNl{\rmL}\otimes\Imat{\Nl{\rmL}}\otimes\Imat{\Nl{\rmL}}\right)
  \matAlmkLH{1}{m}{\rmL}{\Ghat}{red}\left(\PoneD_{\rmL}\IHtoLoneD\otimes\PoneD_{\rmL}\IHtoLoneD\right)
  \left(\eonel{\rmH}\Tr\otimes\Imat{\Nl{\rmH}}\otimes\Imat{\Nl{\rmH}}\right)\right\}\ones{\rmH}\\
  &-\left\{
\left(\eNl{\rmL}\otimes\Imat{\Nl{\rmL}}\otimes\Imat{\Nl{\rmL}}\right)
  \left(\PoneD_{\rmL}\IHtoLoneD\otimes\PoneD_{\rmL}\IHtoLoneD\right)\matAlmkLH{1}{m}{\rmH}{\Ghat}{red}
  \left(\eonel{\rmH}\Tr\otimes\Imat{\Nl{\rmH}}\otimes\Imat{\Nl{\rmH}}\right)\right\}\ones{\rmH},
\end{split}
\end{equation*}
with $\bm{a}_{m}^{\rmL}$ of size $3 (\Nl{\rmL})^{3}\times 1$, 
$\MM^{\rmL}$ of size $(\Nl{\rmL})^{3}\times 3 (\Nl{\rmL})^{3}$ and 
$\bm{c}_{m}^{\rmL}$ of size $(\Nl{\rmL})^{3}\times 1$.
The optimal solution, in the Cartesian $2$-norm, is given by (see Proposition $1$ in Crean~\etal~\cite{Crean2018})

\begin{equation}\label{eq:opt1}
\bm{a}_{m}^{\rmL}=\bm{a}_{m,\text{target}}^{\rmL}-\left(\MM^{\rmL}\right)^{\dagger}\left(\MM^{\rmL} 
\bm{a}_{m,\text{target}}^{\rmL}-\bm{c}_{m}^{\rmL}\right),
\end{equation}
with $\left(\MM^{\rmL}\right)^{\dagger}$ the Moore-Penrose pseudo-inverse of $\MM^{\rmL}$. 
The pseudo-inverse is computed using a singular value decomposition (SVD) of $\MM^{\rmL}$

\begin{equation*}
\MM^{\rmL}=\mat{U}^{\rmL}\mat{\Sigma}^{\rmL}\left(\mat{V}^{\rmL}\right)\Tr,\quad\left(\MM^{\rmL}\right)^{\dagger}=
\mat{V}^{\rmL}\left(\mat{\Sigma}^{\rmL}\right)^{\dagger}\left(\mat{U}^{\rmL}\right)\Tr,
\end{equation*}
with $\mat{U}^{\rmL}$ an $(\Nl{\rmL})^{3}\times (\Nl{\rmL})^{3}$ unitary matrix, 
$\mat{\Sigma}^{\rmL}$ an $(\Nl{\rmL})^{3}\times (\Nl{\rmL})^{3}$ diagonal matrix
containing the singular values of $\MM^{\rmL}$, and  
$\left(\mat{V}^{\rmL}\right)\Tr$ an $(\Nl{\rmL})^{3}\times 3(\Nl{\rmL})^{3}$ matrix with orthonormal rows. 
Although the solution ($\bm{a}_{m}^{\rmL}$ given by \Eq~\eqref{eq:opt1}) is optimal, 
it is not guaranteed to ${\bf exactly}$ satisfy \Eq~\eqref{eq:opt} (i.e., machine precision).  
The following theorem establishes an additional constrain on the surface metrics.
\begin{thrm}\label{thrm:coup5}
The surface metrics, $\bm{c}_{m}^{\rmL}$, must satisfy the additional constraint
\begin{equation}\label{eq:constrintmetrics}
\ones{\rmL}\Tr\bm{c}_{m}^{\rmL} = 0 \:,
\end{equation} 
to guarantee an exact solution of the discrete GCL condition given in \Eq~\eqref{eq:opt}. 
\end{thrm}
\begin{proof}
The matrix $\MM^{\rmL}$ is constructed from the transposes of three derivative operators:
$\left(\Dxil{l}^{\rmL}\right)\Tr\M^{\rmL}\:=\: \left(\Qxil{l}^{\rmL}\right)\Tr$, with the constant vector 
$\ones{\rmL}$, in its null space,
i.e., $\ones{\rmL}\Tr\left(\Qxil{l}^{\rmL}\right)\Tr=\bm{0}\Tr$.  Thus, the
matrix $\mat{M}^{\rmL}$ has a row rank of $(\Nl{\rmL})^{3}-1$ and one zero singular value.

Assemble the singular values in the diagonal matrix $\left(\mat{\Sigma}^{\rmL}\right)$ in descending order. 
This implies that i) the final singular value is zero, $\mat{\Sigma}((\Nl{\rmL})^{3},(\Nl{\rmL})^{3})=0$,
ii) the final orthonormal column vector in $\mat{U}^{\rmL}$ is a constant multiple of the vector, $\ones{\rmL}$, 
and iii) the pseudo-inverse is given by 
$\left(\mat{\Sigma}^{\rmL}\right)^{\dagger}(i,i) = \frac{1}{\mat{\Sigma}^{\rmL}(i,i)}$, $i=1,\dots,(\Nl{\rmL})^{3}-1$. 

The optimal solution given in \Eq~\eqref{eq:opt1} can thus be expressed as 
\begin{equation*}
\bm{a}_{m}^{\rmL}=\bm{a}_{m,\text{target}}^{\rmL}
-
\mat{V}^{\rmL} \left(\mat{\Sigma}^{\rmL}\right)^{\dagger}
\left(
   \mat{\Sigma}^{\rmL}\left(\mat{V}^{\rmL}\right)\Tr \bm{a}_{m,\text{target}}^{\rmL} 
- (\mat{U}^{\rmL})\Tr \bm{c}_{m}^{\rmL} \right) .
\end{equation*}
Since the diagonal matrices $\left(\mat{\Sigma}^{\rmL}\right)^{\dagger}$ and  $\mat{\Sigma}^{\rmL}$ 
are rank deficient in their $((\Nl{\rmL})^{3},(\Nl{\rmL})^{3})$ positions,  a zero solution exists only if the final 
constant vector, dotted with the metric data is zero, i.e., $(\mat{U}^{\rmL})\Tr \bm{c}_{m}^{\rmL} \:=\: 
\ones{\rmL}\Tr\bm{c}_{m}^{\rmL} = 0$.  The proof is established.
\end{proof}

Inspection of the definition for $\bm{c}_{m}^{\rmL}$ given in \Eq~\eqref{eq:constrintmetrics},
reveals that the constraint $\ones{\rmL}\Tr\bm{c}_{m}^{\rmL}\:=\:0$, may be interpreted as 
the discrete integral of surface metric data.  This integral must be zero as is motivated by the following derivation.

Integrate the continuous GCL equations over the domain $\OhatL$, and apply Gauss' divergence theorem 
to the integral.  The resulting expressions are
\begin{equation*}
\int_{\OhatL}\sum\limits_{l=1}^{3}
\frac{\partial}{\partial\xil{l}}\left(\fnc{J}\frac{\partial\xil{l}}{\partial\xm{m}}\right)\mr{d}\OhatL=
\oint_{\GammaL}\sum\limits_{l=1}^{3}\left(\fnc{J}\frac{\partial\xil{l}}{\partial\xm{m}}\right)\nxil{l}\mr{d}\GammaL
\approx \: \ones{\rmL}\Tr\bm{c}_{m}^{\rmL}\:=\:0.
\end{equation*}
Thus, the constraint $\ones{\rmL}\Tr\bm{c}_{m}^{\rmL}\:=\:0$ 
is the discrete surface metric equivalent of the analytic surface integral consistency constraint.

There are at least two convenient ways of enforcing this
constraint: 1) polynomially exact (analytic) surface metrics or 2) metrics constructed using conventional
FD approaches~\cite{Vinokur2002a,Thomas1979}.  The adequacy of these approaches is demonstrated by a careful investigation
of the surface metrics arising from tensor-product transformations. 

Consider a degree $p$ tensor-product curvilinear coordinate transformation. 
The surface metrics and the coupling terms constructed on the joint surface are of degree $2p-1$ in the 
surface orthogonal directions. Thus, a surface mass matrix with quadrature strength $2p-1$, (e.g., LGL) 
integrates the surface terms exactly and immediately satisfies the constraint: $\ones{\rmL}\Tr\bm{c}_{m}^{\rmL}=0$.
The following theorem establishes that the resulting coupling terms satisfy the required integral constraints.



\begin{thrm}\label{thrm:coup}
The coupling terms constructed on a conforming face between a conforming element and an element 
that has at least one other face that is nonconforming satisfy the condition~\eqref{eq:constrintmetrics} if either 
1) analytic or 2) standard FD approaches (e.g., Vinokur and Yee~\cite{Vinokur2002a} or Thomas and Lombard), 
are used to seed the metric terms.
\end{thrm}
\begin{proof}
This proof follows by inserting the approximation of the metric terms using 1) analytic, or 2) either~Vinokur or Yee~\cite{Vinokur2002a} or Thomas and Lombard~\cite{Thomas1979},
and the polynomial exactness of the differentiation matrices and the norm matrices 
\iftoggle{proofs}{ see Appendix~\ref{app:coup}.
}
{
(the details of the proof are given in Appendix G.~\cite{Fernandez2018_TM}).}
\end{proof}

\begin{remark}
The 1) GCL optimized volume metrics on all faces of a nonconforming element, and the
    2) surface metrics specified in the coupling matrices, are both surface normal approximations but
in general differ by a design order term.  They are weakly coupled through the GCL constraint given in equation~\eqref{eq:couplingsecondtwo}.
\end{remark}

The following theorem summarizes the properties of the proposed coupling approach.
\begin{thrm}\label{thrm:summarymort}
The proposed coupling procedure results in a scheme that has the following properties:
\begin{itemize}
\item The scheme is discretely entropy conservative (\ie, neutral discrete $L^{2}$ stability under the assumption of positive temperature and density)
\item The scheme is discretely elementwise conservative and preserves freestream
\item The scheme is of order $\pL+d-1$ where $\pL$ is the lowest degree used (the $+d$ comes from the 
fact that the PDE and therefore the discretization have been multiplied by the metric Jacobian which scales as $h^{d}$)
\end{itemize}
\end{thrm}
\begin{proof}
\iftoggle{proofs}{
The proof of entropy conservation is given in Section~\eqref{sec:semiEuler} while appropriate interface 
dissipation leading to entropy stability is discussed in Section~\ref{sec:dissipation}. Element-wise conservation 
is proven in Appendix~\ref{sec:gen_element_wise_conservation}. The order of the scheme results from the fact that one of the interpolation 
operators is suboptimal ($\pL-1$).
}
{The proof of entropy conservation is given in Section~\eqref{sec:semiEuler} while appropriate interface 
dissipation leading to entropy stability is discussed in Section~\ref{sec:dissipation}. Element-wise conservation 
is proven in Appendix E 
in~\cite{Fernandez2018_TM}. The order of the scheme results from the fact that one of the interpolation 
operators is suboptimal ($\pL-1$).}
\end{proof}
\section{Nonlinearly stable schemes: Burgers' equation}\label{sec:Burgers}
The nonconforming interface coupling mechanics presented in the previous section will be generalized 
for the compressible Euler equations in Section~\ref{sec:Euler}. 
Proving stability of nonlinear conservation laws, in general requires elaborate techniques well beyond 
simply discretizing the derivatives of the fluxes directly with SBP operators. Next, the Hadamard derivative
formalize is introduced in the context of one-dimensional inviscid Burgers' equation, which has the desirable
property of being amenable to either 1) a canonical derivative splitting, or 2) the Hadamard derivative 
approach.  The equivalence between the two approaches for Burgers' equation, is established 
elsewhere \cite{Fisher2013,Carpenter2015}.

Inviscid Burgers' equation and its well-known canonical splitting are of the form
\begin{equation}\label{eq:Burgerssplit}
  \frac{\partial\U}{\partial t}+\frac{\partial}{\partial \xm{1}}\left(\frac{\U^{2}}{2}\right) = 0 \quad;\quad
  \frac{\partial\U}{\partial t}+\frac{1}{3}\frac{\partial}{\partial \xm{1}}\left(\U^{2}\right)+\frac{\U}{3}\frac{\partial\U}{\partial \xm{1}} = 0.
\end{equation}
Performing an energy analysis on the split form of~\eqref{eq:Burgerssplit} gives (see, for instance, \cite{Carpenter2015} for details) 
\begin{equation}\label{eq:energyBurgers}
  \frac{1}{2}\frac{\mr{d}\|\U\|^{2}}{\mr{d}t}+\oint_{\Gamma}\frac{\U^{3}}{3}\nxm{1}\mr{d}\Gamma=0,\quad\|\U\|^{2}\equiv\int_{\Omega}\U^{2}\mr{d}\Omega.
\end{equation}
In order to prove stability, the semi-discrete scheme needs to mimic~\eqref{eq:energyBurgers} in the sense 
that when contracted with $\bm{u}\Tr\M$ (the discrete analogue of multiplying by the 
solution and integrating in space) the result is given by the sum of spatial terms that telescope 
to the boundaries. 

The discretization of~\eqref{eq:Burgerssplit} with SBP operators (ignoring the SATs) is given as 
\begin{equation}\label{eq:Burgerssplitdisc}
  \frac{\mr{d}\bm{u}}{\mr{d} t}+ \frac{1}{3}\Dxm{1}\diag\left(\bm{u}\right)\bm{u}+\frac{1}{3}\diag\left(\bm{u}\right)\Dxm{1}\bm{u}= 0.
\end{equation}
Multiplying~\eqref{eq:Burgerssplitdisc} by $\bm{u}\Tr\M$ results in 
\begin{equation}\label{eq:Burgerssplitdiscenergy}
  \frac{1}{2}\frac{\mr{d}\bm{u}\Tr\M \bm{u}}{\mr{d} t}+\frac{1}{3}\left(\bm{u}^{3}(N)-\bm{u}^{3}(1)\right)= 0,
\end{equation}
for which each term mimics~\eqref{eq:energyBurgers} and has the telescoping property, 
\ie, what remains are terms 
at the boundaries.

Now the discretization~\eqref{eq:Burgerssplitdisc} is recast using the Hadamard derivative formalism.
The Hadamard derivative operator and the equivalent split form operators are given as follows 
\begin{equation}\label{eq:Burgershadamard}
   2\Dxm{1}\circ\matFxm{1}{\bm{u}}{\bm{u}}\ones{1}
  \quad \leftrightarrow \quad
  \frac{1}{3}\Dxm{1}\diag\left(\bm{u}\right)\bm{u}+\frac{1}{3}\diag\left(\bm{u}\right)\Dxm{1}\bm{u} \: .
\end{equation}

The Hadamard operator is capable of compactly representing various 
split forms, and more importantly, extends to nonlinear equations for which a canonical split form is inappropriate.

The Hadamard derivative operator is constructed from two components:  1)  an SBP derivative operator, and 2) a two
point flux function specific to the physics being modeled.  Two point fluxes are constructed between the center
point and all other points of dependency within
the SBP stencil.  The SBP telescoping property~\cite{Fisher2013} that results from precise local cancellation 
of spatial terms, can then be extended directly to nonlinear operators.  A simple example is now presented.

Consider the two point Burgers' flux function defined by \cite{Tadmor2003,Carpenter2015} 
\begin{equation*}
\fxmsc{m}{\bmui{i}}{\bmui{j}}\equiv 
 \frac{\left\{\left(\bmui{i}\right)^{2}+\bmui{i}\bmui{j}+\left(\bmui{j}\right)^{2}\right\}}{6},
\end{equation*}
where $\bmui{i}$ and $\bmui{j}$ are the $i\Th$ and $j\Th$ components of $\bm{u}$,
and for the purpose of demonstration, a simple SBP operator constructed on the LGL nodes $\left(-1,0,1\right)$
\begin{equation*}
    \mat{D}_{\xm{1}}=
    \left[ \begin {array}{ccc} -\frac{3}{2}&2&-\frac{1}{2}\\ \noalign{\medskip}-\frac{1}{2}&0&\frac{1}{2}
\\ \noalign{\medskip}\frac{1}{2}&-2&\frac{3}{2}\end {array} \right]. 
\end{equation*}
The two argument Hadamard matrix flux, \matFxm{m}{\bm{u}}{\bm{u}} is given as
\begin{equation*}
  \begin{split}
    &\matFxm{m}{\bm{u}}{\bm{u}}=\\
    &\left[
        \begin{array}{ccc}
            \frac{\left(\bmui{1}\right)^{2}}{2}&
            \frac{\left(\bmui{1}\right)^{2}+\bmui{1}\bmui{2}+\left(\bmui{2}\right)^{2}}{6}&
            \frac{\left(\bmui{1}\right)^{2}+\bmui{1}\bmui{3}+\left(\bmui{3}\right)^{2}}{6}\\\\
\frac{\left(\bmui{2}\right)^{2}+\bmui{2}\bmui{1}+\left(\bmui{1}\right)^{2}}{6}&
\frac{\left(\bmui{2}\right)^{2}}{2}&
\frac{\left(\bmui{2}\right)^{2}+\bmui{2}\bmui{3}+\left(\bmui{3}\right)^{2}}{6}\\\\
\frac{\left(\bmui{3}\right)^{2}+\bmui{3}\bmui{1}+\left(\bmui{1}\right)^{2}}{6}&
\frac{\left(\bmui{3}\right)^{2}+\bmui{3}\bmui{2}+\left(\bmui{2}\right)^{2}}{6}&
\frac{\left(\bmui{3}\right)^{2}}{2}
        \end{array}
    \right].
      \end{split}
\end{equation*}
Thus, 
\begin{equation*}
  \begin{split}
    &\mat{D}_{\xm{1}}\circ\matFxm{m}{\bm{u}}{\bm{u}}\bm{1}=\\
    &\left[
      \arraycolsep=0.5pt
        \begin{array}{ccc}
           {\scriptscriptstyle-\frac{3}{2} \frac{\left(\bmui{1}\right)^{2}}{2}}&
            {\scriptscriptstyle2\frac{\left(\bmui{1}\right)^{2}+\bmui{1}\bmui{2}+\left(\bmui{2}\right)^{2}}{6}}&
            {\scriptscriptstyle-\frac{1}{2}\frac{\left(\bmui{1}\right)^{2}+\bmui{1}\bmui{3}+\left(\bmui{3}\right)^{2}}{6}}\\\\
{\scriptscriptstyle-\frac{1}{2}\frac{\left(\bmui{2}\right)^{2}+\bmui{2}\bmui{1}+\left(\bmui{1}\right)^{2}}{6}}&
{\scriptscriptstyle0}&
{\scriptscriptstyle\frac{1}{2}\frac{\left(\bmui{2}\right)^{2}+\bmui{2}\bmui{3}+\left(\bmui{3}\right)^{2}}{6}}\\\\
{\scriptscriptstyle\frac{1}{2}\frac{\left(\bmui{3}\right)^{2}+\bmui{3}\bmui{1}+\left(\bmui{1}\right)^{2}}{6}}&
{\scriptscriptstyle-2\frac{\left(\bmui{3}\right)^{2}+\bmui{3}\bmui{2}+\left(\bmui{2}\right)^{2}}{6}}&
{\scriptscriptstyle\frac{3}{2}\frac{\left(\bmui{3}\right)^{2}}{2}}
        \end{array}
    \right]
    \left[
        \begin{array}{c}
            {\scriptscriptstyle1}\\\\
            {\scriptscriptstyle1}\\\\
            {\scriptscriptstyle1}
        \end{array}
        \right].
      \end{split}
\end{equation*}
The equivalence between the two approaches is evident by inspection.

Now the general notation applicable to the compressible Euler equations is described. 
Assume that what is required is the discretization of the derivative of a flux vector $\Fxm{m}$ in the
$\xm{m}$ Cartesian direction. The essential ingredients used above are an SBP matrix difference operator, $\Dxm{m}$, and a 
two argument matrix flux function,
 $\matFxm{m}{\uk}{\ur}$, which is composed of diagonal matrices and is defined block-wise as
\begin{equation*}
    \begin{split}
  &\left(\matFxm{m}{\uk}{\ur}\right)\left(e(i-1)+1:ei,e(j-1)+1:ej\right) \equiv  \diag\left(
  \fxmsc{m}{\uki{(i)}}{\uri{(j)}}    
  \right),\\\\
  &\uki{(i)}\equiv\uk\left(e(i-1)+1:ei\right),\;
  \uri{(j)}\equiv\ur\left(e(j-1)+1:ej\right),\\\\
  &\quad i = 1\dots,\Nl{\kappa}^{3},\; j = 1,\dots,\Nl{r}^{3},
    \end{split}
\end{equation*}
where $e$ is the number of equations in the system of PDEs. For example, for the compressible Euler 
equations, $e=5$ and the two argument matrix flux function is of size 
$\left(e\,\Nl{\kappa}^{3}\right)\times\left(e\,\Nl{r}^{3}\right)$, where $e\Nl{\kappa}^{3}$ and $e\Nl{r}^{3}$ are 
the total number of entries in the vectors $\uk$ and $\ur$ corresponding the solution 
variables in elements $\kappa$ and $r$, respectively. Thus, $\uki{(i)}$ is the vector of the $e$ solution 
variables evaluated at the $i\Th$ node.
The vectors $\fxmsc{m}{\uki{(i)}}{\uri{(j)}}$ 
are constructed from two-point flux functions that are symmetric in their arguments, $\left(\uki{(i)},\uri{(j)}\right)$, and consistent. 
Thus,
\begin{equation*}
 \fxmsc{m}{\uki{(i)}}{\uri{(j)}}= \fxmsc{m}{\uri{(j)}}{\uki{(i)}},
\quad\fxmsc{m}{\uki{(i)}}{\uki{(i)}} = \Fxm{m}\left(\uki{(i)}\right),
\end{equation*}
where $\Fxm{m}$ is the flux vector in the $\xm{m}\Th$ Cartesian direction. Using these ingredients, an approximation to the derivative 
$\frac{\partial\Fxm{m}}{\partial\xm{m}}$ is constructed as 
\begin{equation*}
    2\Dxm{m}\circ\matFxm{m}{\qk{\kappa}}{\qk{\kappa}}\ones{\kappa}\approx
    \frac{\partial\Fxm{m}}{\partial\xm{m}}\left(\bm{\xi}^{\kappa}\right),
\end{equation*}
where $\bm{\xi}^{\kappa}$ is the vector of vectors containing the nodal locations. 
The resulting approximation has the same order properties as differentiating directly with the SBP operator $\Dxm{m}$ 
(see Theorem  $1$ in Crean~\etal~\cite{Crean2018}).

\section{Application to the compressible Euler equations}\label{sec:Euler}
In this section, the nonconforming algorithm presented in Section~\ref{sec:discretizationconvection} is combined with the mechanics presented in Section~\ref{sec:Burgers} 
to construct an entropy conservative discretization of the compressible Euler equations on $p$-nonconforming 
meshes. First, in Section~\ref{sec:ContinuousentropyEuler} the continuous equations and the continuous entropy analysis are reviewed, then 
in Section~\ref{sec:semiEuler} the semi-discrete algorithm for the Euler equations is presented and analyzed.

\subsection{Review of the continuous entropy analysis}\label{sec:ContinuousentropyEuler}
The strong form of the compressible Euler equations in Cartesian coordinates are given as
\begin{equation}\label{eq:compressible_euler}
\begin{aligned}
  &\frac{\partial\Q}{\partial t}+\sum\limits_{m=1}^{3}\frac{\partial \Fxm{m}}{\partial \xm{m}} = 0, &&
\forall \left(\xm{1},\xm{2},\xm{3}\right)\in\Omega,\quad t\ge 0,\\
  &\Q\left(\xm{1},\xm{2},\xm{3},t\right)=\GB\left(\xm{1},\xm{2},\xm{3},t\right), && \forall \left(\xm{1},\xm{2},\xm{3}\right)\in\Gamma,\quad t\ge 0,\\
&\Q\left(\xm{1},\xm{2},\xm{3},0\right)=\Gzero\left(\xm{1},\xm{2},\xm{3},0\right), &&
  \forall \left(\xm{1},\xm{2},\xm{3}\right)\in\Omega.
\end{aligned}
\end{equation}

The vectors $\Q$ and $\Fxm{}$ respectively denote the conserved variables and the inviscid fluxes. The boundary data,
$\GB$, and the initial condition, $\Gzero$, are assumed to be in $L^{2}(\Omega)$,
with the further assumption that $\GB$ will be set to coincide with linear
well-posed boundary conditions and such that entropy conservation or stability is achieved.

It is well-known that the compressible Euler equations \eqref{eq:compressible_euler} possess
a convex extension that, when integrated over the physical domain $\Omega$,
only depends on the boundary data on $\Gamma$. Such an extension yields the entropy function
\begin{equation}\label{eq:entropy_function}
\fnc{S}=-\rho s,
\end{equation}
where $\rho$ and $s$ are the density and the thermodynamic entropy, respectively.
The entropy function, $\fnc{S}$, is convex if the thermodynamic variables
are positive and is a useful tool for proving stability in the $L^{2}$
norm \cite{dafermos-book-2010}.
Following the analysis described in \cite{Carpenter2014,Parsani2015,Carpenter2016},
the system \eqref{eq:compressible_euler} is multiplied by the (local) entropy
variables $\bfnc{W}\Tr = \partial \fnc{S} / \partial \Q$ and by
using the fact that
\begin{equation}\label{eq:compat}
\frac{\partial \fnc{S}}{\partial \Q}\frac{\partial\Fxm{m}}{\partial \xm{m}}=
\frac{\partial \fnc{S}}{\partial \Q}\frac{\partial\Fxm{m}}{\partial \Q}
\frac{\partial\Q}{\partial \xm{m}}=
\frac{\partial\fxm{m}}{\partial\Q}\frac{\partial\Q}{\partial \xm{m}}=
\frac{\partial\fxm{m}}{\partial \xm{m}},\qquad m=1,2,3,
\end{equation}
gives that
\begin{equation}\label{eq:NSCe1}
\begin{split}
\frac{\partial\fnc{S}}{\partial t}+
\sum\limits_{m=1}^{3}\frac{\partial\fxm{m}}{\partial\xm{m}}=0,
\end{split}
\end{equation}
where the scalars $\fxm{m}(\Q)$ are the entropy fluxes in the $\xm{m}$-direction. 
Then, integrating over the domain, $\Omega$, and using integration by parts yields 
\begin{equation}\label{eq:NSCe3}
\displaystyle\int_{\Omega}\frac{\partial\fnc{S}}{\partial t}\mr{d}\Omega\leq
\oint_{\Gamma}\left(\sum\limits_{m=1}^{3}\fxm{m}\nxm{m}\right)\mr{d}\Omega,
\end{equation}
where $\nxm{m}$ is the $m\Th$ component of the outward facing unit normal. 
Note that the equality in \eqref{eq:NSCe3} holds for smooth flows; conversely, 
the inequality is valid for non smooth flow since the mathematical entropy decreases across shocks.

To obtain a bound on the solution, the inequality~\eqref{eq:NSCe3} is integrated in time and assuming 
boundary conditions and an initial condition that are nonlinearly well posed, and positivity of density 
and temperature, the result can be turned into a bound on the solution in terms of the data of the problem (see~\cite{Svard2015,dafermos-book-2010}). Example of fully-discrete explicit entropy conservative/stable algorithms are presented, for instance, 
in \cite{Friedrich2019,ranocha2019relaxation}.
\subsection{A $p$-nonconforming algorithm}\label{sec:semiEuler}
As for the convection equation, a skew-symmetrically split form of the compressible Euler equations is used to 
construct an entropy conservative/stable algorithm:
\begin{equation}\label{eq:NSCCS1}
\begin{split}
&\Jk\frac{\partial\Qk}{\partial t}+
\sum\limits_{l,m=1}^{3}\frac{1}{2}\left\{
\frac{\partial}{\partial\xil{l}}\left(\Jdxildxm{l}{m}\Fxm{m}\right)
+\Jdxildxm{l}{m}\frac{\partial\Fxm{m}}{\partial\xil{1}}
\right\}\\
&-\frac{1}{2}\sum\limits_{l,m=1}^{3}\Fxm{m}\frac{\partial}{\partial\xil{l}}\left(\Jdxildxm{l}{m}\right)
=0,
\end{split}
\end{equation}
where the last term on the left-hand side is zero as a result of the GCL conditions~\eqref{eq:GCL}.

 The discretization is developed using the same macro element SBP 
operator as in Section~\ref{sec:discretizationconvection}. Thus, the discretization of~\eqref{eq:NSCCS1} over the macro element is given as 
\begin{equation}\label{eq:NSCCS1disc}
\begin{split}
&\Jtilde\frac{\mr{d}\qtilde}{\partial t}+
\sum\limits_{l,m=1}^{3}\Dtildelm{l}{m}
\circ\matFxm{m}{\qtilde}{\qtilde}\tildeone
-\frac{1}{2}\sum\limits_{l,m=1}^{3}\diag\left(\bm{f}_{\xm{m}}\right)\Dtildel{l}\tildebarmatAlm{l}{m}\tildeone
=\bm{0},\\&\qtilde\equiv\left[\qk{\rmL}\Tr,\qk{\rmH}\Tr\right]\Tr,
\end{split}
\end{equation}
where $\bm{f}_{\xm{m}}$ is a vector of vectors constructed by evaluating $\Fxm{m}$ at the mesh nodes. 
That the factor of $\frac{1}{2}$ on the skew-symmetric volume terms disappears as a result of 
using the nonlinear operator, \eg, 
$2\Dxil{l}\circ\matFxm{m}{\qk{\kappa}}{\qk{\kappa}}\ones{\kappa}\approx\frac{\partial\Fxm{m}}{\partial\xil{l}}(\bm{\xi}^{\kappa})$.
Moreover, the flux function matrix, $\matFxm{m}{\qtilde}{\qtilde}$, is constructed using a two point flux 
function, $\fxmsc{m}{\tildeqi{(i)}}{\tildeqi{(i)}}$, that satisfies 
the Tadmor shuffle condition~\cite{Tadmor2003}
\begin{equation}\label{eq:shuffle}
\left(\tildewi{(i)}-\tildewi{(j)}\right)\Tr\fxmsc{m}{\tildeqi{(i)}}{\tildeqi{(i)}}=\tilde{\bm{\psi}}_{\xm{m}}^{(i)}-\tilde{\bm{\psi}}_{\xm{m}}^{(j)}.
\end{equation}
The $\Dtildelm{l}{m}$ operators are constructed as in Section~\eqref{sec:discretizationconvection} by tensoring the contributing matrices with an identity 
matrix $\Imat{5}$ to accommodate the system of $5$ equations. For example,
\begin{equation*}
  \Dxil{1}^{\rmL}\equiv\barDxil{1}\otimes\Imat{5},\quad \barDxil{1}^{\rmL}\equiv\DoneD_{\rmL}\otimes\Imat{\Nl{\rmL}}\otimes\Imat{\Nl{\rmL}}.
\end{equation*}

Nonlinear stability requires that the last set of terms on the left-hand side of~\eqref{eq:NSCCS1disc} be zero.
This requirement leads to discrete GCL conditions which are identical to those obtained for the convection equation 
(with $\Dxil{l}$ replaced with $\barDxil{l}$), i.e., conditions~\eqref{eq:discGCLL} and~\eqref{eq:discGCLH}.

The procedure to demonstrate that the proposed scheme is entropy conservative (i.e., it telescopes to the element
boundaries) follows the 
continuous analysis in a one-to-one fashion (see Section \ref{sec:ContinuousentropyEuler}). To simplify the derivation,
the following matrices are introduced: 
\begin{equation*}
  \begin{split}
  &\Dhatm{m}\equiv\sum\limits_{l=1}^{3}\Dtildelm{l}{m},
\quad\Qhatm{m}\equiv\Mtilde\Dhatm{m},\quad\Ehatm{m}\equiv\Qhatm{m}+\Qhatm{m}\Tr.
  \end{split}
\end{equation*}
Thus, assuming that the discrete GCL conditions are satisfied, \eqref{eq:NSCCS1disc} becomes
\begin{equation}\label{eq:NSCCS1disc2}
\begin{split}
&\Jtilde\frac{\mr{d}\qtilde}{\partial t}+
\sum\limits_{m=1}^{3}\Dhatm{m}\circ\matFxm{m}{\qtilde}{\qtilde}\tildeone=\bm{0}.
\end{split}
\end{equation}
The first step is to contract~\eqref{eq:NSCCS1disc2}, point-wise, with the entropy variables 
(at the continuous level this was the step 
$\fnc{W}\frac{\partial\Fxm{m}}{\partial\xm{m}}=\frac{\partial\fnc{F}_{\xm{m}}}{\partial\xm{m}}$). 
This is accomplished by multiplying~\eqref{eq:NSCCS1disc2} by 
$\Imat{\tilde{N}}\otimes\ones{5}\Tr\diag\left(\wk{\kappa}\right)$, which results in 
\begin{equation}\label{eq:NSCCS1disc3}
\begin{split}
&\Jtilde \, \Imat{\tilde{N}}\otimes\ones{5}\Tr\diag\left(\wk{\kappa}\right) \frac{\mr{d}\qtilde}{\partial t}+
\Imat{\tilde{N}}\otimes\ones{5}\Tr\diag\left(\wk{\kappa}\right)\sum\limits_{m=1}^{3}\Dhatm{m}\circ\matFxm{m}{\qtilde}{\qtilde}\tildeone=\bm{0},
\end{split}
\end{equation}
where $\tilde{N}$ is the total number of nodes in the macro element. Adding and subtracting terms to~\eqref{eq:NSCCS1disc3} gives
\begin{equation}\label{eq:NSCCS1disc4}
\begin{split}
\Jtilde\frac{\mr{d}\stilde}{\partial t}+
&\Imat{\tilde{N}}\otimes\ones{5}\Tr\sum\limits_{m=1}^{3}\left\{
  \frac{1}{2}\diag\left(\wk{\kappa}\right)\Dhatm{m}\circ\matFxm{m}{\qtilde}{\qtilde}\tildeone
+\frac{1}{2}\Dhatm{m}\diag\left(\wk{\kappa}\right)\circ\matFxm{m}{\qtilde}{\qtilde}\tildeone\right.\\
  &\left.+\frac{1}{2}\diag\left(\wk{\kappa}\right)\Dhatm{m}\circ\matFxm{m}{\qtilde}{\qtilde}\tildeone
-\frac{1}{2}\Dhatm{m}\diag\left(\wk{\kappa}\right)\circ\matFxm{m}{\qtilde}{\qtilde}\tildeone
  \right\}=\bm{0}.
\end{split}
\end{equation}
Consider the $i\Th$ term of the spatial terms of~\eqref{eq:NSCCS1disc4} (denoted $\Vol(i)$ where $i=1,\dots,\tilde{N}$), 
\begin{equation}\label{eq:NSCCS1disc5}
  \Vol(i)\equiv\sum\limits_{j=1}^{\tilde{N}}\barDhatm{m}(i,j)\left\{
    \frac{\left(\wtildei{i}+\wtildei{j}\right)}{2}\fxmsc{m}{\qtildei{(i)}}{\qtildei{(j)}}
   +\frac{\left(\wtildei{i}-\wtildei{j}\right)}{2}\fxmsc{m}{\qtildei{(i)}}{\qtildei{(j)}} 
  \right\},
\end{equation}
where $\Dhatm{m}=\barDhatm{m}\otimes\Imat{5}$. Using the Tadmor shuffle condition~\eqref{eq:shuffle} on the second set of terms results in
\begin{equation}\label{eq:NSCCS1disc6}
  \Vol(i)=\sum\limits_{j=1}^{\tilde{N}}\barDhatm{m}(i,j)\left\{
    \frac{\left(\wtildei{i}+\wtildei{j}\right)}{2}\fxmsc{m}{\qtildei{(i)}}{\qtildei{(j)}}
   +\frac{\left(\psixmtildei{m}{i}-\psixmtildei{m}{j}\right)}{2} 
  \right\}.
\end{equation}
Adding and subtracting terms to~\eqref{eq:NSCCS1disc6} gives
\begin{equation}\label{eq:NSCCS1disc7}
  \begin{split}
  \Vol(i)=&\sum\limits_{j=1}^{\tilde{N}}\barDhatm{m}(i,j)\left\{
    \frac{\left(\wtildei{i}+\wtildei{j}\right)}{2}\fxmsc{m}{\qtildei{(i)}}{\qtildei{(j)}}
   -\frac{\left(\psixmtildei{m}{i}+\psixmtildei{m}{j}\right)}{2} 
  \right\}\\
  &+\sum\limits_{j=1}^{\tilde{N}}\barDhatm{m}(i,j)
   \psixmtildei{m}{i}
  \end{split}. 
\end{equation}
The last set of terms can be expanded into matrix form as (\ie, the vector constructed from $\Vol(i)$ for $i=1,..,\tilde{N}$) 
\begin{equation}
  \psixmtilde{m}\Tr\sum\limits_{l=1}^{3}\left\{\barDxil{l}\barmatAlmk{l}{m}{}+\barmatAlmk{l}{m}{}\barDxil{l}\right\}\barones{}=\bm{0},
\end{equation}
where the equality follows from the assumption that the discrete GCL conditions and consistency of the first derivative 
operator, \ie, $\barDxil{l}\barones{}=\bm{0}$. \Eq~\eqref{eq:NSCCS1disc7} can now be recast in terms of 
a new flux function matrix, $\matFxmtilde{m}{\qtilde}{\qtilde}$ constructed from the two-point flux function
\begin{equation*}
  \fxmsctilde{m}{\qtildei{(i)}}{\qtildei{(i)}}\equiv\frac{\left(\wtildei{i}+\wtildei{j}\right)}{2}\fxmsc{m}{\qtildei{(i)}}{\qtildei{(j)}}
   -\frac{\left(\psixmtildei{m}{i}+\psixmtildei{m}{j}\right)}{2}.
\end{equation*}
Note that $\tilde{\bm{f}}_{\xm{n}}^{sc}$ is symmetric and consistent with $\fxm{m}$; thus, \eqref{eq:NSCCS1disc3} reduces 
to
\begin{equation}\label{eq:NSCCS1disc9}
\begin{split}
&\Jtilde\frac{\mr{d}\stilde}{\partial t}+
\sum\limits_{m=1}^{3}\barDhatm{m}\circ\matFxmtilde{m}{\qtilde}{\qtilde}\barones{}=\bm{0},
\end{split}
\end{equation}
where 
\iftoggle{proofs}{ 
(see Appendix~\ref{app:nonlinear} for details on the approximation properties of the nonlinear operators)
}
{
(see Appendix A in~\cite{Fernandez2018_TM} for details on the approximation properties of the nonlinear operators)}
\begin{equation*}
\barDhatm{m}\circ\matFxmtilde{m}{\qtilde}{\qtilde}\barones{}\approx
\frac{1}{2}\sum\limits_{l=1}^{3}\left\{
  \frac{\partial}{\partial\xil{l}}\left(\Jdxildxm{l}{m}\fxm{m}\right)
  +\Jdxildxm{l}{m}\frac{\partial\fxm{m}}{\partial\xil{l}}
\right\}\left(\tilde{\bm{\xi}}\right).
\end{equation*}
Next, \eqref{eq:NSCCS1disc9} is discretely integrated over the domain by left multiplying by $\barones{}\Tr\barM$ 
and rearranging, results in 
\begin{equation}\label{eq:NSCCS1disc10}
\begin{split}
\barones{}\Tr\barM\Jtilde\frac{\mr{d}\stilde}{\partial t}&=
-\sum\limits_{m=1}^{3}\barones{}\Tr\barQhatm{m}\circ\matFxmtilde{m}{\qtilde}{\qtilde}\barones{},\\
&=-\frac{1}{2}\sum\limits_{m=1}^{3}\left\{
  \barones{}\Tr\barQhatm{m}\circ\matFxmtilde{m}{\qtilde}{\qtilde}\barones{}
+\barones{}\Tr\barQhatm{m}\Tr\circ\matFxmtilde{m}{\qtilde}{\qtilde}\barones{}
\right\}\\
&=-\frac{1}{2}\sum\limits_{m=1}^{3}\left\{
  \barones{}\Tr\barQhatm{m}\circ\matFxmtilde{m}{\qtilde}{\qtilde}\barones{}
-\barones{}\Tr\barQhatm{m}\circ\matFxmtilde{m}{\qtilde}{\qtilde}\barones{}
+\barones{}\Tr\barEhatm{m}\circ\matFxmtilde{m}{\qtilde}{\qtilde}\barones{}
\right\}\\
&=-\frac{1}{2}\sum\limits_{m=1}^{3}
\barones{}\Tr\barEhatm{m}\circ\matFxmtilde{m}{\qtilde}{\qtilde}\barones{}\\
&= 
-\frac{1}{2}\sum\limits_{m=1}^{3}\sum\limits_{i,j=1}^{N}\barEhatm{m}(i,j)\left\{
\frac{\left(\wtildei{i}+\wtildei{j}\right)}{2}\fxmsc{m}{\qtildei{(i)}}{\qtildei{(j)}}-\frac{\left(\psixmtildei{m}{i}+\psixmtildei{m}{j}\right)}{2}  
\right\}\\
&=-\frac{1}{2}\sum\limits_{m=1}^{3}\wtilde\Tr\Ehatm{m}\circ\matFxm{m}{\qtilde}{\qtilde}\tildeone+\frac{1}{2}\barones{}\Tr\barEhatm{m}\psitildem{m}.
\end{split}
\end{equation}
Thus,
\begin{equation}\label{eq:NSCCS1disc11}
\begin{split}
\barones{}\Tr\barM\Jtilde\frac{\mr{d}\stilde}{\partial t}=
-\frac{1}{2}\sum\limits_{m=1}^{3}\wtilde\Tr\Ehatm{m}\circ\matFxm{m}{\qtilde}{\qtilde}\tildeone+\frac{1}{2}\barones{}\Tr\barEhatm{m}\psitildem{m}.
\end{split}
\end{equation}
The right-hand side of~\eqref{eq:NSCCS1disc11} contains terms constructed from the $\mat{E}$ matrices.
As a result, these terms can be decomposed into the contributions of the separate surfaces of the element (node-wise). For periodic problems, 
these terms would cancel out with the contributions from face SATs coupling terms leading to entropy 
conservation. For non-periodic problems appropriate SATs need to be constructed so that an entropy inequality or equality is attained (this is an active area of research; see, for example, \cite{Parsani2015,Svard2018,dalcin_2019_wall_bc}). It is important to highlight that 
when nonlinear systems of PDEs are considered, it is not possible to attain
the telescoping property \eqref{eq:NSCCS1disc4} by simply using SBP operators to discretely differentiate 
fluxes.

This section is finished by presenting the discretization for each element separately. 
The semi-discrete forms of the elements $\rmL$ and $\rmR$ are, respectively
\begin{equation}\label{eq:discEL}
  \begin{split}
  &\matJk{\rmL}\frac{\mr{d}\qL}{\mr{d}t}+\sum\limits_{l,m=1}^{3}\left(\Dxil{l}^{\rmL}\matAlmk{l}{m}{\rmL}+\matAlmk{m}{l}{\rmL}\Dxil{l}^{\rmL}\right)
  \circ\matFxm{m}{\qL}{\qL}\ones{\rmL}=\\
 &\left(\M^{\rmL}\right)^{-1}\sum\limits_{m=1}^{3}\left\{
  \left(\eNl{\rmL}\eNl{\rmL}\Tr\PoneD_{\rmL}\otimes\PoneD_{\rmL}\otimes\Imat{5}\right)\matAlmk{1}{m}{\rmL}
  \right\}
  \circ\matFxm{m}{\qL}{\qL}\ones{\rmL}\\
 &\left(\M^{\rmL}\right)^{-1}\sum\limits_{m=1}^{3}\left\{
  \matAlmk{1}{m}{\rmL}\left(\eNl{\rmL}\eNl{\rmL}\Tr\PoneD_{\rmL}\otimes\PoneD_{\rmL}\otimes\Imat{5}\right)\right\}
  \circ\matFxm{m}{\qL}{\qL}\ones{\rmL}\\
  &-\left(\M^{\rmL}\right)^{-1}\sum\limits_{m=1}^{3}\left\{
\left(\eNl{\rmL}\otimes\Imat{\Nl{\rmL}}\otimes\Imat{\Nl{\rmL}}\otimes\Imat{5}\right)
  \matAlmkLH{1}{m}{\rmL}{\Ghat}{red}\left(\PoneD_{\rmL}\IHtoLoneD\otimes\PoneD_{\rmL}\IHtoLoneD\otimes\Imat{5}\right)\right.\\
  &\left.\left(\eonel{\rmH}\Tr\otimes\Imat{\Nl{\rmH}}\otimes\Imat{\Nl{\rmH}}\otimes\Imat{5}\right)\right\}\circ\matFxm{m}{\qL}{\qH}\ones{\rmH}\\
  &-\left(\M^{\rmL}\right)^{-1}\sum\limits_{m=1}^{3}\left\{
\left(\eNl{\rmL}\otimes\Imat{\Nl{\rmL}}\otimes\Imat{\Nl{\rmL}}\otimes\Imat{5}\right)
  \left(\PoneD_{\rmL}\IHtoLoneD\otimes\PoneD_{\rmL}\IHtoLoneD\otimes\Imat{5}\right)\matAlmkLH{1}{m}{\rmH}{\Ghat}{red}\right.\\
  &\left.\left(\eonel{\rmH}\Tr\otimes\Imat{\Nl{\rmH}}\otimes\Imat{\Nl{\rmH}}\otimes\Imat{5}\right)\right\}\circ\matFxm{m}{\qL}{\qH}\ones{\rmH},
  \end{split}
\end{equation}

and

\begin{equation}\label{eq:discEH}
  \begin{split}
  &\matJk{\rmH}\frac{\mr{d}\qH}{\mr{d}t}+\sum\limits_{l,m=1}^{3}\left(\Dxil{l}^{\rmH}\matAlmk{l}{m}{\rmH}+\matAlmk{m}{l}{\rmH}\Dxil{l}^{\rmH}\right)
  \circ\matFxm{m}{\qH}{\qH}\ones{\rmL}=\\
 &-\left(\M^{\rmH}\right)^{-1}\sum\limits_{m=1}^{3}\left\{
  \left(\eonel{\rmH}\eonel{\rmH}\Tr\PoneD_{\rmH}\otimes\PoneD_{\rmH}\otimes\Imat{5}\right)\matAlmk{1}{m}{\rmH}
  \right\}
  \circ\matFxm{m}{\qH}{\qH}\ones{\rmH}\\
 &-\left(\M^{\rmH}\right)^{-1}\sum\limits_{m=1}^{3}\left\{
  \matAlmk{1}{m}{\rmH}\left(\eonel{\rmH}\eonel{\rmH}\Tr\PoneD_{\rmH}\otimes\PoneD_{\rmH}\otimes\Imat{5}\right)\right\}
  \circ\matFxm{m}{\qH}{\qH}\ones{\rmH}\\
  &+\left(\M^{\rmH}\right)^{-1}\sum\limits_{m=1}^{3}\left\{
\left(\eonel{\rmH}\otimes\Imat{\Nl{\rmH}}\otimes\Imat{\Nl{\rmH}}\otimes\Imat{5}\right)
  \matAlmkLH{1}{m}{\rmH}{\Ghat}{red}\left(\PoneD_{\rmH}\ILtoHoneD\otimes\PoneD_{\rmH}\ILtoHoneD\otimes\Imat{5}\right)\right.\\
  &\left.\left(\eNl{\rmL}\Tr\otimes\Imat{\Nl{\rmL}}\otimes\Imat{\Nl{\rmL}}\otimes\Imat{5}\right)\right\}\circ\matFxm{m}{\qH}{\qL}\ones{\rmL}\\
  &+\left(\M^{\rmH}\right)^{-1}\sum\limits_{m=1}^{3}\left\{
\left(\eonel{\rmH}\otimes\Imat{\Nl{\rmH}}\otimes\Imat{\Nl{\rmH}}\otimes\Imat{5}\right)
  \left(\PoneD_{\rmH}\ILtoHoneD\otimes\PoneD_{\rmH}\ILtoHoneD\otimes\Imat{5}\right)\matAlmkLH{1}{m}{\rmL}{\Ghat}{red}\right.\\
  &\left.\left(\eNl{\rmL}\Tr\otimes\Imat{\Nl{\rmL}}\otimes\Imat{\Nl{\rmL}}\otimes\Imat{5}\right)\right\}\circ\matFxm{m}{\qH}{\qL}\ones{\rmL}.
  \end{split}
\end{equation}

 The algorithm presented above leads to an entropy conservative discretization (modulo what boundary conditions are imposed). 
 However, the main interest is in entropy stable algorithms and the approach use herein to achieve this is to augment the 
 entropy conservative scheme with an appropriate interface dissipation. How this is accomplished is discussed in 
 the next section. 
\section{Inviscid interface dissipation and boundary SATs}\label{sec:dissipation}
To make the inviscid entropy conservative scheme entropy stable, an interface dissipation needs to be added. 
Herein, as in the conforming algorithms \cite{Carpenter2014,Parsani2015,Carpenter2016,Parsani2016}, 
the numerical dissipation is motivated by the up-winding used in a Roe approximate Riemann solver which has the form
\begin{equation}\label{eq:numbflux}
\bfnc{F}^{*}=\frac{\bfnc{F}^{+}+\bfnc{F}^{-}}{2}-\frac{1}{2}\Y\left|\mat{\Lambda}\right|\Y^{-1}\left(\bfnc{Q}^{+}-\bfnc{Q}^{-}\right),
\end{equation}
where $+$ and $-$ refer to quantities evaluated on the side of the interface in the positive and negative 
face normal directions, respectively (see Figure~\ref{fig:non}).  

The original Roe flux: $\bfnc{F}^{*}$ is composed of an inviscid flux average:
$(\bfnc{F}^{+}+\bfnc{F}^{-})/2$ and a dissipation term.  
The flux average term is not in general entropy conservative, and is replaced with an entropy conservative 
two-point flux, \eg, Chandrashekar~\cite{Chandrashekar2013} or Ismail and Roe~\cite{Ismail2009affordable}.
%
The remaining dissipation term, is reformulated in quadratic form to facilitate an entropy stability proof. 
This step is accomplished by using the flux Jacobian with respect to the 
entropy variables, $\bfnc{W}$, rather than the conservative variables, $\bfnc{Q}$. 
A unique scaling of the eigenvectors of the conservative variable flux Jacobian facilitates the 
reformulation~\cite{Merriam1989}.
\begin{equation*}
\frac{\partial\bfnc{Q}}{\partial\bfnc{W}}=\Y\Y\Tr \quad ; \quad
\frac{\partial\bfnc{F}_{\xil{l}}}{\partial\bfnc{W}}
\:=\: \frac{\partial\bfnc{F}_{\xil{l}}}{\partial\bfnc{Q}}\frac{\partial\bfnc{Q}}{\bfnc{W}}
\:=\: \mat{Y}\mat{\Lambda}_{\xil{l}}\mat{Y}^{-1}\:\:\mat{Y} \mat{Y}\Tr
\:=\: \mat{Y}\mat{\Lambda}_{\xil{l}}\mat{Y}\Tr
\end{equation*}
A full description of the various required matrices appears in Fisher~\cite{Fisher2012phd}.  

The interface dissipation term:
$\mat{Y}|\mat{\Lambda}|\mat{Y}^{-1}\left(\bfnc{Q}^{+}-\bfnc{Q}^{-}\right)$ is replaced with the equivalent 
quadratic dissipation term: $\mat{Y}|\mat{\Lambda}|\mat{Y}\Tr\left(\bfnc{W}^{+}-\bfnc{W}^{-}\right)$. 
The final form of the interface flux is
\begin{equation}\label{eq:numbfluxentropy}
\bfnc{F}^{*}=\fxmsc{m}{\bfnc{W}^{+}}{\bfnc{W}^{-}}-\frac{1}{2}\mat{Y}\left|\mat{\Lambda}\right|\mat{Y}\Tr\left(\bfnc{W}^{+}-\bfnc{W}^{-}\right).
\end{equation}
and leads to an entropy stable scheme.
The extension to the curvilinear case follows immediately by constructing all three computational fluxes followed by the dot product with the outward facing normal.

Therefore, the dissipation on element $\rmL$ is given as
\begin{equation}\label{eq:dissmodiL}
\begin{split}
\dissL\equiv&
-\frac{1}{2}\left(\M^{\rmL}\right)^{-1}
\RL\Tr\Porthol{1}^{\rmL}\dfdwL\left(\RL\wk{\rmL}-\IHtoL\RH\wk{\rmH}\right)\\
&-\frac{1}{2}\left(\M^{\rmL}\right)^{-1}\RL\Tr\Porthol{1}^{\rmL}\IHtoL\dfdwH\left(\ILtoH\RL\wk{\rmL}-\RH\wk{\rmH}\right),
\end{split}
\end{equation}
and on element $\rmH$ takes the following form
\begin{equation}\label{eq:dissmodiH}
\begin{split}
\dissH\equiv&
-\frac{1}{2}\left(\M^{\rmH}\right)^{-1}\RH\Tr\Porthol{1}^{\rmH}\dfdwH\left(\RH\wk{\rmH}-\ILtoH\RL\wk{\rmL}\right)\\
&-\frac{1}{2}\left(\M^{\rmH}\right)^{-1}\RH\Tr\Porthol{1}^{\rmH}\ILtoH\dfdwL\left(\IHtoL\RH\wk{\rmH}-\RL\wk{\rmL}\right),
\end{split}
\end{equation}
where $\dfdwL$ is constructed from the Roe average of the states $\wk{\rmL}$ and $\IHtoL\wk{\rmH}$, and   
$\dfdwH$ is constructed from the Roe average of the states $\wk{\rmH}$ and $\ILtoH\wk{\rmL}$. Furthermore, 
the various operators which appear in \eqref{eq:dissmodiL} and \eqref{eq:dissmodiH} are defined as follows:
\begin{equation*}
  \begin{split}
  &\RL\equiv\eNl{\rmL}\Tr\otimes\Imat{\Nl{\rmL}}\otimes\Imat{\Nl{\rmL}}\otimes\Imat{5},\quad
  \RH\equiv\eonel{\rmH}\Tr\otimes\Imat{\Nl{\rmH}}\otimes\Imat{\Nl{\rmH}}\otimes\Imat{5},\\
  &\Porthol{l}^{\rmL}\equiv\PoneD_{\rmL}\otimes\PoneD_{\rmL}\otimes\Imat{5},\quad
\Porthol{l}^{\rmH}\equiv\PoneD_{\rmH}\otimes\PoneD_{\rmH}\otimes\Imat{5},\\
&\IHtoL\equiv\IHtoLoneD\otimes\IHtoLoneD\otimes\Imat{5},\quad
\ILtoH\equiv\ILtoHoneD\otimes\ILtoHoneD\otimes\Imat{5}.
  \end{split}
\end{equation*}
HERE1234
All the theorems regarding the accuracy and the entropy stability are presented 
\iftoggle{proofs}{
(element-wise conservation is discussed in Appendix~\ref{sec:gen_element_wise_conservation} while it is straightforward 
to see that free-stream preservation is attained if the discrete GCL are satisfied).
}{(element-wise conservation is discussed in Appendix E in~\cite{Fernandez2018_TM} while it is straightforward 
to see that free-stream preservation is attained if the discrete GCL are satisfied).}
\begin{thrm}
The dissipation term $\dissL$ is a term of order $\pL+d-1$ and the dissipation term 
$\dissH$ is a term of order $\pL+d-1$.
\end{thrm}
\begin{proof}
This follows from the accuracy of the interpolation operators.
\end{proof}
\begin{thrm}
The dissipation terms~\eqref{eq:dissmodiL} and~\eqref{eq:dissmodiH} lead to an entropy stable scheme.
\end{thrm}
\begin{proof}
Contracting the global dissipation operator with the entropy variables results in contractions of 
the interface dissipation terms $\dissL$ and $\dissH$ by
$\wk{\rmL}\Tr\mat{P}^{\rmL}$ and $\wk{\rmH}\Tr\mat{P}^{\rmH}$, respectively.
Summing the interface contributions $\wk{\rmL}\Tr\M^{\rmL}\dissL$ and $\wk{\rmH}\Tr\M^{\rmH}\dissH$, 
using the property $\IHtoL=\left(\M^{\rmL}\right)^{-1}\ILtoH\Tr\M^{\rmH}$ and 
rearranging terms, results in
\begin{equation*}
\begin{split}
&\wk{\rmL}\Tr\M^{\rmL}\dissL+\wk{\rmH}\Tr\M^{\rmH}\dissH =\\ 
&-\frac{1}{2}\left(\RL\wk{\rmL}-\IHtoL\RH\wk{\rmH}\right)\Tr\Porthol{1}^{\rmL}\dfdwL
\left(\RL\wk{rmL}-\IHtoL\RH\wk{rmH}\right)\\
&-\frac{1}{2}\left(\RH\wk{rmH}-\ILtoH\RL\wk{\rmL}\right)\Tr\Porthol{1}^{\rmH}\dfdwH
\left(\RH\wk{\rmH}-\ILtoH\RH\wk{\rmL}\right),
\end{split}
\end{equation*}
which is a negative semi-definite term. Therefore, the added interface terms are entropy dissipative.
\end{proof}

In Section~\ref{sec:num}, two problems are used to characterize the nonconforming algorithms: 1) the propagation of an isentropic vortex 
and 2) the inviscid Taylor-Green vortex problem. For all of them, the boundary conditions are weakly imposed by reusing the interface SAT mechanics (see, for instance, \cite{Parsani2015,dalcin_2019_wall_bc}).
\section{Numerical experiments}\label{sec:num}
This section presents numerical evidence that demonstrates that the proposed $p$ nonconforming algorithm retains 
the accuracy and robustness of the spatial conforming discretization 
reported in \cite{Carpenter2014,Parsani2015,Carpenter2016,Parsani2016}. 

Herein, the conforming \cite{Carpenter2014,Parsani2015,Carpenter2016,Parsani2016} and $p$-adaptive solver for unstructured grids 
developed at the Extreme Computing Research Center (ECRC) at KAUST is used to perform numerical experiments. This parallel solver is built on 
top of the Portable and Extensible Toolkit for Scientific computing (PETSc)~\cite{petsc-user-ref}, its mesh topology 
abstraction (DMPLEX)~\cite{KnepleyKarpeev09} and scalable ordinary differential equation (ODE)/differential algebraic equations (DAE) solver library~\cite{abhyankar2018petsc}. 
The systems of ordinary differential equations arising from the spatial
discretizations are integrated using the fourth-order
accurate Dormand--Prince method \cite{dormand_rk_1980} endowed with an adaptive time stepping technique based on digital signal processing \cite{Soderlind2003,Soderlind2006}. To make the temporal error negligible, a tolerance of $10^{-8}$ is always used for the time-step adaptivity. The two-point entropy consistent flux
of Chandrashekar~\cite{Chandrashekar2013} is used for all the test cases.

The errors are computed using volume scaled (for the $L^{1}$ and $L^{2}$ norms) discrete norms as follows:
\begin{equation*}
\begin{split}
  &\|\bm{u}\|_{L^{1}}=\Omega_{c}^{-1}\sum\limits_{\kappa=1}^{K}\ones{\kappa}\Tr\M^{\kappa}\matJk{\kappa}\textrm{abs}\left(\bm{u}_{\kappa}\right),\,
\|\bm{u}\|_{L^{2}}^{2}=\Omega_{c}^{-1}\sum\limits_{\kappa=1}^{K}\bm{u}_{\kappa}\M^{\kappa}\matJk{\kappa}\bm{u}_{\kappa},\\
&\|\bm{u}\|_{L^{\infty}}=\max\limits_{\kappa=1\dots K}\textrm{abs}\left(\bm{u}_{\kappa}\right),
\end{split}
\end{equation*}
where $\Omega_{c}$ indicates the volume of $\Omega$ computed as 
$\Omega_{c}\equiv\sum\limits_{\kappa=1}^{K}\ones{\kappa}\Tr\M^{\kappa}\matJk{\kappa}\ones{\kappa}$.

\subsection{Isentropic Euler vortex propagation}\label{subsec:iv}
For verification and characterization of the inviscid components of the algorithm, 
the propagation of an isentropic vortex is used. This benchmark problem
has an analytical solution which is given by
\begin{equation*}
\begin{split}
& \fnc{G}\left(\xm{1},\xm{2},\xm{3},t\right) = 1
-\left\{
\left[
\left(\xm{1}-x_{1,0}\right)
-U_{\infty}\cos\left(\alpha\right)t
\right]^{2}
+
\left[
\left(\xm{2}-x_{2,0}\right)
-U_{\infty}\sin\left(\alpha\right)t
\right]^{2}
\right\},\\
&\rho = T^{\frac{1}{\gamma-1}},\\
  &\Um{1} = U_{\infty}\cos(\alpha)-\epsilon_{\nu}
\frac{\left(\xm{2}-x_{2,0}\right)-U_{\infty}\sin\left(\alpha\right)t}{2\pi}
\exp\left(\frac{\fnc{G}}{2}\right),\\
  &\Um{2} = U_{\infty}\sin(\alpha)-\epsilon_{\nu}
\frac{\left(\xm{1}-x_{1,0}\right)-U_{\infty}\cos\left(\alpha\right)t}{2\pi}
\exp\left(\frac{\fnc{G}}{2}\right),\\
  &\Um{3} = 0,\\
&T = \left[1-\epsilon_{\nu}^{2}M_{\infty}^{2}\frac{\gamma-1}{8\pi^{2}}\exp\left(\fnc{G}\right)\right],
\end{split}
\end{equation*}
where $U_{\infty}$, $M_{\infty}$, and $\left(x_{1,0},x_{2,0},x_{3,0}\right)$ are the modulus of the 
freestream velocity, the freestream Mach number, and the vortex center, respectively. In this paper, 
the following values are used: $U_{\infty}=M_{\infty} c_{\infty}$, $\epsilon_{\nu}=5$, $M_{\infty}=0.5$, 
$\gamma=1.4$, $\alpha=45^{\degree}$, and $\left(x_{1,0},x_{2,0},x_{3,0}\right)=\left(0,0,0\right)$.
The computational domain is $\xm{1}\in[-5,5]$, $\xm{2}\in[-5,5]$, $\xm{3}\in[-5,5]$, and $t\in[0,2]$.
The analytical solution is used to furnish data for the initial condition.

First, results aimed at validating the entropy conservation properties of 
the interior domain SBP-SAT algorithm are reported. Thus, periodic boundary conditions are used on all six faces of the computational domain.
Furthermore, all the dissipation terms used for the 
interface coupling are turned off.
The discrete integral over the computational domain of the time rate of change of the entropy function,
$\displaystyle\int_{\Omega}\frac{\partial\fnc{S}}{\partial t}\mr{d}\Omega$
in \Eq~\eqref{eq:NSCe3}, is monitored at every time step.

The computational
domain is subdivided using ten hexahedrons in each coordinate direction and  
the solution polynomial degree in 
each element is assigned a random integer chosen uniformly from the set $p_s=\{2,3,4,5\}$ (\ie, each member in the set has 
an equal probability of being chosen). To test the conservation of entropy and therefore the 
freestream condition when curved element interfaces are used, the LGL collocation 
point coordinates at element interfaces are constructed\footnote{In a general setting, element interfaces can also 
be boundary element interfaces.} as follows:
\begin{itemize}
\item Construct a mesh by describing the element interfaces with a second-order polynomial
  representation. Note that many mesh generators
    use a uniform distribution of points (or nodes) to construct such a geometrical representation.
\item Perturb the nodes that are used to define the second-order polynomial
  approximation of the element interfaces as follows:
\begin{equation*}
\begin{split}
&x_1 = x_{1,*} + \frac{1}{15} L_1 \cos \left(   a \right) \cos \left( 3 b \right) \sin \left( 4 c \right), \,
  x_2 = x_{2,*} + \frac{1}{15} L_2 \sin \left( 4 a \right) \cos \left(   b \right) \cos \left( 3 c \right),\\
&x_3 = x_{3,*} + \frac{1}{15} L_3 \cos \left( 3 a \right) \sin \left( 4 b \right) \cos \left(   c \right),
\end{split}
\end{equation*}
where
\begin{equation*}
\begin{split}
  &a = \frac{\pi}{L_1} \left(x_{1,*}-\frac{x_{1,H}+x_{1,L}}{2}\right), 
  b = \frac{\pi}{L_2} \left(x_{2,*}-\frac{x_{2,H}+x_{2,L}}{2}\right),\\
  &c = \frac{\pi}{L_3} \left(x_{3,*}-\frac{x_{3,H}+x_{3,L}}{2}\right).
\end{split}
\end{equation*}
The lengths $L_1$, $L_2$ and $L_3$ are the dimensions of the computational domain in the three
coordinate directions and the sub-script $*$ indicates the unperturbed coordinates of
the nodes. This step yields a ``perturbed'' second-order interface polynomial representation.
\item Compute the coordinate of the LGL points at the element interface
by evaluating the ``perturbed'' second-order polynomial at the tensor-product LGL points used 
    to define the cell solution polynomial of order $p_s$.
\end{itemize}

Figure \ref{fig:iv_mesh_cut} shows a cut of the mesh
and the polynomial order distribution used for this first test case. The propagation
of the vortex is simulated for two time units.
\begin{figure}[!tbp]
  \centering
  \begin{minipage}[b]{0.45\textwidth}
    \includegraphics[width=\textwidth]{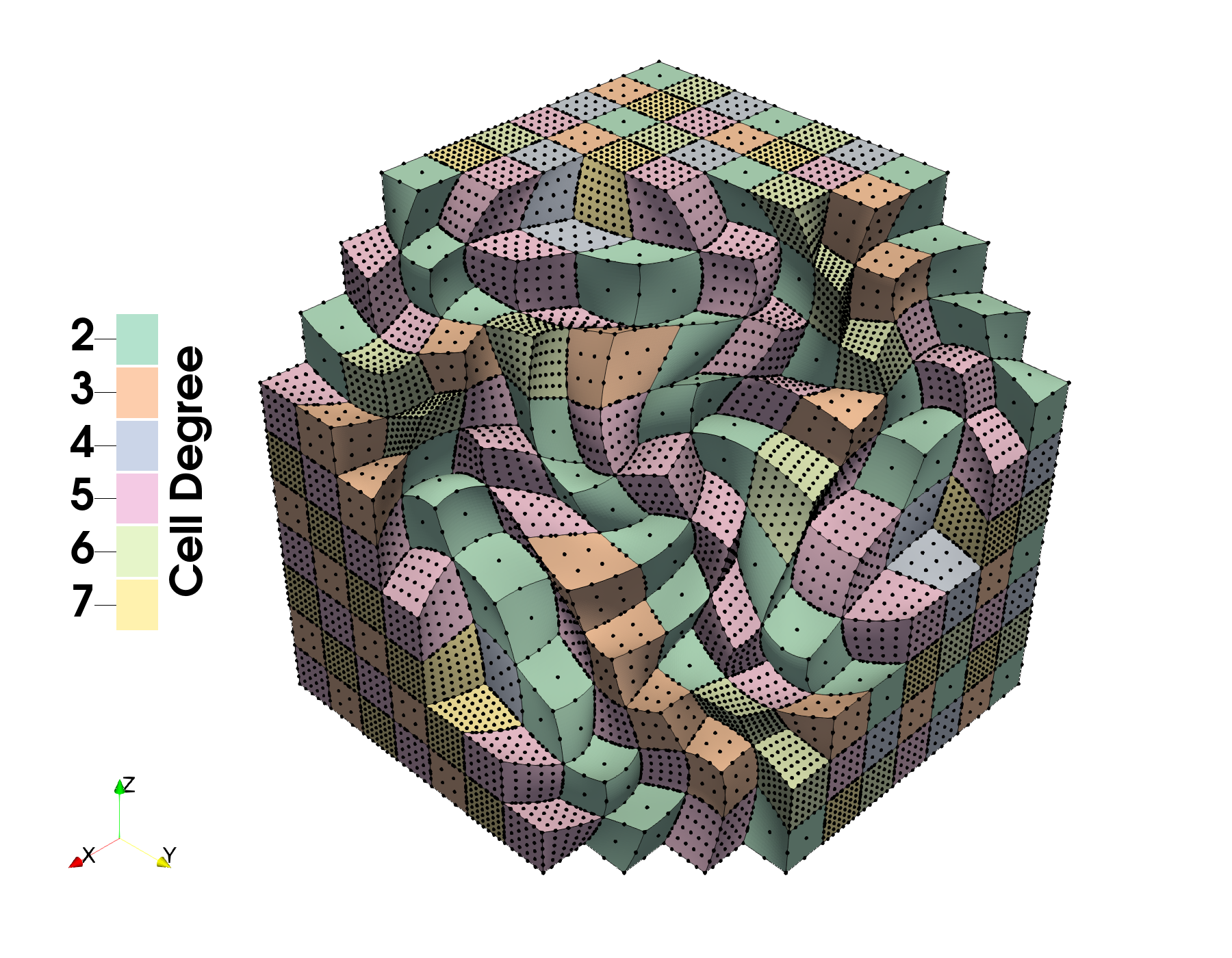}
    \caption{Isentropic vortex: mesh cut and polynomial order distribution.}\label{fig:iv_mesh_cut}
  \end{minipage}
  \hfill
  \begin{minipage}[b]{0.45\textwidth}
    \includegraphics[width=\textwidth]{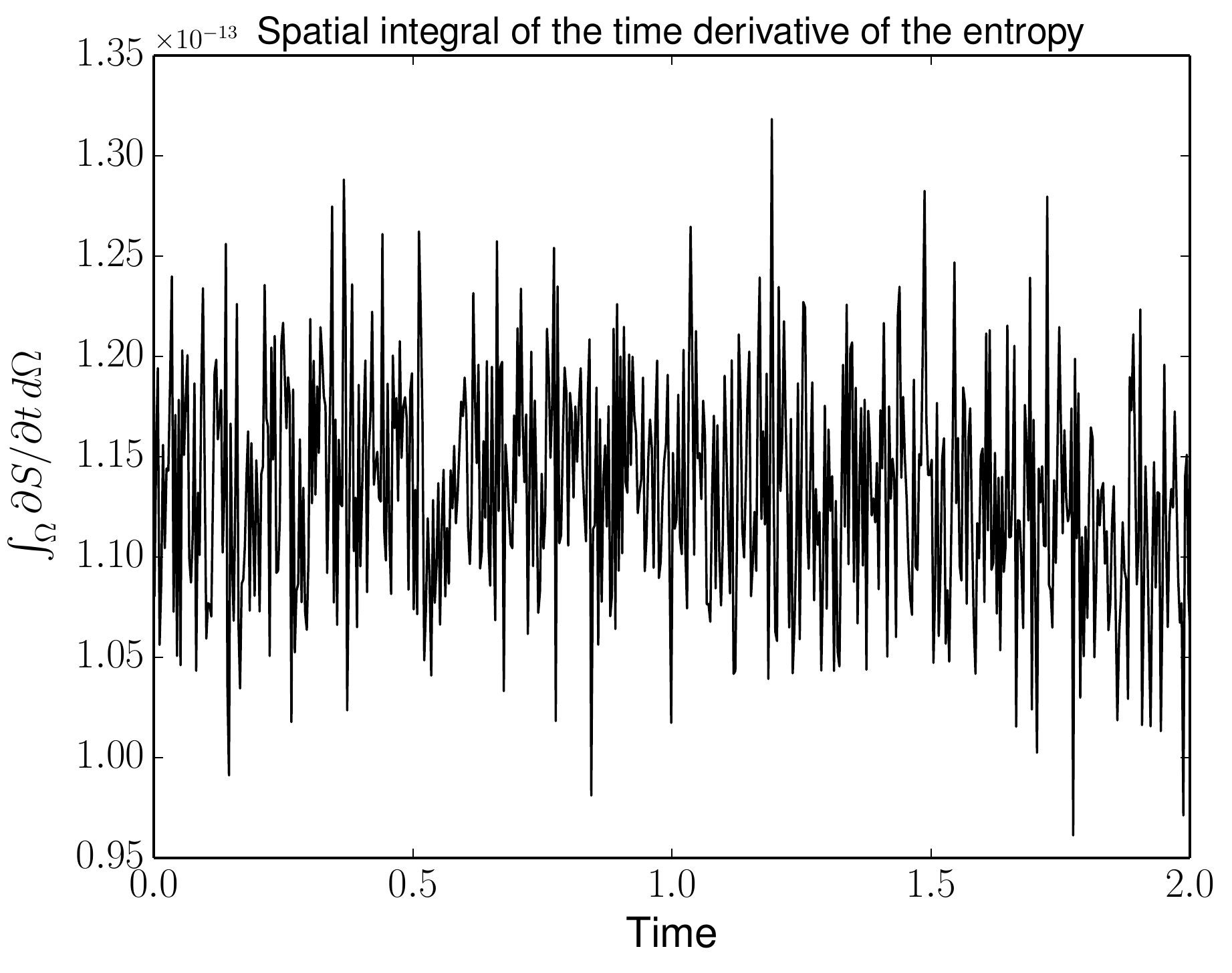}
    \caption{Isentropic vortex: time rate of change of the entropy function.}\label{fig:iv_dsdt}
  \end{minipage}
\end{figure}
Figure \ref{fig:iv_dsdt} plots the integral of the time derivative of the entropy function. It is 
seen that the global variation of the time rate of the of $\fnc{S}$ is practically
zero. This implies that the spatially nonconforming algorithm is entropy conservative.

Second, a grid convergence study is performed to investigate the order of convergence
of the nonconforming algorithm.
The base grid (i.e., the coarsest grid) is constructed as follows:
\begin{itemize}
\item Divide the computational domain with four hexahedral elements in each coordinate direction.
\item Assign the solution polynomial degree in each element to a random integer chosen uniformly from the set $\{p_s, p_s+1\}$.
\item Approximate with a $p_s$th-order accurate polynomial the element interfaces.
\item Construct the perturbed elements and their corresponding LGL points as described previously. 
\end{itemize}
From the base grid, which is similar to the one depicted in Figure \ref{fig:iv_mesh_cut}, 
a sequence of nested grids is then generated to perform the convergence study.
The results are reported in Tables \ref{tab:iv_p1p2} through \ref{tab:iv_p6p7} for the error on the
density. The number
listed in the first column denoted by``Grid'' indicates the number of hexahedrons in each coordinate direction.
\begin{table}[htbp!]
\vspace{0.5cm}
\begin{center}
\scriptsize
\begin{tabular}{||l||c|c|c|c|c|c||c|c|c|c|c|c||}
\hline \hline
        & \multicolumn{6}{c||}{Conforming, $p=1$} & \multicolumn{6}{c||}{Nonconforming, $p=1$ and $p=2$}\\ \hline
 Grid & $L^1$    & Rate  & $L^2$    & Rate  & $L^{\infty}$ & Rate  & $L^1$    & Rate  & $L^2$    & Rate  & $L^{\infty}$ & Rate  \\ \hline
 4    & 2.74E-02 & -     & 1.32E-03 & -     &  1.55E-01    & -     & 1.87E-02 & -     & 9.48E-04 & -     & 1.20E-01     & -     \\ \hline
 8    & 1.14E-02 & -1.26 & 6.61E-04 & -1.00 &  1.12E-01    & -0.47 & 1.02E-02 & -0.87 & 5.87E-04 & -0.69 & 1.23E-01     & +0.04 \\ \hline
 16   & 5.13E-03 & -1.16 & 3.31E-04 & -1.00 &  7.29E-02    & -0.62 & 4.55E-03 & -1.17 & 2.99E-04 & -0.97 & 7.06E-02     & -0.80 \\ \hline
 32   & 1.70E-03 & -1.59 & 1.15E-04 & -1.52 &  3.01E-02    & -1.28 & 1.52E-03 & -1.58 & 1.07E-04 & -1.49 & 3.97E-02     & -0.83 \\ \hline
 64   & 4.76E-04 & -1.84 & 3.24E-05 & -1.83 &  8.53E-03    & -1.82 & 4.36E-04 & -1.81 & 3.15E-05 & -1.76 & 2.15E-02     & -0.88 \\ \hline
 128  & 1.23E-04 & -1.96 & 8.33E-06 & -1.96 &  2.13E-03    & -2.00 & 1.14E-04 & -1.93 & 8.72E-06 & -1.85 & 1.05E-02     & -1.03 \\ \hline
 256  & 3.08E-05 & -1.99 & 2.09E-06 & -1.99 &  5.24E-04    & -2.02 & 2.91E-05 & -1.98 & 2.46E-06 & -1.83 & 5.10E-03     & -1.05 \\ \hline
 512  & 7.68E-06 & -2.00 & 5.23E-07 & -2.00 &  1.30E-04    & -2.01 & 7.23E-06 & -2.01 & 6.17E-07 & -2.00 & 2.26E-03     & -1.17 \\ \hline
\end{tabular}
\caption{Convergence study of the isentropic vortex propagation: $p=1$ with $p=2$; density error.}.
\label{tab:iv_p1p2}
\end{center}
\end{table}

\begin{table}[htbp!]
\vspace{0.5cm}
\begin{center}
\scriptsize
\begin{tabular}{||l||c|c|c|c|c|c||c|c|c|c|c|c||}
\hline \hline
      & \multicolumn{6}{c||}{Conforming, $p=2$} & \multicolumn{6}{c||}{Nonconforming, $p=2$ and $p=3$}\\ \hline
 Grid & $L^1$    & Rate  & $L^2$    & Rate  & $L^{\infty}$ & Rate  & $L^1$    & Rate  & $L^2$    & Rate  & $L^{\infty}$ & Rate  \\ \hline
 4    & 9.75E-03 & -     & 5.32E-04 & -     & 1.24E-01     & -     & 8.57E-03 & -     & 5.05E-04 & -     & 1.21E-01     & -     \\ \hline
 8    & 3.18E-03 & -1.61 & 2.09E-04 & -1.35 & 6.97E-02     & -0.83 & 2.70E-03 & -1.67 & 1.84E-04 & -1.46 & 8.55E-02     & -0.50 \\ \hline
 16   & 5.18E-04 & -2.62 & 3.88E-05 & -2.43 & 2.51E-02     & -1.47 & 4.25E-04 & -2.67 & 3.58E-05 & -2.36 & 3.23E-02     & -1.40 \\ \hline
 32   & 6.38E-05 & -3.02 & 5.50E-06 & -2.82 & 7.23E-03     & -1.79 & 4.99E-05 & -3.09 & 5.14E-06 & -2.80 & 7.20E-03     & -2.17 \\ \hline
 64   & 7.61E-06 & -3.07 & 7.23E-07 & -2.93 & 1.21E-03     & -2.58 & 5.78E-06 & -3.11 & 7.05E-07 & -2.87 & 1.42E-03     & -2.35 \\ \hline
 128  & 9.48E-07 & -3.00 & 9.95E-08 & -2.86 & 2.75E-04     & -2.14 & 7.14E-07 & -3.02 & 1.06E-07 & -2.74 & 3.52E-04     & -2.01 \\ \hline
 256  & 1.23E-07 & -2.94 & 1.43E-08 & -2.83 & 3.41E-05     & -3.01 & 9.46E-08 & -2.92 & 1.52E-08 & -2.80 & 9.16E-05     & -1.94 \\ \hline
\end{tabular}
  \caption{Convergence study of the isentropic vortex propagation: $p=2$ with $p=3$; density error.}.
\label{tab:iv_p2p3}
\end{center}
\end{table}

\begin{table}[htbp!]
\vspace{0.5cm}
\begin{center}
\scriptsize
\begin{tabular}{||l||c|c|c|c|c|c||c|c|c|c|c|c||}
\hline \hline
      & \multicolumn{6}{c||}{Conforming, $p=3$} & \multicolumn{6}{c||}{Nonconforming, $p=3$ and $p=4$}\\ \hline
 Grid & $L^1$    & Rate  & $L^2$    & Rate  & $L^{\infty}$ & Rate  & $L^1$    & Rate  & $L^2$    & Rate  & $L^{\infty}$ & Rate  \\ \hline
 4    & 5.22E-03 & -     & 3.38E-04 & -     & 9.16E-02     & -     & 4.58E-03 & -     & 3.07E-04 & -     & 8.90E-02     & -     \\ \hline
 8    & 6.84E-04 & -2.93 & 5.30E-05 & -2.67 & 4.71E-02     & -0.96 & 5.81E-04 & -2.98 & 4.94E-05 & -2.63 & 5.26E-02     & -0.76 \\ \hline
 16   & 5.50E-05 & -3.64 & 4.54E-06 & -3.55 & 6.61E-03     & -2.83 & 4.51E-05 & -3.69 & 4.15E-06 & -3.58 & 5.04E-03     & -3.38 \\ \hline
 32   & 3.48E-06 & -3.98 & 3.33E-07 & -3.77 & 5.47E-04     & -3.59 & 2.88E-06 & -3.97 & 3.06E-07 & -3.76 & 4.98E-04     & -3.34 \\ \hline
 64   & 2.10E-07 & -4.05 & 2.45E-08 & -3.76 & 4.93E-05     & -3.47 & 1.78E-07 & -4.01 & 2.34E-08 & -3.71 & 5.42E-05     & -3.20 \\ \hline
 128  & 1.39E-08 & -3.92 & 1.87E-09 & -3.71 & 6.32E-06     & -2.96 & 1.24E-08 & -3.85 & 1.90E-09 & -3.63 & 7.96E-06     & -2.77 \\ \hline
\end{tabular}
  \caption{Convergence study of the isentropic vortex propagation: $p=3$ with $p=4$; density error.}.
\label{tab:iv_p3p4}
\end{center}
\end{table}

\begin{table}[htbp!]
\vspace{0.5cm}
\begin{center}
\scriptsize
\begin{tabular}{||l||c|c|c|c|c|c||c|c|c|c|c|c||}
\hline \hline
      & \multicolumn{6}{c||}{Conforming, $p=4$} & \multicolumn{6}{c||}{Nonconforming, $p=4$ and $p=5$}\\ \hline
 Grid & $L^1$    & Rate  & $L^2$    & Rate  & $L^{\infty}$ & Rate  & $L^1$    & Rate  & $L^2$    & Rate  & $L^{\infty}$ & Rate  \\ \hline
 4    & 2.34E-03 & -     & 1.56E-04 & -     & 9.92E-02     & -     & 2.10E-03 & -     & 1.48E-04 & -     & 9.28E-02     & -     \\ \hline
 8    & 1.70E-04 & -3.78 & 1.43E-05 & -3.45 & 2.32E-02     & -2.09 & 1.48E-04 & -3.82 & 1.37E-05 & -3.43 & 2.30E-02     & -2.01 \\ \hline
 16   & 6.07E-06 & -4.81 & 6.41E-07 & -4.48 & 1.27E-03     & -4.19 & 5.17E-06 & -4.85 & 6.00E-07 & -4.51 & 1.26E-03     & -4.19 \\ \hline
 32   & 1.99E-07 & -4.93 & 2.25E-08 & -4.83 & 5.13E-05     & -4.64 & 1.64E-07 & -4.98 & 2.07E-08 & -4.86 & 5.50E-05     & -4.52 \\ \hline
 64   & 7.11E-09 & -4.81 & 8.60E-10 & -4.71 & 3.01E-06     & -4.09 & 6.36E-09 & -4.69 & 7.95E-10 & -4.70 & 2.99E-06     & -4.20 \\ \hline
\end{tabular}
  \caption{Convergence study of the isentropic vortex propagation: $p=4$ with $p=5$; density error.}.
\label{tab:iv_p4p5}
\end{center}
\end{table}

\begin{table}[htbp!]
\vspace{0.5cm}
\begin{center}
\scriptsize
\begin{tabular}{||l||c|c|c|c|c|c||c|c|c|c|c|c||}
\hline \hline
      & \multicolumn{6}{c||}{Conforming, $p=5$} & \multicolumn{6}{c||}{Nonconforming, $p=5$ and $p=6$}\\ \hline
 Grid & $L^1$    & Rate  & $L^2$    & Rate  & $L^{\infty}$ & Rate  & $L^1$    & Rate  & $L^2$    & Rate  & $L^{\infty}$ & Rate  \\ \hline
 4    & 1.01E-03 & -     & 6.88E-05 & -     & 4.53E-02     & -     & 9.14E-04 & -     & 6.63E-05 & -     & 4.45E-02     & -     \\ \hline
 8    & 3.81E-05 & -4.73 & 3.80E-06 & -4.18 & 6.37E-03     & -2.83 & 3.26E-05 & -4.81 & 3.61E-06 & -4.20 & 6.44E-03     & -2.79 \\ \hline
 16   & 7.23E-07 & -5.72 & 7.95E-08 & -5.58 & 2.23E-04     & -4.84 & 5.93E-07 & -5.78 & 7.30E-08 & -5.63 & 1.63E-04     & -5.30 \\ \hline
 32   & 1.24E-08 & -5.87 & 1.59E-09 & -5.65 & 4.99E-06     & -5.48 & 1.05E-08 & -5.82 & 1.46E-09 & -5.65 & 4.99E-06     & -5.03 \\ \hline
\end{tabular}
  \caption{Convergence study of the isentropic vortex propagation: $p=5$ with $p=6$; density error.}.
\label{tab:iv_p5p6}
\end{center}
\end{table}

\begin{table}[htbp!]
\vspace{0.5cm}
\begin{center}
\scriptsize
\begin{tabular}{||l||c|c|c|c|c|c||c|c|c|c|c|c||}
\hline \hline
      & \multicolumn{6}{c||}{Conforming, $p=6$} & \multicolumn{6}{c||}{Nonconforming, $p=6$ and $p=7$}\\ \hline
 Grid & $L^1$    & Rate  & $L^2$    & Rate  & $L^{\infty}$ & Rate  & $L^1$    & Rate  & $L^2$    & Rate  & $L^{\infty}$ & Rate  \\ \hline
 4    & 4.28E-04 & -     & 3.14E-05 & -     & 2.17E-02     & -     & 3.84E-04 & -     & 3.03E-05 & -     & 2.35E-02     & -     \\ \hline
 8    & 8.58E-06 & -5.64 & 9.16E-07 & -5.10 & 1.57E-03     & -3.78 & 7.27E-06 & -5.72 & 8.72E-07 & -5.12 & 1.58E-03     & -3.89 \\ \hline
 16   & 7.97E-08 & -6.75 & 1.02E-08 & -6.49 & 2.20E-05     & -6.16 & 6.58E-08 & -6.79 & 9.41E-09 & -6.53 & 2.37E-05     & -6.06 \\ \hline
 32   & 5.95E-10 & -7.07 & 8.17E-11 & -6.96 & 2.55E-07     & -6.43 & 6.05E-10 & -6.76 & 7.71E-11 & -6.93 & 3.55E-07     & -6.06  \\ \hline
\end{tabular}
  \caption{Convergence study of the isentropic vortex propagation: $p=6$ with $p=7$; density error.}.
\label{tab:iv_p6p7}
\end{center}
\end{table}
For all the degrees tested (i.e., $p=1$ to $p=7$), the order of convergence of the conforming algorithm is very close to that of the nonconforming. 
However, note that in the $L^{1}$ and 
$L^{2}$ norms the nonconforming algorithm is more accurate than the conforming one. 
In the discrete $L^{\infty}$ norm, the nonconforming scheme 
is sometimes slightly worse than the conforming scheme; this results from the interpolation
matrices being sub-optimal at nonconforming interfaces.

\subsection{Inviscid Taylor--Green vortex propagation}\label{subsec:itgv}
The purpose of this section is to demonstrate that robustness of the nonconforming entropy 
stable algorithm. To do so, 
the inviscid Taylor--Green vortex problem is solved using a coarse grid with four elements in
each coordinate direction. 

The test case is solved on a periodic cube $[-\pi L\leq x,y,z\leq \pi L]$, where the initial condition 
is given by the initial condition used for the simulation of the standard (viscous)
Taylor--Green problem:
\begin{equation*}\label{eq:TG}
\begin{split}
&\fnc{U}_{1} = \fnc{V}_{0}\sin\left(\frac{x_{1}}{L}\right)\cos\left(\frac{x_{2}}{L}\right)\cos\left(\frac{x_{3}}{L}\right), \,
\fnc{U}_{2} = -\fnc{V}_{0}\cos\left(\frac{x_{1}}{L}\right)\sin\left(\frac{x_{2}}{L}\right)\cos\left(\frac{x_{3}}{L}\right),\\
&\fnc{U}_{3} = 0, \, 
\fnc{P} = \fnc{P}_{0}+\frac{\rho_{0}\fnc{V}_{0}^{2}}{16}\left[\cos\left(\frac{2x_{1}}{L}+\cos\left(\frac{2x_{2}}{L}\right)\right)\right]
\left[\cos\left(\frac{2x_{3}}{L}+2\right)\right].
\end{split}
\end{equation*}
The flow is initialized using $\fnc{P}/\rho=\fnc{P}_{0}/\rho_{0}=R\fnc{T}_{0}$, and $\fnc{P}_{0}=1$, $\fnc{T}_{0}=1$, $L=1$, and $\fnc{V}_{0} = 1$.   
In order to obtain results that are reasonably close to those found for the incompressible 
equations, a Mach number of $M = 0.05$ is used.

The problem is simulated for twenty time units using an unstructured grid constructed as follows:
\begin{itemize}
\item Divide the computational domain with four hexahedral elements in each coordinate direction.
\item Assign the solution polynomial degree in each element to a random integer chosen uniformly from a set $\{p_{min}, p_{max}\}$.
\item Approximate the element interfaces with a $p_{min}$th accurate polynomial.
\item Construct the perturbed elements and their corresponding LGL points as described in Section \ref{subsec:iv}. 
\end{itemize}

Figure \ref{fig:dkedt_itg} shows the time rate of change of the kinetic energy, $dke/dt$, for
the nonconforming algorithm using a random distribution of solution polynomial
order between i) $p_{min}=2$ and $p_{max}=7$ and ii) $p_{min}=2$ and $p_{max}=10$.
\begin{figure}[htbp!]
   \centering
   \includegraphics[width=0.85\textwidth]{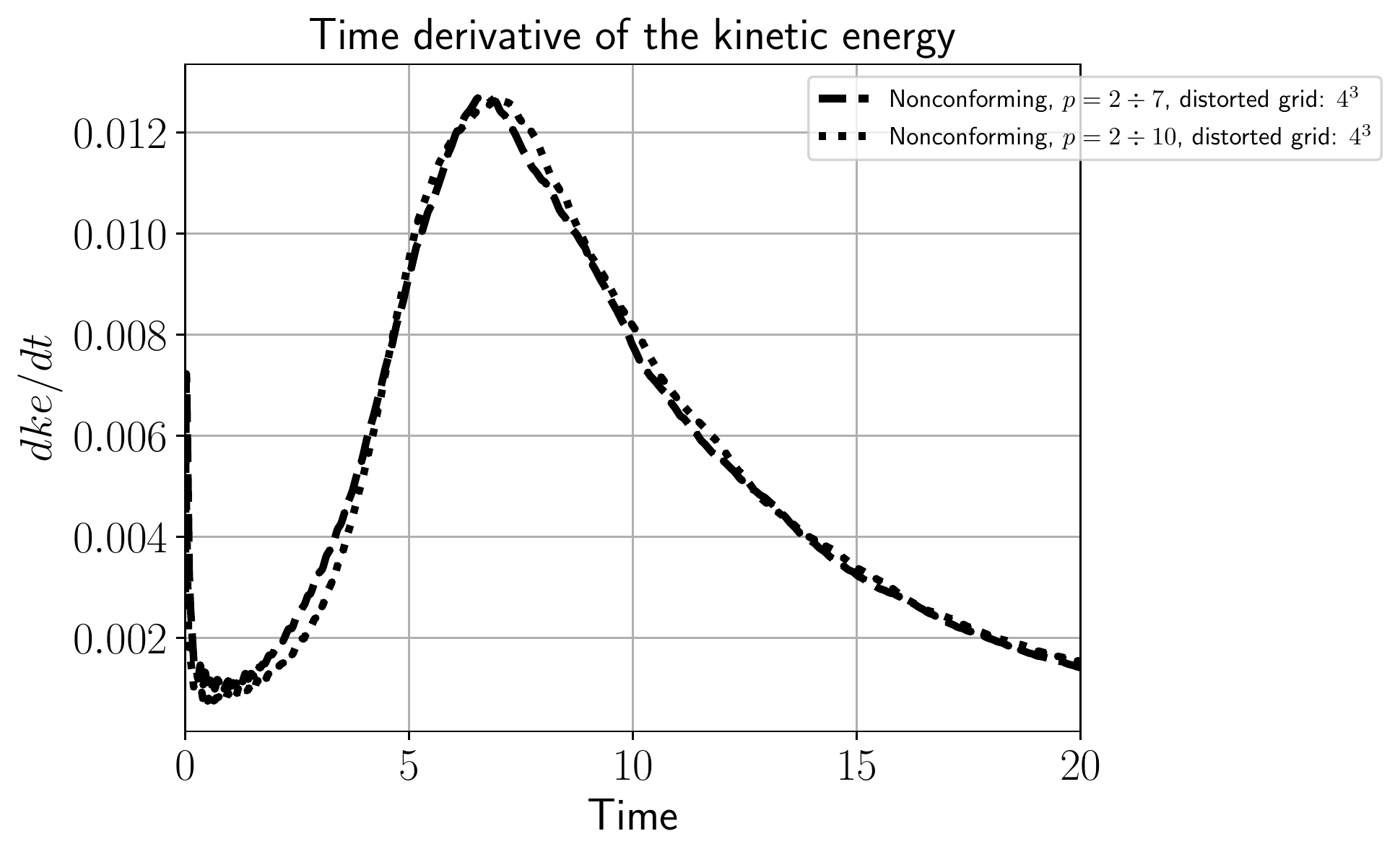}
   \caption{Evolution of the time derivative of the kinetic energy 
   for the inviscid Taylor--Green vortex at $M = 0.05$.}
   \label{fig:dkedt_itg}
\end{figure}
The main 
take-away from the figure is that all simulations are stable. This 
is numerical evidence that the $p$-nonconforming 
scheme inherits the stability characteristics of the conforming 
and fully staggered algorithms \cite{Carpenter2014,Parsani2015,Carpenter2016b,Carpenter2016,Parsani2016}.
\section{Conclusions}\label{sec:Conclusion}
In this paper, the entropy conservative $p-$refinement/coarsening nonconforming algorithm of 
Friedrich~\etal~\cite{Friedrich2018} is extended to curvilinear coordinates applicable to the 
compressible Euler equations. The coupling between nonconforming interfaces 
is achieved using SBP preserving interpolation operators and prescribed metric terms. To maintain entropy 
conservation/stability, elementwise conservation, and the order of the parent element, the procedure of 
Crean~\etal~\cite{Crean2018} is used to approximate grid metrics. 
Finally, the accuracy and stability characteristics of the resulting numerical schemes are shown to be 
comparable to those of the original conforming scheme \cite{Carpenter2014,Parsani2015}, in the context of two canonical tests.

\section*{Acknowledgments}
Special thanks are extended to Dr. Mujeeb R. Malik for partially funding this work as part of
NASA's ``Transformational Tools and Technologies'' ($T^3$) project. The research reported in this publication was
also supported by funding from King Abdullah University of Science and Technology (KAUST).
We are thankful for the computing resources of the Supercomputing Laboratory and the Extreme Computing Research Center at KAUST. Gregor Gassner and Lucas Friedrich has been supported by the European Research Council (ERC) under the European Union’s Eights Framework Program Horizon 2020 with the research project Extreme, ERC grant agreement no. 714487.

\bibliographystyle{siamplain}
\bibliography{Bib}

\iftoggle{proofs}{
  \appendix
\section{Generalized notation and nonlinear analysis}\label{app:nonlinear}
The various proofs presented in these appendices are greatly simplified (and also generalized) by using 
notation appropriate for multidimensional SBP operators. In this section, the notation that will be 
used is presented and linked to the tensor-product notation used in the paper. Furthermore, the nonlinear approximation 
used to construct entropy conservative discretizations will be analyzed in detail.

\subsection{Generalized notation}
The derivative 
operators that are used are constructed using tensor-products. For a system of $e$ PDEs they read 
\begin{equation*}
  \begin{split}
  &\Dxil{1}\equiv\barDxil{1}\otimes\Imat{e},\quad\barDxil{1}\equiv\DxiloneD{1}\otimes\Imat{\Nl{2}}\otimes\Imat{\Nl{3}},\\
  &\Dxil{2}\equiv\barDxil{2}\otimes\Imat{e},\quad\barDxil{2}\equiv\Imat{\Nl{1}}\otimes\DxiloneD{2}\otimes\Imat{\Nl{3}},\\
  &\Dxil{3}\equiv\barDxil{3}\otimes\Imat{e},\quad\barDxil{3}\equiv\Imat{\Nl{1}}\otimes\Imat{\Nl{2}}\otimes\DxiloneD{3}.
  \end{split}
\end{equation*}
The above tensor-product SBP operators can be recast as multidimensional SBP operators as follows:
\begin{equation*}
  \begin{split}
  &\Dxil{1}=\M^{-1}\Qxil{1},\quad\Qxil{1}\equiv\barQxil{1}\otimes\Imat{e},\quad \barQxil{1}\equiv\QxiloneD{1}\otimes\PxiloneD{2}\otimes\PxiloneD{3},\\
  &\Dxil{2}=\M^{-1}\Qxil{2},\quad\Qxil{2}\equiv\barQxil{2}\otimes\Imat{e},\quad \barQxil{2}\equiv\PxiloneD{1}\otimes\QxiloneD{2}\otimes\PxiloneD{3},\\
  &\Dxil{3}=\M^{-1}\Qxil{3},\quad\Qxil{3}\equiv\barQxil{3}\otimes\Imat{e},\quad \barQxil{3}\equiv\PxiloneD{1}\otimes\PxiloneD{2}\otimes\QxiloneD{3},\\
  &\M\equiv\barM\otimes\Imat{e},\quad\barM\equiv\PxiloneD{1}\otimes\PxiloneD{2}\otimes\PxiloneD{3}.
  \end{split}
\end{equation*}

The resulting recast of the tensor-product SBP operators as multidimensional SBP operators respects the SBP property, \ie, $\Qxil{l}+\Qxil{l}\Tr=\Exil{l}$, 
where the surface matrices $\Exil{l}$ are given as
\begin{equation*}
  \begin{split}
  &\Exil{1}=\barExil{1}\otimes\Imat{e},\quad\barExil{1}\equiv\left(\eNl{1}\eNl{1}\Tr-\eonel{1}\eonel{1}\Tr\right)\otimes\PxiloneD{2}\otimes\PxiloneD{3},\\
  &\Exil{2}=\barExil{2}\otimes\Imat{e},\quad\barExil{2}\equiv\PxiloneD{1}\otimes\left(\eNl{2}\eNl{2}\Tr-\eonel{2}\eonel{2}\Tr\right)\otimes\PxiloneD{3},\\
  &\Exil{3}=\barExil{3}\otimes\Imat{e},\quad\barExil{3}\equiv\PxiloneD{1}\otimes\PxiloneD{2}\otimes\left(\eNl{3}\eNl{3}\Tr-\eonel{3}\eonel{3}\Tr\right).
  \end{split}
\end{equation*}
In precisely the same way that the one-dimensional $\mat{E}$ can be decomposed into separate surface contributions, 
\ie, $\mat{E} = \eNl{}\eNl{}\Tr-\eonel{}\eonel{}\Tr$, 
the multi-dimensional $\mat{E}$ given above can be decomposed into contributions to the faces of the hexahedral element. Thus, $\Exil{1}$ is decomposed as 
\begin{equation*}
  \begin{split}
  &\Exil{1}=\Ebetal{1}-\Ealphal{1},\quad,\Ebetal{1}\equiv\Rbetal{1}\Tr\Porthol{1}\Rbetal{1},\quad\Ealphal{1}\equiv\Ralphal{1}\Tr\Porthol{1}\Ralphal{1},\\
  &\Rbetal{1}\equiv\barRbetal{1}\otimes\Imat{e}\quad\barRbetal{1}\equiv\eonel{\Nl{1}}\Tr\otimes\Imat{\Nl{2}}\otimes\Imat{\Nl{3}},\\
  &\Ralphal{1}\equiv\barRalphal{1}\otimes\Imat{e}\quad\barRalphal{1}\equiv\eonel{1}\Tr\otimes\Imat{\Nl{2}}\otimes\Imat{\Nl{3}},\\
  &\Porthol{1}\equiv\barPorthol{1}\otimes\Imat{e},\quad\barPorthol{1}\equiv\PxiloneD{2}\otimes\PxiloneD{3},
  \end{split}
\end{equation*}
$\Exil{2}$ is decomposed as 
\begin{equation*}
  \begin{split}
  &\Exil{2}=\Ebetal{2}-\Ealphal{2},\quad,\Ebetal{2}\equiv\Rbetal{2}\Tr\Porthol{2}\Rbetal{2},\quad\Ealphal{2}\equiv\Ralphal{2}\Tr\Porthol{2}\Ralphal{2},\\
  &\Rbetal{2}\equiv\barRbetal{2}\otimes\Imat{e}\quad\barRbetal{2}\equiv\Imat{\Nl{1}}\otimes\eonel{\Nl{2}}\Tr\otimes\Imat{\Nl{3}},\\
  &\Ralphal{2}\equiv\barRalphal{2}\otimes\Imat{e}\quad\barRalphal{2}\equiv\Imat{\Nl{1}}\otimes\eonel{1}\Tr\otimes\Imat{\Nl{3}},\\
  &\Porthol{2}\equiv\barPorthol{2}\otimes\Imat{e},\quad\barPorthol{2}\equiv\PxiloneD{1}\otimes\PxiloneD{3},
  \end{split}
\end{equation*}
 and $\Exil{3}$ is decomposed as 
\begin{equation*}
  \begin{split}
  &\Exil{3}=\Ebetal{3}-\Ealphal{3},\quad,\Ebetal{2}\equiv\Rbetal{3}\Tr\Porthol{3}\Rbetal{3},\quad\Ealphal{3}\equiv\Ralphal{3}\Tr\Porthol{3}\Ralphal{3},\\
  &\Rbetal{3}\equiv\barRbetal{3}\otimes\Imat{e}\quad\barRbetal{3}\equiv\Imat{\Nl{1}}\otimes\Imat{\Nl{2}}\otimes\eonel{\Nl{3}}\Tr,\\
  &\Ralphal{3}\equiv\barRalphal{3}\otimes\Imat{e}\quad\barRalphal{3}\equiv\Imat{\Nl{1}}\otimes\Imat{\Nl{2}}\otimes\eonel{3}\Tr,\\
  &\Porthol{3}\equiv\barPorthol{3}\otimes\Imat{e},\quad\barPorthol{3}\equiv\PxiloneD{1}\otimes\PxiloneD{2}.
  \end{split}
\end{equation*}
To give further insight into the properties of the above operators, a connection to various bilinear forms is given:
\begin{equation*}
  \begin{array}{ll}
  \bm{v}\Tr\barM\bm{u}\equiv\int_{\Ohat}\fnc{V}\fnc{U}\mr{d}\Ohat,&\bm{v}\Tr\barQxil{l}\bm{u}\equiv\int_{\Ohat}\fnc{V}\frac{\partial\fnc{U}}{\partial \xil{l}}\mr{d}\Ohat,\\\\
  \bm{v}\Tr\barExil{l}\bm{u}\approx\oint_{\Ghat}\fnc{V}\fnc{U}\nxil{l}\mr{d}\Ghat,\\\\
  \bm{v}\Tr\barEbetal{l}\bm{u}\approx\oint_{\Ghat_{\betal{l}}}\fnc{V}\fnc{U}\nxil{l}\mr{d}\Ghat,&\bm{v}\Tr\barEalphal{l}\bm{u}\approx\oint_{\Ghat_{\alphal{l}}}\fnc{V}\fnc{U}\nxil{l}\mr{d}\Ghat,
  \end{array}
\end{equation*}
 where $\Ghat_{\betal{l}}$ is the surface of the hexahedral element where $\xil{l}$ is at a maximum and $\Ghat_{\alphal{l}}$ is the surface where $\xil{l}$ is at a minimum. 
 Furthermore, the operators $\barRbetal{l}$ and $\barRalphal{l}$ interpolate to the nodes of the $\Ghat_{\betal{l}}$ and $\Ghat_{\alphal{l}}$ surfaces, respectively.

\subsection{Analysis of the nonlinear discretization}
SBP operators are constructed so that the continuous stability proofs can be mimicked at the semi-discrete and fully-discrete levels. 
The critical property that is needed is that, like the continuous analysis, the spatial (and temporal in the fully discrete case) terms 
telescope to the boundaries. Then if appropriate numerical boundary closures can be found, discrete stability statements can be constructed. 
Often times this is linked to the schemes mimicking integration by parts (this is what is used at the continuous level). However,
any combination of differentiation matrix $\mat{D}$ and norm matrix $\mat{P}$ is mimetic of integration by parts. It is the fact 
that SBP operators result into (node-wise) separable approximations to surface integrals that allows stability estimates to be constructed. 

The telescoping notion is best explained by carefully examining how SBP operators mimic integration by parts. In multiple dimensions, 
integration by parts on a hexahedral element is given as 
\begin{equation}\label{eqapp:IBP}
  \oint_{\Ohat}\left(\fnc{V}\frac{\partial\fnc{U}}{\partial\xil{l}}+\fnc{V}\frac{\partial\fnc{U}}{\partial\xil{l}}\right)\mr{d}\Ohat=
  \oint_{\Ghat}\left(\fnc{V}\fnc{U}\nxil{l}\right) = \oint_{\Ghat_{\beta{l}}}\left(\fnc{V}\fnc{U}\nxil{l}\right)+\oint_{\Ghat_{\alpha{l}}}\left(\fnc{V}\fnc{U}\nxil{l}\right).
\end{equation}

Discretizing the left-hand side of~\eqref{eqapp:IBP} using SBP operators and their properties results in the following equality:
\begin{equation}\label{eqapp:SBP1}
  \begin{split}
  \bm{v}\Tr\barM\barDxil{l}\bm{u}+\bm{u}\barM\barDxil{l}\bm{v}=\bm{v}\Tr\barExil{l}\bm{u}.
  \end{split}
\end{equation}
Each term in~\eqref{eqapp:SBP1} is a high-order approximation to the analogous term in integration by parts formula~\eqref{eqapp:IBP}. Notice, though, that 
this would also occur for any combination of a high-order norm matrix that approximates the inner product and a high-order derivative operator.

However, the SBP operator has an $\mat{E}$ matrix that can be further decomposed into separate contributions from opposing surfaces (this property does not in general hold 
for arbitrary combinations of $\mat{P}$ and $\mat{D}$), namely

\begin{equation}\label{eqapp:SBP2}
  \begin{split}
  \bm{v}\Tr\barM\barDxil{l}\bm{u}+\bm{u}\barM\barDxil{l}\bm{v}=
  \bm{v}\Tr\barEbetal{l}\bm{u}-\bm{v}\Tr\barEalphal{l}\bm{u}.
  \end{split}
\end{equation}

Now the right-hand side of \eqref{eqapp:SBP1} has been decomposed in terms of contributions to the $\Ghat^{\betal{l}}$ and $\Ghat^{\alphal{l}}$ surfaces; this is still insufficient 
unless one is happy with imposing the same boundary condition over the entire face. What is needed is the ability to impose boundary conditions 
point-wise. The next decomposition demonstrates that a point-wise interpretation of the telescoping property is possible:
\begin{equation}\label{eqapp:SBP3}
  \begin{split}
  \bm{v}\Tr\barM\barDxil{l}\bm{u}+\bm{u}\barM\barDxil{l}\bm{v}=
  \left(\barRbetal{l}\bm{v}\right)\Tr\barPorthol{l}\barRbetal{l}\bm{u}-\left(\barRalphal{l}\bm{v}\right)\Tr\barPorthol{l}\barRalphal{l}\bm{u}.
  \end{split}
\end{equation}
To see why the right-hand side is separable, point-wise, consider the term $ \left(\barRbetal{l}\bm{v}\right)\Tr\barPorthol{l}\barRbetal{l}\bm{u}$. 
The action of $\Ralphal{l}$ is to interpolate $\bm{v}$ or $\bm{u}$ to the $\Ghat_{\betal{l}}$ boundary. Defining the vectors 
$\bm{v}_{\Ghat_{\betal{l}}}\equiv\barRbetal{l}\bm{v}$  and $\bm{u}_{\Ghat_{\betal{l}}}\equiv\barRbetal{l}\bm{u}$  gives 
\begin{equation}
\left(\barRbetal{l}\bm{v}\right)\Tr\barPorthol{l}\barRbetal{l}\bm{u} = \bm{v}_{\Ghat_{\betal{l}}}\Tr\barPorthol{l}\bm{u}_{\Ghat_{\betal{l}}}.
\end{equation}
Since $\barPorthol{l}$ is diagonal, this means that the surface nodes can be traversed and at each node an appropriate boundary condition 
can be imposed. 

The nonlinear approximations used in this paper, constructed from SBP operators and two-point flux function matrices, poses 
the same type of properties as the above SBP operators. Namely, the result is mimetic of the following (nonlinear) form of integration 
by parts having the telescoping property
\begin{equation}\label{eq:nonIBP}
\frac{1}{2}\sum\limits_{l,m=1}^{3}\int_{\Ohat}\fnc{W}\Tr\left\{\frac{\partial}{\partial\xil{l}}\left(\Jdxildxm{l}{m}\Fxm{m}\right)+\Jdxildxm{l}{m}\frac{\partial\Fxm{m}}{\partial\xil{l}}\right\}\mr{d}\Ohat=
\sum\limits_{l,m=1}^{3}\oint_{\Ghat}\left(\Jdxildxm{l}{m}\Fxm{m}\nxil{m}\right)\mr{d}\Ghat.
\end{equation} 

The starting point is to analyze the properties of the nonlinear approximation. The following is a general result on the accuracy of the 
nonlinear approximations and is an extension of the proofs give in Crean~\etal~\cite{Crean2018}. 
\begin{thrm}\label{thrm:accDxil}
Let $\Dxil{l}^{\kappa}$ be any degree $p$ finite-difference approximation to the first derivative 
$\partial/\partial\xil{l}$, on a set of nodes $\bm{\xi}$. 
Consider a PDE whose fluxes in the $\xm{m}$, $m=1,2,3$, coordinate directions are continuously differentiable functions 
$\Fxm{m}:\mathbb{R}^{5}\rightarrow\mathbb{R}^{5}$ and a variable coefficient $\fnc{A}_{\kappa}$ that 
is sufficiently smooth. If $\fxmsc{m}{\bm{u}}{\bm{v}}:\mathbb{R}^{5}\times\mathbb{R}^{5}\rightarrow\mathbb{R}^{5}$ are 
dyadic functions that are continuously differentiable, symmetric in their arguments, and satisfy 
$\fxmsc{m}{\qki{\kappa}{(i)}}{\qki{\kappa}{(i)}}=\Fxm{m}\left(\qki{\kappa}{(i)}\right)$, then for sufficiently 
smooth solutions $\bfnc{Q}$
\begin{equation}\label{acc:Hadamard}
\begin{split}
&\left(2\Dxil{l}^{\kappa}\left[\fnc{A}\right]_{\kappa}\right)\circ\matFxm{m}{\qk{\kappa}}{\qk{\kappa}}\ones{\kappa}=
\left(\frac{\partial\left(\fnc{A}\bfnc{F}_{\xm{m}}\right)}{\partial \xil{l}}\right)\left(\bm{\xi}_{\kappa}\right)
+\left(\Fxm{m}\frac{\partial\fnc{A}}{\partial\xil{l}}\right)\left(\bm{\xi}_{\kappa}\right)
+\mathcal{O}\left(h^{p+1}\right),\\
&\left(2\left[\fnc{A}\right]_{\kappa}\Dxil{l}^{\kappa}\right)\circ\matFxm{m}{\qk{\kappa}}{\qk{\kappa}}\ones{\kappa}=
\left(\fnc{A}\frac{\partial\Fxm{m}}{\partial\xil{l}}\right)\left(\bm{\xi}_{\kappa}\right)+\mathcal{O}\left(h^{p+1}\right),
\end{split}
\end{equation}
where $\left[\fnc{A}\right]_{\kappa}$ is a diagonal matrix containing, along its diagonal, the evaluation of the variable 
coefficient at the mesh nodes $\bm{\xi}_{\kappa}$, and $h$ is some appropriate measure of the mesh spacing within the 
element. 
\end{thrm}
\begin{proof}
  The proof is given in~\ref{proof:thrm:accDxil}
\end{proof}
Theorem~\ref{thrm:accDxil} implies that 
\begin{equation}\label{acc:Hadamardmetrics}
\begin{split}
&\left(2\Dxil{l}^{\kappa}\matAlmk{l}{m}{\kappa}\right)\circ\matFxm{m}{\qk{\kappa}}{\qk{\kappa}}\ones{\kappa}=
\left(\frac{\partial\left(\Jdxildxm{l}{m}\bfnc{F}_{\xm{m}}\right)}{\partial \xil{l}}\right)\left(\bm{\xi}_{\kappa}\right)
+\left(\Fxm{m}\frac{\partial\fnc{A}}{\partial\xil{l}}\right)\left(\bm{\xi}_{\kappa}\right)
+\mathcal{O}\left(h^{p+d}\right),\\
&\left(2\matAlmk{l}{m}{\kappa}\Dxil{l}^{\kappa}\right)\circ\matFxm{m}{\qk{\kappa}}{\qk{\kappa}}\ones{\kappa}=
\left(\Jdxildxm{l}{m}\frac{\partial\Fxm{m}}{\partial\xil{l}}\right)\left(\bm{\xi}_{\kappa}\right)+\mathcal{O}\left(h^{p+d}\right),
\end{split}
\end{equation}
where the change in the truncation terms comes from the fact that $\Jdxildxm{l}{m}\propto h^{d-1}$. Notice that summing the first equality 
in the three computational directions and using the continuous GCL gives that 
\begin{equation}\label{acc:HadamardmetricsGCL}
\begin{split}
&\sum\limits_{l=1}^{3}\left(2\Dxil{l}^{\kappa}\matAlmk{l}{m}{\kappa}\right)\circ\matFxm{m}{\qk{\kappa}}{\qk{\kappa}}\ones{\kappa}=
\sum\limits_{l=1}^{3}\left(\frac{\partial\left(\Jdxildxm{l}{m}\bfnc{F}_{\xm{m}}\right)}{\partial \xil{l}}\right)\left(\bm{\xi}_{\kappa}\right)
+\mathcal{O}\left(h^{p+d}\right).
\end{split}
\end{equation}

In order to demonstrate that the nonlinear approximation leads to a telescoping approximation to~\ref{eq:nonIBP}, like 
in the analysis of integration by parts, first, the constituent matrices of the nonlinear approximation are characterized as approximations to 
various bilinear forms.

The starting point is the accuracy of the on element surface matrix terms such as ,$\Exil{l}\matAlmk{l}{m}{\kappa}\circ\matFxm{m}{\qk{\kappa}}{\qk{\kappa}}\ones{\kappa}$.
For generality and so that the results of this section can be utilized for the element-wise conservation analysis, the bilinear forms that are consider are 
products of continuous scalar functions, $\fnc{V}$, against one of the components of the derivative of the flux vector $\Fxm{m}$. For this purpose, 
the scalar version of the flux function matrix is introduced as follows: 

\begin{equation}\label{eq:matFxmsca}
  \matFxmscai{m}{\qk{\kappa}}{\qk{r}}{i}\equiv
  \left[
    \begin{array}{ccc}
      \left(\fxmsc{m}{\qki{\kappa}{(1)}}{\qki{r}{(1)}}\right)(i)&\dots&\left(\fxmsc{m}{\qki{\kappa}{(1)}}{\qki{r}{(\Nl{r})}}\right)(i)\\
      \vdots&\vdots\\
\left(\fxmsc{m}{\qki{\kappa}{(\Nl{\kappa})}}{\qki{r}{(1)}}\right)(i)&\dots&\left(\fxmsc{m}{\qki{\kappa}{(\Nl{\kappa})}}{\qki{r}{(\Nl{r})}}\right)(i)
    \end{array}
  \right],\quad i = 1,\dots,5.
\end{equation}
\begin{thrm}\label{thrm:accExil}
Under the same conditions as in Thrm.~\ref{thrm:accDxil} and considering the constituent matrices 
of a degree $p$ SBP operator, $\barDxil{l}^{\kappa}$, with a norm matrix $\barM$ of degree $\pP$ and $\mat{R}$ matrices of degree $r$, 
for all smooth functions $\fnc{V}$
\begin{equation}\label{eq:accExil}
\begin{split}
&\vk{\kappa}\Tr\left(\barExil{l}^{\kappa}\barmatAk{\kappa}\right)\circ\matFxmscai{m}{\qk{\kappa}}{\qk{\kappa}}{i}\barones{\kappa}=
\oint_{\Ghatk}\fnc{V}\fnc{A}\Fxm{m}(i)\nxil{l}\mr{d}\Ghat+\max\left[\mathcal{O}\left(h^{\pP+1}\right),\mathcal{O}\left(h^{r+1}\right)\right],\\
&\vk{\kappa}\Tr\left(\barmatAk{\kappa}\barExil{l}^{\kappa}\right)\circ\matFxmscai{m}{\qk{\kappa}}{\qk{\kappa}}{i}\barones{\kappa}=
\oint_{\Ghatk}\fnc{V}\fnc{A}\Fxm{m}(i)\nxil{l}\mr{d}\Ghat
+\max\left[\mathcal{O}\left(h^{\pP+1}\right),\mathcal{O}\left(h^{r+1}\right)\right],\\
&i=1,\dots,5.
\end{split}
\end{equation}
\end{thrm}
\begin{proof}
The proof is given in Appendix~\ref{proof:thrm:accExil}.
\end{proof}
Theorem~\ref{thrm:accExil} implies that 
\begin{equation}\label{eq:accExilmetrics}
\begin{split}
&\vk{\kappa}\Tr\left(\barExil{l}^{\kappa}\barmatAlmk{l}{m}{\kappa}\right)\circ\matFxmscai{m}{\qk{\kappa}}{\qk{\kappa}}{i}\barones{\kappa}=
\oint_{\Ghatk}\fnc{V}\Jdxildxm{l}{m}\Fxm{m}(i)\nxil{l}\mr{d}\Ghat+\max\left[\mathcal{O}\left(h^{\pP+d}\right),\mathcal{O}\left(h^{r+d}\right)\right],\\
&\vk{\kappa}\Tr\left(\barmatAlmk{l}{m}{\kappa}\barExil{l}^{\kappa}\right)\circ\matFxmscai{m}{\qk{\kappa}}{\qk{\kappa}}{i}\barones{\kappa}=
\oint_{\Ghatk}\fnc{V}\Jdxildxm{l}{m}\Fxm{m}(i)\nxil{l}\mr{d}\Ghat
+\max\left[\mathcal{O}\left(h^{\pP+d}\right),\mathcal{O}\left(h^{r+d}\right)\right],\\
&i=1,\dots,5.
\end{split}
\end{equation}

In order to demonstrate the telescoping flux form, the following general result is necessary
\begin{thrm}\label{thrm:telescope}
Consider the matrix of $\overline{\mat{A}}$ of size $\Nl{\kappa}\times \Nl{r}$ with a tensor extension 
$\mat{A}\equiv\overline{\mat{A}}\otimes\Imat{5}$, and a two argument matrix flux 
function $\matFxm{m}{\qk{\kappa}}{\qk{r}}$ constructed from the two point 
flux function $\fxmsc{m}{\qki{\kappa}{(i)}}{\qki{r}{(j)}}$ that satisfies the Tadmor shuffle condition
\[
\left(\wki{\kappa}{(i)}-\wki{r}{(j)}\right)\Tr\fxmsc{m}{\qki{\kappa}{(i)}}{\qki{r}{(j)}}=
\psixmki{m}{\kappa}{i}-\psixmki{m}{r}{j}
\]
and is symmetric, \ie, $\fxmsc{m}{\qki{\kappa}{(i)}}{\qki{r}{(j)}}= \fxmsc{m}{\qki{r}{(j)}}{\qki{\kappa}{(i)}}$, then
\begin{equation*}
\wk\Tr\left(\mat{A}\right)\circ\matFxm{m}{\qk{\kappa}}{\qk{r}}\ones{r}-
\ones{\kappa}\Tr\mat{A}\circ\matFxm{m}{\qk{\kappa}}{\qk{r}}\wk{r} =
\left(\psixmk{m}{\kappa}\right)\Tr\overline{\mat{A}}\barones{r}-\barones{\kappa}\Tr\overline{\mat{A}}\psixmk{m}{r}.
\end{equation*}
\end{thrm}
\begin{proof}
\begin{equation*}
\begin{split}
  &\wk\Tr\left(\mat{A}\right)\circ\matFxm{m}{\qk{\kappa}}{\qk{r}}\ones{r}-
\ones{\kappa}\Tr\mat{A}\circ\matFxm{m}{\qk{\kappa}}{\qk{r}}\wk{r} =\\
&\sum\limits_{i=1}^{\Nl{\kappa}}\sum\limits_{j=1}^{\Nl{r}}
\left\{
\overline{\mat{A}}(i,j)\left(\wki{\kappa}{(i)}\right)\Tr\matFxm{m}{\qki{\kappa}{(i)}}{\qki{r}{(i)}}-
\overline{\mat{A}}(i,j)\left(\wki{r}{(j)}\right)\Tr\matFxm{m}{\qki{\kappa}{(i)}}{\qki{r}{(i)}}
\right\}=\\
&\sum\limits_{i=1}^{\Nl{\kappa}}\sum\limits_{j=1}^{\Nl{r}}
\left\{
\overline{\mat{A}}(i,j)\left(\wki{\kappa}{(i)}-\wki{r}{(j)}\right)\Tr\matFxm{m}{\qki{\kappa}{(i)}}{\qki{r}{(i)}}
\right\}=\\
&\sum\limits_{i=1}^{\Nl{\kappa}}\sum\limits_{j=1}^{\Nl{r}}
\left\{
\overline{\mat{A}}(i,j)\left(\psixmki{m}{\kappa}{i}-\psixmki{m}{r}{j}\right)
\right\} = \left(\psixmk{m}{\kappa}\right)\Tr\overline{\mat{A}}\barones{r}-\barones{\kappa}\Tr\overline{\mat{A}}\psixmk{m}{r}.
\end{split}
\end{equation*}
\end{proof}

Now it is demonstrated, using~\eqref{acc:Hadamardmetrics}, \eqref{acc:HadamardmetricsGCL}, \eqref{eq:accExilmetrics}, and Theorem \ref{thrm:telescope}, that the nonlinear approximation results 
in a form that is telescoping and mimetic of the nonlinear integration by parts formula~\eqref{eq:nonIBP}. 
Discretizing the left-hand side of~\eqref{eq:nonIBP} using the nonlinear operator gives
\begin{equation}\label{eq:nonSBP1}
  \begin{split}
  &\bm{lhs}\equiv\sum\limits_{l,m=1}^{3}\wk{\kappa}\Tr
  \left(\Qxil{l}\matAlmk{l}{m}{\kappa}+\matAlmk{l}{m}{\kappa}\Qxil{l}\right)\circ\matFxm{m}{\qk{\kappa}}{\qk{\kappa}}\ones{k},
  \end{split}
\end{equation}
where for a nonconforming element the macro element is considered. Using~\eqref{acc:Hadamardmetrics}, \eqref{acc:HadamardmetricsGCL}, then it can be shown that \Eq~\eqref{eq:nonSBP1} is an approximation 
of the left-hand side of~\eqref{eq:nonIBP}, \ie,  
\begin{equation}\label{eq:nonSBP6}
  \begin{split}
  &\sum\limits_{l,m=1}^{3}\wk{\kappa}\Tr
  \left(\Qxil{l}\matAlmk{l}{m}{\kappa}+\matAlmk{l}{m}{\kappa}\Qxil{l}\right)\circ\matFxm{m}{\qk{\kappa}}{\qk{\kappa}}\ones{k}\approx\\
  &\frac{1}{2}\sum\limits_{l,m=1}^{3}\int_{\Ohat}\fnc{W}\Tr\left\{\frac{\partial}{\partial\xil{l}}\left(\Jdxildxm{l}{m}\Fxm{m}\right)+\Jdxildxm{l}{m}\frac{\partial\Fxm{m}}{\partial\xil{l}}\right\}\mr{d}\Ohat.
  \end{split}
\end{equation}
 
Next, we demonstrate that \Eq~\eqref{eq:nonSBP1} reduces to a telescoping and consistent approximation to the right-hand side of~\eqref{eq:nonIBP}. 
Transposing the second term in~\eqref{eq:nonSBP1} and using the symmetry of $\matFxm{m}{\qk{\kappa}}{\qk{\kappa}}$ gives
\begin{equation}\label{eq:nonSBP2}
  \begin{split}
  &\bm{lhs}=\sum\limits_{l,m=1}^{3}
  \left\{\wk{\kappa}\Tr\left(\Qxil{l}\matAlmk{l}{m}{\kappa}\right)\circ\matFxm{m}{\qk{\kappa}}{\qk{\kappa}}\ones{k}
  +\ones{k}\Tr\left(\Qxil{l}\Tr\matAlmk{l}{m}{\kappa}\right)\circ\matFxm{m}{\qk{\kappa}}{\qk{\kappa}}\wk{\kappa}\right\}.
  \end{split}
\end{equation}
Using the SBP property $\Qxil{l}\Tr=-\Qxil{l}+\Exil{l}$ yields
\begin{equation}\label{eq:nonSBP3}
  \begin{split}
  \bm{lhs}=&
  \left\{\wk{\kappa}\Tr\left(\Qxil{l}\matAlmk{l}{m}{\kappa}\right)\circ\matFxm{m}{\qk{\kappa}}{\qk{\kappa}}\ones{k}
  -\ones{k}\Tr\left(\Qxil{l}\matAlmk{l}{m}{\kappa}\right)\circ\matFxm{m}{\qk{\kappa}}{\qk{\kappa}}\wk{\kappa}\right\}\\
   &+\sum\limits_{l,m=1}^{3}\wk{\kappa}\Tr\Exil{l}\matAlmk{l}{m}{\kappa}\circ\matFxm{m}{\qk{\kappa}}{\qk{\kappa}}\ones{k}
  \end{split}
\end{equation}
Applying \Theorem~\ref{thrm:telescope} on the first set of terms results in 
\begin{equation}\label{eq:nonSBP4}
  \begin{split}
  \bm{lhs}=&\sum\limits_{l,m=1}^{3}
  \left\{\left(\psixmk{m}{\kappa}\right)\Tr\barQxil{l}\barmatAlmk{l}{m}{\kappa}\barones{\kappa}-\barones{\kappa}\Tr\barQxil{l}\barmatAlmk{l}{m}{\kappa}\psixmk{m}{\kappa}\right\}\\
  &+\sum\limits_{l,m=1}^{3}\wk{\kappa}\Tr\Exil{l}\matAlmk{l}{m}{\kappa}\circ\matFxm{m}{\qk{\kappa}}{\qk{\kappa}}\ones{k}.
  \end{split}
\end{equation}
The first set of terms on the right-hand side are the discrete GCL conditions and are therefore zero. The second set of terms 
is reduced by using the SBP property $\barQxil{l}=-\barQxil{l}\Tr+\barExil{l}$ and the consistency of the derivative operator which implies that 
$\barQxil{l}\barones{\kappa}=0$. Therefore, \eqref{eq:nonSBP4} reduces to 
\begin{equation}\label{eq:nonSBP5}
  \begin{split}
  \bm{lhs}=&\sum\limits_{l,m=1}^{3}\left\{\wk{\kappa}\Tr\Exil{l}\matAlmk{l}{m}{\kappa}\circ\matFxm{m}{\qk{\kappa}}{\qk{\kappa}}\ones{k}
  -\barones{\kappa}\Tr\barExil{l}\barmatAlmk{l}{m}{\kappa}\psixmk{m}{\kappa}
  \right\}.
  \end{split}
\end{equation}
By~\eqref{eq:accExilmetrics} and the exactness properties of $\barExil{l}$, the right-hand side of~\eqref{eq:nonSBP1} 
is an approximation to the right-hand side of~\eqref{eq:nonIBP}, \ie,  
\begin{equation}\label{eq:nonSBP7}
  \begin{split}
  &\bm{lhs}=\sum\limits_{l,m=1}^{3}\left(\wk{\kappa}\Tr\Exil{l}\matAlmk{l}{m}{\kappa}\circ\matFxm{m}{\qk{\kappa}}{\qk{\kappa}}\ones{\kappa}
  -\barones{\kappa}\Tr\barExil{l}\barmatAlmk{l}{m}{\kappa}\psixmk{m}{\kappa}
  \right)\approx\\
  &
  \sum\limits_{l,m=1}^{3}\oint_{\Ghat}\left\{\left(\bfnc{W}\Tr\Fxm{m}-\psi_{\xm{m}}\right)\Jdxildxm{l}{m}\nxil{l}\right\}\mr{d}\Ghat=
  \sum\limits_{l,m=1}^{3}\oint_{\Ghat}\left(\Jdxildxm{l}{m}\fxm{m}\nxil{l}\right)\mr{d}\Ghat.
  \end{split}
\end{equation}
Finally, and of critical importance, the right-hand side of~\eqref{eq:nonSBP1} is in a telescopic form which when combined 
with appropriate interface SATs telescopes to the boundaries of the domain. 

The accuracy of the coupling terms in the SATs is necessary to prove that the scheme is element-wise conservative and 
is given below.
\begin{thrm}\label{thrm:weakcouplingmort}
The coupling matrices constructed using the mortar-element approach satisfy the following accuracy conditions:
\begin{equation}\label{eq:weakcouplingmort}
\begin{split}
&\vk{\kappa}\Tr\barEHtoLm{m}\circ\matFxmscai{m}{\qL}{\qH}{i}\barones{\Nl{\rmH}} = 
\oint_{\Ghat^{\rm{L}}}\fnc{V}\Fxm{m}(i)\Jdxildxm{1}{m}\nxil{1}\mr{d}\Ghat+\mathcal{O}\left(h^{\pL+d}\right)\\
&\vk{\kappa}\Tr\barELtoHm{m}\circ\matFxmscai{m}{\qL}{\qH}{i}\barones{\rmL} = 
\oint_{\Ghat^{\rm{H}}}\fnc{V}\Fxm{m}(i)\Jdxildxm{1}{m}\nxil{1}\mr{d}\Ghat+\mathcal{O}\left(h^{\pL+d}\right),
\end{split}
\end{equation}
where in the current context
\begin{equation*}
  \begin{split}
    &\barEHtoLm{m}\equiv\frac{1}{2}\left\{
      \left(\barRbetal{1}^{\rmL}\right)\Tr\barmatAlmkLH{1}{m}{\rmL}{\Ghat^{\rmL}}{red}\barPorthol{1}^{\rmL}\IHtoL\Ralphal{1}^{\rmH}
      +\left(\barRbetal{1}^{\rmL}\right)\Tr\barPorthol{1}^{\rmL}\IHtoL\barmatAlmkLH{1}{m}{\rmH}{\Ghat^{\rmH}}{red}\Ralphal{1}^{\rmH}
      \right\}\\
    &\barELtoHm{m}\equiv-\frac{1}{2}\left\{\left(\barRalphal{1}^{\rmH}\right)\Tr\barmatAlmkLH{1}{m}{\rmH}{\Ghat^{\rmH}}{red}\barPorthol{1}^{\rmH}\ILtoH\Rbetal{1}^{\rmL}
      +\left(\barRbetal{1}^{\rmL}\right)\Tr\barPorthol{1}^{\rmH}\IHtoL\barmatAlmkLH{1}{m}{\rmH}{\Ghat^{\rmH}}{red}\Rbetal{1}^{\rmL}
      \right\}
  \end{split}
\end{equation*}
\end{thrm}
\begin{proof}
The proof follows identically to that given in Thrm.~\ref{thrm:accExil}.
\end{proof}
\section{Proof of Theorem.~\ref{thrm:accDxil}}\label{proof:thrm:accDxil}
In this appendix, we prove that the nonlinear approximations have the following error properties: 
\begin{equation*}
\begin{split}
&\left(2\Dxil{l}^{\kappa}\left[\fnc{A}\right]_{\kappa}\right)\circ\matFxm{m}{\qk{\kappa}}{\qk{\kappa}}\ones{\kappa}=
\left(\frac{\partial\left(\fnc{A}\bfnc{F}_{\xm{m}}\right)}{\partial \xil{l}}\right)\left(\bm{\xi}_{\kappa}\right)
+\left(\Fxm{m}\frac{\partial\fnc{A}}{\partial\xil{l}}\right)\left(\bm{\xi}_{\kappa}\right)
+\mathcal{O}\left(h^{p+1}\right),\\
&\left(2\left[\fnc{A}\right]_{\kappa}\Dxil{l}^{\kappa}\right)\circ\matFxm{m}{\qk{\kappa}}{\qk{\kappa}}\ones{\kappa}=
\left(\fnc{A}\frac{\partial\Fxm{m}}{\partial\xil{l}}\right)\left(\bm{\xi}_{\kappa}\right)+\mathcal{O}\left(h^{p+1}\right).
\end{split}
\end{equation*}
The first error estimate is not intuitive and is the focus of this appendix, while the second error estimate follows 
directly from the error estimate of the nonlinear approximation $\Dxil{l}\circ\matFxm{m}{\qk{\kappa}}{\qk{\kappa}}\ones{\kappa}$, 
which has been derived by several authors (for example see~\cite{Crean2018} \Theorem $1$). The approach that is taken 
is to examine, point-wise, what the action of the derivative operator is. Then a careful examination of the derivative of the 
two-point flux function, taking advantage of its consistency and symmetry, reveals the final error estimate. 

Note that the dyadic flux vector valued function that is a function of two vector valued arguments, which is represented in a 
generic fashion as $\fxmsc{m}{\QL}{\QR}$ (note that the Tadmor shuffle condition~\eqref{eq:shuffle} is not required here). Starting with the first 
equality, 
\begin{equation}\label{eq:Daccone}
\begin{split}
&\left(\left(2\Dxil{l}^{\kappa}\matAk{\kappa}\right)\circ\matFxm{m}{\qk{\kappa}}{\qk{\kappa}}\ones{\kappa}\right)(i)=
\sum\limits_{j=1}^{N_{\kappa}}2\Dxil{l}^{\kappa}(i,j)\fnc{A}(j)\fxmsc{m}{\qki{\kappa}{(i)}}{\qki{\kappa}{(j)}}\\
&=2\left(\frac{\partial\left(\fnc{A}\fxmsc{m}{\QL}{\QR}\right)}{\partial \xil{l}}\right)(\QL=\QR=\qki{\kappa}{(i)})
+\mathcal{O}\left(h^{p+1}\right)\\
&=2\left(\fnc{A}\frac{\partial\fxmsc{m}{\QL}{\QR}}{\partial \xil{l}}\right)\left(\QL=\QR=\qki{\kappa}{(i)}\right)
+2\left(\fxmsc{m}{\QL}{\QR}\frac{\partial\fnc{A}}{\partial \xil{l}}\right)\left(\QL=\QR=\qki{\kappa}{(i)}\right)\\
&+\mathcal{O}\left(h^{p+1}\right)\\
&=2\left(\fnc{A}\frac{\fxmsc{m}{\QL}{\QR}}{\partial \QR}\frac{\partial\QR}{\partial\xil{l}}\right)\left(\QL=\QR=\qki{\kappa}{(i)}\right)
+2\left(\Fxm{m}\frac{\partial\fnc{A}}{\partial \xil{l}}\right)\left(\QL=\QR=\qki{\kappa}{(i)}\right)\\
&+\mathcal{O}\left(h^{p+1}\right)\\
\end{split}
\end{equation}
where the following are used: 1) $\Dxil{l}^{\kappa}$ is a degree $p$ differentiation operator and therefore of order $p+1$ (note that this occurs because the derivative 
that is being approximated is in computational space, see Appendix~\ref{sec:polyorder} for a thorough discussion), and 
2) $\fxmsc{m}{\Q}{\Q}=\Fxm{m}$. Now 
\begin{equation}\label{eq:symmetry}
\begin{split}
\frac{\partial\Fxm{m}\left(\Q\right)}{\partial\xil{l}}=&\frac{\partial\fxmsc{m}{\Q}{\Q}}{\partial\xil{l}}
=\left(\frac{\partial\fxmsc{m}{\QL}{\QR}}{\partial\xil{l}}\right)\left(\QL=\QR=\Q\right)\\\\
=&\left(
\frac{\partial\fxmsc{m}{\QL}{\QR}}{\partial\QL}\frac{\partial\QL}{\partial\xil{l}}
+\frac{\partial\fxmsc{m}{\QL}{\QR}}{\partial\QR}\frac{\partial\QR}{\partial\xil{l}}
\right)\left(\QL=\QR=\Q\right).
\end{split}
\end{equation}
It is now shown that
\begin{equation*}
\left(\frac{\partial\fxmsc{m}{\QL}{\QR}}{\partial\QL}\right)\left(\QL=\QR=\Q\right)
=\left(\frac{\partial\fxmsc{m}{\QL}{\QR}}{\partial\QR}\right)\left(\QL=\QR=\Q\right).
\end{equation*}
For the term on the left of the equality,
\begin{equation*}
  \begin{split}
\frac{\partial\fxmsc{m}{\QL}{\QR}}{\partial\QL}\left(\QL=\QR=\Q\right) &= 
\left(\lim_{\Delta\Q\rightarrow\bm{0}}\frac{\fxmsc{m}{\Q+\Delta\Q}{\QR}}{\Delta \Q}\right)
\left(\QL=\QR=\Q\right)\\
&=\lim_{\Delta\Q\rightarrow\bm{0}}\frac{\fxmsc{m}{\Q+\Delta\Q}{\Q}}{\Delta \Q}.
  \end{split}
\end{equation*}
For the term on the right of the equality,
\begin{equation*}
\begin{split}
\left(\frac{\partial\fxmsc{m}{\QL}{\QR}}{\partial\QR}\right)\left(\QL=\QR=\Q\right) &= 
\left(
\lim_{\Delta\Q\rightarrow\bm{0}}\frac{\fxmsc{m}{\QL}{\QR+\Delta\Q}}{\Delta \Q}
\right)\left(\QL=\QR=\Q\right)\\
&=\lim_{\Delta\Q\rightarrow\bm{0}}\frac{\fxmsc{m}{\Q}{\Q+\Delta\Q}}{\Delta \Q}\\\\
&=\lim_{\Delta\Q\rightarrow\bm{0}}\frac{\fxmsc{m}{\Q+\Delta\Q}{\Q}}{\Delta \Q},
\end{split}
\end{equation*}
where the last equality results from the symmetry of $\bm{f}_{\xm{m}}^{sc}$; thus, \eqref{eq:symmetry} becomes
\begin{equation}\label{eq:symmetrytwo}
\begin{split}
\frac{\partial\Fxm{m}\left(\Q\right)}{\partial\xil{l}}=&\frac{\partial\fxmsc{m}{\Q}{\Q}}{\partial\xil{l}}
=\left(\frac{\partial\fxmsc{m}{\QL}{\QR}}{\partial\xil{l}}\right)\left(\QL=\QR=\Q\right)\\\\
=&2\left(\frac{\partial\fxmsc{m}{\QL}{\QR}}{\partial\QL}\frac{\partial\QL}{\partial\xil{l}}\right)\left(\QL=\QR=\Q\right)
=2\left(\frac{\partial\fxmsc{m}{\QL}{\QR}}{\partial\QR}\frac{\partial\QR}{\partial\xil{l}}\right)\left(\QL=\QR=\Q\right).
\end{split}
\end{equation}
Therefore, by~\eqref{eq:symmetrytwo}, \eqref{eq:Daccone} reduces to
\begin{equation}\label{eq:Dacctwo}
\begin{split}
\left(2\Dxil{l}^{\kappa}\matAk\circ\matFxm{m}{\qk{\kappa}}{\qk{\kappa}}\ones{\kappa}\right)(i)&=
\left(\fnc{A}\frac{\partial\Fxm{m}}{\partial \xil{l}}\right)\left(\Q=\qki{\kappa}{(i)}\right)
+\left(2\Fxm{m}\frac{\partial\fnc{A}}{\partial\xil{l}}\right)\left(\Q=\qki{\kappa}{(i)}\right)
+\mathcal{O}\left(h^{p+1}\right)
\\\\
&=\left(\frac{\partial\fnc{A}\Fxm{m}}{\partial \xil{l}}\right)\left(\Q=\qki{\kappa}{(i)}\right)
+\left(\Fxm{m}\frac{\partial\fnc{A}}{\partial\xil{l}}\right)\left(\Q=\qki{\kappa}{(i)}\right)
+\mathcal{O}\left(h^{p+1}\right).
\end{split}
\end{equation}

The second equality in~\eqref{acc:Hadamard} follows in a similar manner.
\section{On order of polynomial exactness}\label{sec:polyorder}
In this section, a careful analysis is undertaken of the relation between polynomial exactness and order.
To do so, it is convenient to 
examine the accuracy of a degree $p$ approximation of the $\xi$ derivative in one dimension using 
a degree $p$ one-dimensional SBP operator $\DxiloneD{1}$ on the $N$ nodes $\bm{\xi}_{1}$ generated by a linear transformation
\begin{equation}\label{eq:afine}
\xm{1}(\xil{1}) = \frac{h}{2}\xil{1}+\frac{\xR+\xL}{2},
\end{equation}
where $h\equiv\xR-\xL$. The approximation 
at the $i\Th$ node of the derivative of the function $\fnc{F}$ is given by 
\begin{equation*}
\frac{\partial\fnc{F}}{\partial\xi}\left(\bmxi{i}\right)\approx\left(\DxiloneD{1}\bm{f}\right)(i)=
\sum\limits_{j=1}^{N}\DxiloneD{1}(i,j)\bm{f}(j).
\end{equation*}
The error at the $i^{\mathrm{th}}$ node is found by taking the difference between the approximation and the 
exact derivative, giving 
\begin{equation*}
\bm{\mathrm{error}}(\bmxi{i})=\sum\limits_{j=1}^{n}\DxiloneD{l}(i,j)\bm{f}(j)-\frac{\partial\fnc{F}}{\partial\xil{1}}\left(\bmxi{i}\right).
\end{equation*}
Expanding $\fnc{F}$ and its derivative about $\xil{1}=0$ via Taylor series, inserting the result into the above, and using the 
fact that $\DxiloneD{l}$ is degree $p$, after some algebra, results in 
 \begin{equation}\label{eq:error}
\begin{split}
\sum\limits_{j=1}^{N}\DxiloneD{1}(i,j)\bm{f}(j)-\frac{\partial\fnc{F}}{\partial\xil{l}}\left(\bmxi{i}\right)=&
\sum\limits_{j=1}^{N}\DxiloneD{1}(i,j)\sum\limits_{k=0}^{\infty}\frac{\partial^{k}\fnc{F}}{\partial \xil{1}^{k}}\left(\xil{1}=0\right)
\frac{\left(\bmxi{j}\right)^{k}}{k!}
-\sum\limits_{k=0}^{\infty}\frac{\partial^{k+1}\fnc{F}}{\partial\xil{1}^{k+1}}\left(\xil{1}=0\right)\frac{\left(\bmxi{i}\right)^{k}}{k!}
\\
=&\sum\limits_{k=p}^{\infty}\frac{\partial^{k+1}\fnc{F}}{\partial\xil{1}^{k+1}}\left(\xil{1}=0\right)
\left(\sum\limits_{j=1}^{N}\DxiloneD{1}(i,j)\frac{\left(\bmxi{j}\right)^{k+1}}{(k+1)!}-\frac{\left(\bmxi{i}\right)^{k}}{k!}\right).
\end{split}
\end{equation}
To obtain an error estimate from~\eqref{eq:error}, it is necessary to introduce the element size, h, and expand the partials in terms of 
the physical coordinate $\xm{1}$. This is accomplished by take take advantage of the Fa\'a di Bruno formula~\cite{Johnsor2002}
\begin{equation*}
\frac{\partial^{k}\fnc{F}\left(\xm{1}\left(\xil{1}\right)\right)}{\partial\xil{1}^{k}}=
\sum\limits_{m=1}^{k}\frac{\partial^{m}\fnc{F}}{\partial\xm{1}^{m}}
\fnc{B}_{k,m}\left(\frac{\partial\xm{1}}{\partial\xil{1}},\dots,\frac{\partial^{k-m+1}\xm{1}}{\partial \xil{1}^{k-m+1}}\right).
\end{equation*}
The Bell polynomials, $\fnc{B}_{k,m}$, are given by
\begin{equation*}
\begin{split}
&\fnc{B}_{k,m}\left(\frac{\partial\xm{1}}{\partial\xil{1}},\dots,\frac{\partial^{k-m+1}\xm{1}}{\partial \xil{1}^{k-m+1}}\right)\equiv\\
&\sum\frac{k!}{j_{1}!j_{2}!\dots j_{k-m+1}!}\left(\frac{\partial\xm{1}}{\partial\xil{1}}\frac{1}{1!}\right)^{j_{1}}
\left(\frac{\partial^{2}\xm{}1}{\partial\xil{1}^{2}}\frac{1}{2!}\right)^{j_{2}}\dotsm
\left(\frac{\partial^{k-m+1}\xm{1}}{\partial\xil{1}^{k-m+1}}\frac{1}{(k-m+1)!}\right)^{j_{k-m+1}},
\end{split}
\end{equation*}
where the sum is over all positive solutions to 
\begin{equation*}
\begin{split}
\sum\limits_{i=1}^{k-m+1}j_{i}=m,\qquad\sum\limits_{i=1}^{k-m+1}ij_{i}=k.
\end{split}
\end{equation*}
By~\eqref{eq:afine} and the definition of the Bell polynomials, \eqref{eq:error} reduces to
 \begin{equation*}
\sum\limits_{j=1}^{n}\DxiloneD{1}(i,j)\bm{f}(j)-\frac{\partial\fnc{F}}{\partial\xi}\left(\bmxi{i}\right)=
\sum\limits_{k=p}^{\infty}\frac{\partial^{k+1}\fnc{F}}{\partial\xm{1}^{k+1}}\left(\xil{1}=0\right)\left(\frac{h}{2}\right)^{k+1}
\left(\sum\limits_{j=1}^{n}\DxiloneD{1}(i,j)\frac{\left(\bmxi{j}\right)^{k+1}}{(k+1)!}-\frac{\left(\bmxi{i}\right)^{k}}{k!}\right).
 \end{equation*}
The leading truncation error is of order $\mathcal{O}(h^{p+1})$, which is a natural result since the PDE 
itself is scaled by the Jacobian $\fnc{J}\propto\mathcal{O}(h)$. Extrapolating, for problems in $d$ dimensions, 
a degree $p$ differentiation operator has the following error properties:
\begin{equation*}
\Dxil{1}\bm{f}=\frac{\partial\fnc{F}}{\partial\xil{1}}\left(\bm{\xi}\right)+\mathcal{O}\left(h^{p+1}\right).
\end{equation*}
Using a similar analysis, it can be concluded that the degree $r$
interpolation operators that for general $\mat{R}$ (the ones used in the main paper are exact) have error properties
\begin{equation*}
\Ralphal{l}\bm{f}=\fnc{F}\left(\bm{\xi}_{\Ghat^{\alphal{l}}}\right)+\mathcal{O}\left(h^{r+1}\right),\qquad
\Rbetal{l}\bm{f}=\fnc{F}\left(\bm{\xi}_{\Ghat^{\betal{l}}}\right)+\mathcal{O}\left(h^{r+1}\right)
\end{equation*}
where $\bm{\xi}_{\Ghat^{\alphal{l}}}$ and $\bm{\xi}_{\Ghat^{\betal{l}}}$ are the face nodes on the $\alphal{l}$ and $\betal{l}$ surfaces, respectively. 
From the above 
discussion, it is now possible to relate polynomial exactness to order and these relations are used in developing 
error estimates for the nonlinear approximations.

In addition, it is necessary to understand the scaling effect on order of the metric Jacobian and the metric terms; these can be characterized as
\begin{equation}\label{eq:scaling}
\fnc{J}\propto\mathcal{O}\left(h^{d}\right),\qquad\Jdxildxm{l}{m}\propto\mathcal{O}\left(h^{d-1}\right).
\end{equation}
\begin{remark}
One way of seeing how the scaling in the Jacobian and the metric terms arises in a proper mesh refinement sequence, is to break 
the curvilinear coordinate transformation into two steps. In the first, there is an affine transformation from the 
child element to the parent element (this is where the scaling shows up); the second transformation is a curvilinear 
transformation from the parent element to physical space.
\end{remark}
\section{Proof of Thrm.~\ref{thrm:accExil}}\label{proof:thrm:accExil}
It is shown how to construct the first estimate, as the second follows in a similar manner. Using the 
decomposition of the surface matrix gives
\begin{equation}\label{eq:accEdecompose}
\begin{split}
&\vk{\kappa}\Tr\left\{\left(\barExil{l}^{\kappa}\barmatAk{\kappa}\right)\circ\matFxmscai{m}{\qk{\kappa}}{\qk{\kappa}}{i}\barones{\kappa}\right\}= 
\vk{\kappa}\Tr\left\{\left(\left(\barRbetal{l}^{\kappa}\right)\Tr\barPorthol{l}^{\kappa}\barRbetal{l}^{\kappa}\barmatAk{\kappa}\right)\circ\matFxmscai{m}{\qk{\kappa}}{\qk{\kappa}}{i}\barones{\kappa}\right\}\\
&-\vk{\kappa}\Tr\left\{\left(\left(\barRalphal{l}^{\kappa}\right)\Tr\barPorthol{l}^{\kappa}\barRalphal{l}^{\kappa}\barmatAk{\kappa}\right)\circ\matFxmscai{m}{\qk{\kappa}}{\qk{\kappa}}{i}\barones{\kappa}\right\}.
\end{split}
\end{equation}
Concentrating on the first term in the right-hand side of~\eqref{eq:accEdecompose}

\begin{equation*}
\begin{split}
&\vk{\kappa}\Tr\left\{\left(\left(\barRbetal{l}^{\kappa}\right)\Tr\barPorthol{l}^{\kappa}\barRbetal{l}^{\kappa}\barmatAk{\kappa}\right)\circ\matFxmscai{m}{\qk{\kappa}}{\qk{\kappa}}{i}\barones{\kappa}\right\}=\\
&=\sum\limits_{a=1}^{\Nl{\kappa}}\sum\limits_{b=1}^{\Nl{\kappa}}
\vk{\kappa}(a)\left(\left(\barRbetal{l}^{\kappa}\right)\Tr\barPorthol{l}^{\kappa}\barRbetal{l}^{\kappa}\barmatAk{\kappa}\right)(a,b)\fxmsc{m}{\qki{\kappa}{(a)}}{\qki{\kappa}{(b)}}(i)\\
&=\sum\limits_{a=1}^{\Nl{\kappa}}\sum\limits_{c=1}^{\Nl{\Ghat^{\betal{l}}}}
\vk{\kappa}(a)\left(\left(\barRbetal{l}^{\kappa}\right)\Tr\barPorthol{l}^{\kappa}\right)(a,c)
\sum\limits_{b=1}^{\Nl{\kappa}}
\barRalphal{l}^{\kappa}(c,b)\barmatAk{\kappa}(b,b)
\fxmsc{m}{\qki{\kappa}{(a)}}{\qki{\kappa}{(b)}},
\end{split}
\end{equation*}
with $\Nl{\Ghat^{\betal{l}}}$ the number of nodes on face $\Ghat^{\betal{l}}$.

The interpolation operator is of degree $r$ and therefore of order $r+1$, thus,
\begin{equation*}
\begin{split}
&\vk{\kappa}\Tr\left\{\left(\left(\barRbetal{l}^{\kappa}\right)\Tr\barPorthol{l}^{\kappa}\barRbetal{l}^{\kappa}\barmatAk{\kappa}\right)\circ\matFxmscai{m}{\qk{\kappa}}{\qk{\kappa}}{i}\barones{\kappa}\right\}=\\
&\sum\limits_{a=1}^{\Nl{\kappa}}\sum\limits_{c=1}^{\Nl{\Ghat^{\betal{l}}}}
\vk{\kappa}(a)\left(\left(\barRalphal{l}^{\kappa}\right)\Tr\barPorthol{l}^{\kappa}\right)(a,c)
\fnc{A}(\bm{\xi}^{(c)})
\fxmsc{m}{\qki{\kappa}{(a)}}{\Q\left(\bm{\xi}^{(c)}}\right)
+\mathcal{O}\left(h^{r+1}\right),
\end{split}
\end{equation*}
where $\fnc{A}\left(\bm{\xi}^{c}\right)$ and $\Q\left(\bm{\xi}^{(c)}\right)$ are $\fnc{A}$ and $\bfnc{Q}$ 
evaluated at the $c\Th$ node on surface $\Ghat^{\betal{l}}$ of element $\kappa$. Continuing, 
\begin{equation*}
\begin{split}
&\vk{\kappa}\Tr\left\{\left(\left(\barRbetal{l}^{\kappa}\right)\Tr\barPorthol{l}^{\kappa}\barRbetal{l}^{\kappa}\barmatAk{\kappa}\right)\circ\matFxmscai{m}{\qk{\kappa}}{\qk{\kappa}}{i}\barones{\kappa}\right\}=\\
&\sum\limits_{c=1}^{\Ghat^{\betal{l}}}\sum\limits_{d=1}^{\Ghat^{\betal{l}}}
\barPorthol{l}^{\kappa}(d,c)\left(\fnc{A}\right)\left(\bm{\xi}^{(c)}\right)\sum\limits_{a=1}^{\Nl{\kappa}}\barRbetal{l}(d,a)
\vk{\kappa}(a)\fxmsc{m}{\qki{\kappa}{(a)}}{\Q\left(\bm{\xi}^{(c)}\right)}
+\mathcal{O}\left(h^{r+1}\right).
\end{split}
\end{equation*}
By the accuracy of the interpolation operator (\ie, in the general case $\overline{\mat{R}}$ is assumed to be of degree $r$ 
and therefore order $r+1$), noting that $\barPorthol{l}^{\kappa}$ is a diagonal matrix, and 
using the fact that $\fxmsc{m}{\qki{\kappa}{(j)}}{\qki{\kappa}{(j)}}=\Fxm{m}\left(\qki{\kappa}{(j)}\right)$, \ie the $j^{\rm{th}}$ entry in $\Fxm{m}$ evaluated at $\qki{\kappa}{(j)}$,
\begin{equation*}
\begin{split}
&\vk{\kappa}\Tr\left\{\left(\left(\barRbetal{l}^{\kappa}\right)\Tr\barPorthol{l}^{\kappa}\barRbetal{l}^{\kappa}\barmatAk{\kappa}\right)\circ\matFxmscai{m}{\qk{\kappa}}{\qk{\kappa}}{i}\barones{\kappa}\right\}=\\
&\sum\limits_{c=1}^{\Nl{\Ghat^{\betal{l}}}}\barPorthol{l}(c,c)\vk{\kappa}(c)\fnc{A}\left(\bm{\xi}^{(c)}\right)\left(\Fxm{m}\left(\Q\left(\bm{\xi}^{(c)}\right)\right)\right)(i)
+\mathcal{O}\left(h^{r+1}\right).
\end{split}
\end{equation*}
Noting that $\barPorthol{l}$ is a degree $\pP$ approximation to the $L^{2}$ inner 
product over planes orthogonal to $\xil{l}$,
\begin{equation*}
\begin{split}
&\vk{\kappa}\Tr\left\{\left(\left(\barRbetal{l}^{\kappa}\right)\Tr\barPorthol{l}^{\kappa}\barRbetal{l}^{\kappa}\barmatAk{\kappa}\right)\circ\matFxmscai{m}{\qk{\kappa}}{\qk{\kappa}}{i}\barones{\kappa}\right\}=
\oint_{\Ghatk^{\betal{l}}}\fnc{V}\fnc{A}(\Fxm{m})(i)\nxil{l}\mr{d}\Ghat
+\max\left[\mathcal{O}\left(h^{\pP+1}\right),\mathcal{O}\left(h^{r+1}\right)\right].
\end{split}
\end{equation*}
Similarly, 
\begin{equation*}
\begin{split}
&\vk{\kappa}\Tr\left\{\left(\left(\barRalphal{l}^{\kappa}\right)\Tr\barPorthol{l}^{\kappa}\barRalphal{l}^{\kappa}
\barmatAk{\kappa}\right)\circ\matFxmscai{m}{\qk{\kappa}}{\qk{\kappa}}{i}\barones{\kappa}\right\}=
\oint_{\Ghatk^{\alphal{l}}}\fnc{V}\fnc{A}(\Fxm{m})(i)\nxil{l}\mr{d}\Ghat
+\max\left[\mathcal{O}\left(h^{\pP+1}\right),\mathcal{O}\left(h^{r+1}\right)\right].
\end{split}
\end{equation*}
Therefore, via the additive property of integrals, the first equality in~\eqref{eq:accEdecompose} is obtained.

\section{Element-wise conservation}\label{sec:gen_element_wise_conservation}
The non-linear hyperbolic nature of the Euler equations means that in finite 
time non-smooth solutions can result despite being closed with smooth data. To allow for non-smooth solutions, 
it is necessary to consider the weak form of the conservation law

\begin{equation}\label{eq:Eulerweak}
\begin{split}
&\int_{t=0}^{T}\int_{\Ohatk}\left(\bfnc{Q}\frac{\partial\fnc{V}}{\partial t}\Jk+
\sum\limits_{l,m=1}^{3}\Jdxildxm{l}{m}\Fxm{m}\frac{\partial\fnc{V}}{\partial \xil{l}}\right)\mr{d}\Ohat\mr{d}t
-\int_{\Ohatk}\left.\fnc{V}\bfnc{Q}\Jk\right|_{t=0}^{T}\mr{d}\Ohat\\
&-\int_{t=0}^{T}\oint_{\Ghatk}\fnc{V}\sum\limits_{l,m=1}^{3}\Jdxildxm{l}{m}\Fxm{m} \nxil{l}\mr{d}\Ghat\mr{d}t=\bm{0},\\ 
&t>0,\kappa=1,2,\dots,K,
\end{split}
\end{equation}
for all smooth test functions $\fnc{V}$ with compact support.

The weak form supports a restricted class of discontinuous 
solutions which satisfy the jump conditions~\cite{Lax1973}

\begin{equation*}
v[[\bfnc{Q}]]=[[\bfnc{F}_{n}]],
\end{equation*}
where $v$ is the speed of the discontinuity, $\bfnc{F}_{n}$ is the flux normal to the discontinuity, 
and $[[\bfnc{V}]]$ is the jump in $\bfnc{V}$ across the 
discontinuity. The interest is in numerically approximating this restricted set of discontinuous 
solutions satisfying the above jump conditions. Thus, the class of schemes that are of 
interest are those that are a consistent approximation to~\eqref{eq:Eulerweak} for non-smooth 
solutions almost everywhere (consistent for smooth solutions everywhere); this is an essential 
property for demonstrating that if the numerical solution converges then it converges to a solution 
satisfying the weak form almost everywhere~\cite{Lax1960}.
 
In this report, a method of lines approach is used and the focus is on the analysis of the semi-discrete 
equations; thus, rather than use~\eqref{eq:Eulerweak}, conservation 
is discussed in the context of~\eqref{eq:Eulerweak2}.  
\begin{equation}\label{eq:Eulerweak2}
\int_{\Ohatk}\fnc{V}\frac{\partial\bfnc{Q}}{\partial t}\Jk\mr{d}\Ohat
-\int_{\Ohatk}\sum\limits_{l,m=1}^{3}\Jdxildxm{l}{m}\Fxm{m}\frac{\partial\fnc{V}}{\partial \xil{l}}\mr{d}\Ohat
+\oint_{\Ghatk}\fnc{V}\sum\limits_{l,m=1}^{3}\Jdxildxm{l}{m}\Fxm{m} \nxil{l}\mr{d}\Ghat=\bm{0}.
\end{equation}

Conservation of the fully discrete scheme can always be achieved with an appropriate choice of 
time integration scheme, \eg, Euler implicit or explicit.

At its core, the analysis that is used relies on a semi-discrete version of the Lax-Wendroff Thrm.~\cite{Lax1960}. To use this theorem, 
it is necessary to express the scheme in telescoping flux form, which in one dimension,  
at node $j$ over a control volume $\left[x_{j-\frac{1}{2}},x_{j+\frac{1}{2}}\right]$ is given as
\begin{equation*}
\frac{\mr{d}q_{j}}{\mr{d}t} = -\frac{\left(f_{j+\frac{1}{2}}-f_{j-\frac{1}{2}}\right)}{\Delta x},
\end{equation*}
where $f$ is a general flux function at the boundaries of the control volume. The importance of this form 
can be seen by viewing it as a finite-volume discretization and the link to the integral form is immediate, recalling that 
the integral and weak forms are equivalent. 
The Lax-Wendroff Thrm.~\cite{Lax1960} proceeds instead by multiplying by a continuous test function and using summation-by-parts. Assuming that 
the general flux function is reasonably well-behaved, one can show that the limit solution is a solution to  
the weak form almost everywhere~\cite{Lax1960,LeVeque1992}. It is clear then that an essential feature of the analysis 
is to demonstrate that the scheme can be algebraically manipulated into a form that is a consistent approximation 
to the weak form~\eqref{eq:Eulerweak2}.

For the purpose of analysis, a general form of the semidiscrete equations for the $\kappa\Th$ element is introduced,
\begin{equation}\label{eq:gen}
    \begin{split}
    \matJk{\kappa}\frac{\mr{d}\qk{\kappa}}{\mr{d}t}&+
    \sum\limits_{l,m=1}^{3}\left(\Dxil{l}^{\kappa}\matAlmk{l}{m}{\kappa}+\matAlmk{l}{m}{\kappa}\Dxil{l}^{\kappa}\right)\circ
    \matFxm{m}{\qk{\kappa}}{\qk{\kappa}}\ones{\kappa}=\\\
    &\left(\M^{\kappa}\right)^{-1}\left\{\sum\limits_{l,m=1}^{3}\left(\Exil{l}^{\kappa}\matAlmk{l}{m}{\kappa}\right)\circ
    \matFxm{m}{\qk{\kappa}}{\qk{\kappa}}\ones{\kappa}
    +\sum\limits_{f=1}^{6}\sum\limits_{m=1}^{3}\Efm{f}{\kappa}{f}{m}\circ\matFxm{m}{\qk{\kappa}}{\qk{f}}\ones{f}\right\},
    \end{split}
\end{equation}
where $\Efm{f}{\kappa}{f}{m}$ are the coupling matrices acting on the six faces, ($f$) with orthogonal coordinate direction $\xil{f}$, of the hexahedral element in the three Cartesian directions ($m$), 
and $\qk{f}$ is the solution from the element touching face $f$. 
In the simple example in the paper, these are the matrices $-\ELtoHm{m}$ and $\EHtoLm{m}$ (note that the $\nxil{l}$ component 
of the unit normal has been absorbed into the definition of $\Efm{f}{\kappa}{f}{m}$).

Moreover, for the conservation proofs it is convenient to introduce the scalar version of the semidiscrete equations for the $\kappa\Th$ element is introduced,
\begin{equation}\label{eq:DEulerCCSsecond}
    \begin{split}
    \barmatJk{\kappa}\frac{\mr{d}\qki{\kappa}{[i]}}{\mr{d}t}&+
    \sum\limits_{l,m=1}^{3}\left(\barDxil{l}^{\kappa}\barmatAlmk{l}{m}{\kappa}+\barmatAlmk{l}{m}{\kappa}
    \barDxil{l}^{\kappa}\right)\circ
    \matFxmi{m}{\qk{\kappa}}{\qk{\kappa}}{[i]}\barones{\kappa}=\\\
    &\left(\barM^{\kappa}\right)^{-1}\left\{\sum\limits_{l,m=1}^{3}\left(\barExil{l}^{\kappa}\barmatAlmk{l}{m}{\kappa}\right)\circ
    \matFxmi{m}{\qk{\kappa}}{\qk{\kappa}}{[i]}\barones{\kappa}
    +\sum\limits_{f=1}^{6}\sum\limits_{m=1}^{3}\barEfm{f}{\kappa}{f}{m}\circ
    \matFxmi{m}{\qk{\kappa}}{\qk{f}}{[i]}\barones{f}\right\},\quad i = 1,\dots,5.
    \end{split}
\end{equation}
In this report, element-wise conservation is proven by demonstrating that the scheme has a telescoping 
flux form at the element level. Recently, Shi and Shu~\cite{Shi2018} have presented an extension of the Lax-Wendroff Thrm. to 
consider element-wise conservation for general multidimensional discretizations. We rely on the proofs in that paper 
and therefore need only show that the semi-discrete equations satisfy the following:
\begin{itemize}
\item Telescoping form: That the scheme can be algebraically manipulated into a general telescoping flux form 
at the element level given as
\begin{equation}\label{eq:gencon}
\frac{\mr{d}\qbark}{\mr{d}t}+\sum\limits_{f=1}^{6}\gfk{f}{\kappa} = 0,
\end{equation}
where $\qbark$ is a generalized locally conserved quantity and $\gfk{f}{\kappa}$ 
is a generalized flux on the $f$ face of the $\kappa\Th$ hexahedral element. In order to 
telescope, the fluxes must be uniquely defined at each surface. Thus, for the two element 
example, $\gL=-\gH$. 
\item Consistency: For a constant flow $\bfnc{Q}=\Qc$,
\begin{equation}\label{eq:consitency}
\begin{split}
&\qbark=\left(\int_{\Ohatk}\Jk\mr{d}\Ohat+\mathcal{O}\left(h\right)\right)\Qc,\\
&\gfk{f}{\kappa} =\left(\sum\limits_{m=1}^{3}\oint_{\Ghatk^{f}}\Jdxildxm{l}{m}\nxil{l}\mr{d}\Ghat
+\mathcal{O}\left(h\right)\right)\Fxm{m}\left(\Qc\right)
\end{split}
\end{equation}
where $\Ghatk^{f}$ is the $f\Th$ face of element $\kappa$, and 
where in contrast to Shi and Shu~\cite{Shi2018}, error terms have been added to account for the 
curvilinear coordinate transformation (this results because, in general, the curvilinear coordinate 
transformation bumps terms in the discretization outside of the polynomial space that can be resolved 
by the discrete integrals); this addition has no impact on the proofs presented in~\cite{Shi2018}.
\item Boundedness: the generalized conserved quantity and fluxes are bounded in terms of the $L^{\infty}$ 
norm of the numerical solution:
\begin{equation}\label{eq:bounded}
\begin{split}
&\left|\qbark(i)-\vbark(i)\right|\leq Ch^{d}\|\qh-\vh\|_{L^{\infty}(\Ballk)},\\
&\left|(\gfk{f}{\kappa}\left(\qh\right))(i)-(\gfk{f}{\kappa}\left(\vh\right))(i)\right|\leq 
Ch^{d-1}\|\qk{\kappa}-\vk{\kappa}\|_{L^{\infty}(\Ballk)},\\
\end{split}
\end{equation}
where $\qh$ and $\vh$ are numerical solutions over the entire mesh and $C$ is some positive constant. Moreover, 
$\Ballk\equiv\left\{\bm{x}\in\mathbb{R}^{d}:\,\left|\bm{x}-\bm{x}_{c}\right|< ch\right\}$, $\bm{x}_{c}$ is 
the element center, and $c$ ($>1$) is independent of the mesh size. Note that the $h^{d}$ and $h^{d-1}$ scaling 
originate from the metric Jacobian and metric terms, respectively.
\item Global conservation:
\begin{equation}\label{eq:globalC}
\sum\limits_{\kappa=1}^{K}\qbark(i) = \int_{\Omega}\fnc{Q}(i)\mr{d}\Omega+\mathcal{O}\left(h\right),\qquad i = 1,\dots,5,
\end{equation}
where again, a discretization error has been introduced for the above stated reasons and again this 
has no impact on the proofs in Shi and Shu~\cite{Shi2018}.
\end{itemize}
Now, a theorem is present that delineates the conditions that need to be satisfied by the 
semi-discrete equations so that element-wise conservation is obtained.
\begin{thrm}\label{thrm:genelement}
If the coupling matrices, for example $\EHtoLm{m}$, in the SATs on either side of a given 
interface are the negative transpose of each other, \eg,
\begin{equation*}
\EHtoLm{m}=-\left(\ELtoHm{m}\right)\Tr,
\end{equation*}
 then the semi-discrete 
form~\eqref{eq:gen} can be algebraically manipulated into the general element-wise 
telescoping form~\eqref{eq:gencon} where
\begin{equation*}
\begin{split}
&\qbark(i) \equiv \barones{\kappa}\Tr\barM\barmatJk{\kappa}\qki{\kappa}{[i]}\\
&\gfk{f}{\kappa}(i)\equiv\sum\limits_{m=1}^{3}\barEfm{f}{\kappa}{f}{m}\circ\matFxmscai{m}{\qk{\kappa}}{\qk{f}}{i}\barones{f},\\
&i=1,\dots,5.
\end{split}
\end{equation*}
The scheme is element-wise conservative if, in addition, for a constant state $\Qc$ the coupling terms satisfy
\begin{equation}\label{eq:consistencycoupling}
\begin{split}
&\barones{\kappa}\Tr\barEfm{f}{\kappa}{f}{m}\circ\circ\matFxmscai{m}{\qk{\kappa}}{\qk{f}}{i}\barones{f}=
\left(\oint_{\Ghatk^{f}}\Jdxildxm{l}{m}\nxil{l}\mr{d}\Ghat+\mathcal{O}\left(h\right)\right)\Fxm{m}(i)\left(\Qc\right),\\
&i=1,\dots,5.
\end{split}
\end{equation}
\end{thrm}
\begin{proof}
The proof is given in Appendix~\ref{app:conservationsecond}.
\end{proof}
\section{Element-wise conservation}\label{app:conservationsecond}
In the subsections that follow, the various requirements for element-wise conservation 
given in Section~\ref{sec:gen_element_wise_conservation} are proven under the assumptions in 
Thrm.~\ref{thrm:genelement}.
\subsection{Consistency}
It is assumed that the coupling terms satisfy the consistency conditions and therefore what remains is to 
prove that the generalized conservative quantity satisfies the consistency condition. However, this is 
immediate since $\barM^{\kappa}$ is at least a degree $2p-1$ approximation to the $L^{2}$ inner product, that is
for scalar functions $\fnc{V}$ and $\fnc{U}$
\begin{equation*}
\vk{\kappa}\Tr\barM^{\kappa}\uk{\kappa}=\int_{\Ohatk}\fnc{V}\fnc{U}\mr{d}\Ohat+\mathcal{O}\left(h^{2p}\right).
\end{equation*}
For a constant state $\Qc$, $\qki{\kappa}{[i]} = \barones{\kappa}\Qc(i)$ thus,
\begin{equation*}
\qbark(i) = \barones{\kappa}\Tr\barM^{\kappa}\barmatJk{\kappa}\barones{\kappa}\Qc(i) =
\left(\int_{\Ohatk}\Jk\mr{d}\Ohat+\mathcal{O}\left(h^{2p}\right)\right)\Qc(i) .
\end{equation*}
\subsection{Telescoping flux form}
To obtain the telescoping flux form, the semi-discrete form of each of the 
scalar conservation laws is discretely integrated over each element. This is
accomplished by multiplying the scalar semi-discrete forms obtained from~\eqref{eq:DEulerCCSsecond} 
by $\barones{\kappa}\Tr\barM^{\kappa}$, which gives
\begin{equation}\label{eq:conservationproofDEulerCCsecond}
\begin{split}
&\barones{\kappa}\Tr\barM^{\kappa}\barmatJk{\kappa}\frac{\mr{d}\qki{\kappa}{[i]}}{\mr{d}t}
+\barones{\kappa}\Tr\sum\limits_{l,m=1}^{3}\left(\barQxil{l}^{\kappa}\barmatAlmk{l}{m}{\kappa}+
\barmatAlmk{l}{m}{\kappa}\barQxil{l}^{\kappa}\right)\circ\matFxmi{m}{\qk{\kappa}}{\qk{\kappa}}{[i]}\barones{\kappa} =\\
&\barones{\kappa}\Tr\sum\limits_{l,m=1}^{3}\barExil{l}^{\kappa}\barmatAlmk{l}{m}{\kappa}\circ
\matFxmi{m}{\qk{\kappa}}{\qk{\kappa}}{[i]}\barones{\kappa}-\barones{\kappa}\Tr\bm{CT},
\end{split}
\end{equation}
where $\bm{CT}$ are the coupling terms. By the symmetry of $\matFxmi{m}{\qk{\kappa}}{\qk{\kappa}}{[i]}$, 
$\barExil{l}^{\kappa}$, and $\barmatAlmk{l}{m}{\kappa}$, \eqref{eq:conservationproofDEulerCCsecond} reduces to 
\begin{equation*}
\begin{split}
&\barones{\kappa}\Tr\barM^{\kappa}\barmatJk{\kappa}\frac{\mr{d}\qki{\kappa}{[i]}}{\mr{d}t}
+\barones{\kappa}\Tr\sum\limits_{l,m=1}^{3}\left(\barSxil{l}^{\kappa}\barmatAlmk{l}{m}{\kappa}+
\barmatAlmk{l}{m}{\kappa}\barSxil{l}^{\kappa}\right)\circ\matFxmi{m}{\qk{\kappa}}{\qk{\kappa}}{[i]}\barones{\kappa} =\\
&-\barones{\kappa}\Tr\bm{CT},
\end{split}
\end{equation*}

The matrix $\left(\left(\barSxil{l}^{\kappa}\barmatAlmk{l}{m}{\kappa}+
\barmatAlmk{l}{m}{\kappa}\barSxil{l}^{\kappa}\right)\circ\matFxmi{m}{\qk{\kappa}}{\qk{\kappa}}{[i]}\right)$ is 
skew-symmetric. Rearrangement and explicitly writing out the coupling terms, yields
\begin{equation}\label{eq:FVfirst}
\begin{split}
&\barones{\kappa}\Tr\barM^{\kappa}\barmatJk{\kappa}\frac{\mr{d}\qki{\kappa}{[i]}}{\mr{d}t}+
\barones{\kappa}\Tr\sum\limits_{f=1}^{6}\sum\limits_{m=1}^{3}\barEfm{f}{m}{\kappa}\circ\matFxmi{m}{\qk{\kappa}}{\qk{f}}{[i]}\barones{f}.
=0,
\end{split}
\end{equation}
which is in the form of~\eqref{eq:gencon} with 
\begin{equation*}
\begin{split}
&\qbark(i) \equiv \barones{\kappa}\Tr\barM^{\kappa}\barmatJk{\kappa}\qki{\kappa}{[i]},\\
&\gfk{f}{\kappa}\equiv\barones{\kappa}\Tr\sum\limits_{m=1}^{3}\barEfm{f}{m}{\kappa}\circ\matFxmi{m}{\qk{\kappa}}{\qk{f}}{[i]}\barones{f},\\
&i=1,\dots,5.
\end{split}
\end{equation*}
What remains to be shown is that the flux at the element 
boundaries is unique, or equivalently that the contributions from two abutting elements; this can be 
readily demonstrated by considering the coupling terms on the simple two element example 
and using the SBP preserving property of the interpolation operators. 
\subsection{Boundedness}
The boundedness estimate on the generalized conserved quantity can be shown as follows:
\begin{equation*}
\begin{split}
\left|\qbark(i)-\vbark(i)\right|=&\left|\barones{\kappa}\Tr\barM^{\kappa}\barmatJk{\kappa}
\left(\qki{\kappa}{[i]}-\vki{\kappa}{[i]}\right)\right|\\
&=\left|\sum\limits_{j=1}^{\Nl{\kappa}}\barM^{\kappa}(j,j)\barmatJk{\kappa}(j,j)
\left(\qki{\kappa}{[i]}(j)-\vki{\kappa}{[i]}(j)\right)\right|\\
&\leq h^{d}\sum\limits_{j=1}^{\Nl{\kappa}}\barM^{\kappa}(j,j)\left|\qki{\kappa}{[i]}(j)-\vki{\kappa}{[i]}(j)\right|\\
&\leq Ch^{d}\max\left(\left|\qki{\kappa}{[i]}(j)-\vki{\kappa}{[i]}(j)\right|\right)\\
&=Ch^{d}\|\qki{\kappa}{[i]}-\vki{\kappa}{[i]}\|_{L^{\infty}}
\leq Ch^{d}\|\qh-\vh\|_{L^{\infty}(\Ballk)},
\end{split}
\end{equation*}
where the scaling comes from the fact that $\Jk\propto h^{d}$.

The generalized flux is constructed from linear combinations of the two-point flux function, 
for which the Ishmael-Roe flux has been shown to be continuously differentiable with respect to its arguments
(see Crean \etal~\cite{Crean2018}) and therefore the generalized flux is bounded in the $L^{\infty}$ norm, where 
the scaling in the inequalities comes from the fact that $\Jdxildxm{l}{m}\propto h^{d-1}$.
\subsection{Global conservation}
The matrix norm is an $L^{2}$ discrete inner product and naturally leads to global conservation;  thus,
\begin{equation*}
\begin{split}
\sum\limits_{\kappa=1}^{k}\qbark(i) =& \sum\limits_{\kappa=1}^{K}\barones{\kappa}\Tr\barM^{\kappa}
\barmatJk{\kappa}\qki{\kappa}{[i]}\\
=&\sum\limits_{\kappa=1}^{K}\int_{\Ohatk}\bfnc{Q}(i)\Jk\mr{d}\Ohat+\mathcal{O}\left(h^{2p}\right)\\
=&\int_{\Omega}\bfnc{Q}(i)\mr{d}\Omega+\mathcal{O}\left(h^{2p}\right).
\end{split}
\end{equation*}

\section{Proof of Thrm.~\ref{thrm:coup}}\label{app:coup}
For simplicity, the proof given here is in terms of the Thomas and Lombard~\cite{Thomas1979} approximate 
metrics; the proof for the symmetric metrics of Vinokur and Yee~\cite{Vinokur2002a}, follows identically.

The Thomas Lombard approximate metrics as well as what they approximate are given below.
\begin{equation}
\begin{split}
\barmatAlmk{1}{1}{\kappa}&=
\barDxil{3}\diag(\bm{x}_{3})\barDxil{2}\bm{x}_{2}
 -\barDxil{2}\diag(\bm{x}_{3})\barDxil{3}\bm{x}_{2}\\
 \approx&\left(\frac{\partial\xm{3}}{\partial\xil{3}}\frac{\partial\xm{2}}{\partial\xi{2}}
       -\frac{\partial\xm{3}}{\partial\xil{2}}\frac{\partial\xm{2}}{\partial\xil{3}}\right)(\bm{\xi})
       =\left(\frac{\partial}{\partial\xil{3}}\left(\xm{3}\frac{\partial\xm{2}}{\partial\xil{2}}\right)
       -\frac{\partial}{\partial\xil{2}}\left(\xm{3}\frac{\partial\xm{2}}{\partial\xil{3}}\right)\right)(\bm{\xi})
,\\
\barmatAlmk{1}{2}{\kappa}
&=\barDxil{3}\diag(\bm{x}_{1})\barDxil{2}\bm{x}_{3}
                                                          -\barDxil{2}\diag(\bm{x}_{1})\barDxil{3}\bm{x}_{3}\\
 \approx&\left(\frac{\partial\xm{1}}{\partial\xil{3}}\frac{\partial\xm{3}}{\partial\xil{2}}
       -\frac{\partial\xm{1}}{\partial\xil{2}}\frac{\partial\xm{3}}{\partial\xil{3}}\right)_{\Ck}
       =\left(\frac{\partial}{\partial\xil{3}}\left(\xm{1}\frac{\partial\xm{3}}{\partial\xil{2}}\right)
       -\frac{\partial}{\partial\xil{2}}\left(\xm{1}\frac{\partial\xm{3}}{\partial\xil{3}}\right)\right)(\bm{\xi})
,\\
\barmatAlmk{1}{3}{\kappa}
&=
\barDxil{3}\diag(\bm{x}_{2})\barDxil{2}\bm{x}_{1}
                                                            -\barDxil{2}\diag(\bm{x}_{2})\barDxil{3}\bm{x}_{1}\\
 \approx&\left(\frac{\partial\xm{2}}{\partial\xil{3}}\frac{\partial\xm{1}}{\partial\xil{2}}
       -\frac{\partial\xm{2}}{\partial\xil{2}}\frac{\partial\xm{1}}{\partial\xil{3}}\right)(\bm{\xi})
       =\left(\frac{\partial}{\partial\xil{3}}\left(\xm{2}\frac{\partial\xm{1}}{\partial\xil{2}}\right)
       -\frac{\partial}{\partial\xil{2}}\left(\xm{2}\frac{\partial\xm{1}}{\partial\xil{3}}\right)\right)(\bm{\xi})
,\\
\barmatAlmk{2}{1}{\kappa}
&=
\barDxil{1}\diag(\bm{x}_{3})\barDxil{3}\bm{x}_{2}
                                                         -\barDxil{3}\diag(\bm{x}_{3})\barDxil{1}\bm{x}_{2}\\
 \approx&\left(\frac{\partial\xm{3}}{\partial\xil{1}}\frac{\partial\xm{2}}{\partial\xil{3}}
       -\frac{\partial\xm{3}}{\partial\xil{3}}\frac{\partial\xm{2}}{\partial\xil{1}}\right)(\bm{\xi})
       =\left(\frac{\partial}{\partial\xil{1}}\left(\xm{3}\frac{\partial\xm{2}}{\partial\xil{3}}\right)
       -\frac{\partial}{\partial\xil{3}}\left(\xm{3}\frac{\partial\xm{2}}{\partial\xil{1}}\right)\right)(\bm{\xi})
,\\
\barmatAlmk{2}{2}{\kappa}
&=\barDxil{1}\diag(\bm{x}_{1})\barDxil{3}\bm{x}_{3}
                                                         -\barDxil{3}\diag(\bm{x}_{1})\barDxil{1}\bm{x}_{3}\\
 \approx&\left(\frac{\partial\xm{1}}{\partial\xil{1}}\frac{\partial\xm{3}}{\partial\xil{3}}
       -\frac{\partial\xm{1}}{\partial\xil{3}}\frac{\partial\xm{3}}{\partial\xil{1}}\right)(\bm{\xi})
       =\left(\frac{\partial}{\partial\xil{1}}\left(\xm{1}\frac{\partial\xm{3}}{\partial\xil{3}}\right)
       -\frac{\partial}{\partial\xil{3}}\left(\xm{1}\frac{\partial\xm{3}}{\partial\xil{1}}\right)\right)(\bm{\xi})
,\\
\barmatAlmk{2}{3}{\kappa}&=
\barDxil{1}\diag(\bm{x}_{2})\barDxil{3}\bm{x}_{1}
                                                         -\barDxil{3}\diag(\bm{x}_{2})\barDxil{1}\bm{x}_{1}\\
 \approx&\left(\frac{\partial\xm{2}}{\partial\xil{1}}\frac{\partial\xm{1}}{\partial\xil{3}}
       -\frac{\partial\xm{2}}{\partial\xil{3}}\frac{\partial\xm{1}}{\partial\xil{1}}\right)(\bm{\xi})
       =\left(\frac{\partial}{\partial\xil{1}}\left(\xm{2}\frac{\partial\xm{1}}{\partial\xil{3}}\right)
       -\frac{\partial}{\partial\xil{3}}\left(\xm{2}\frac{\partial\xm{1}}{\partial\xil{1}}\right)\right)(\bm{\xi})
,\\
\barmatAlmk{3}{1}{\kappa}
&=\barDxil{2}\diag(\bm{x}_{3})\barDxil{1}\bm{x}_{2}
                                                         -\barDxil{1}\diag(\bm{x}_{3})\barDxil{2}\bm{x}_{2}\\
 \approx&\left(\frac{\partial\xm{3}}{\partial\xil{2}}\frac{\partial\xm{2}}{\partial\xil{1}}
       -\frac{\partial\xm{3}}{\partial\xil{1}}\frac{\partial\xm{2}}{\partial\xm{2}}\right)(\bm{\xi})
       =\left(\frac{\partial}{\partial\xil{2}}\left(\xm{3}\frac{\partial\xm{2}}{\partial\xil{1}}\right)
       -\frac{\partial}{\partial\xil{1}}\left(\xm{3}\frac{\partial\xm{2}}{\partial\xil{2}}\right)\right)(\bm{\xi})
,\\
\barmatAlmk{3}{2}{\kappa}
&=\Dxil{2}\diag(\bm{x}_{1})\Dxil{1}\bm{x}_{2}
                                                         -\Dxil{1}\diag(\bm{x}_{1})\Dxil{2}\bm{x}_{2}\\
 \approx&\left(\frac{\partial\xm{1}}{\partial\xil{2}}\frac{\partial\xm{2}}{\partial\xil{1}}
       -\frac{\partial\xm{1}}{\partial\xil{1}}\frac{\partial\xm{2}}{\partial\xm{2}}\right)(\bm{\xi})
       =\left(\frac{\partial}{\partial\xil{2}}\left(\xm{1}\frac{\partial\xm{2}}{\partial\xil{1}}\right)
       -\frac{\partial}{\partial\xil{1}}\left(\xm{1}\frac{\partial\xm{2}}{\partial\xil{2}}\right)\right)(\bm{\xi})
,\\
\barmatAlmk{3}{3}{\kappa}
&=\Dxil{2}\diag(\bm{x}_{2})\Dxil{1}\bm{x}_{1}
                                                         -\Dxil{1}\diag(\bm{x}_{2})\Dxil{2}\bm{x}_{1}\\
 \approx&\left(\frac{\partial\xm{2}}{\partial\xil{2}}\frac{\partial\xm{1}}{\partial\xil{1}}
       -\frac{\partial\xm{2}}{\partial\xil{1}}\frac{\partial\xm{1}}{\partial\xm{2}}\right)(\bm{\xi})
       =\left(\frac{\partial}{\partial\xil{2}}\left(\xm{2}\frac{\partial\xm{1}}{\partial\xil{1}}\right)
       -\frac{\partial}{\partial\xil{1}}\left(\xm{2}\frac{\partial\xm{1}}{\partial\xil{2}}\right)\right)(\bm{\xi}).
\end{split}
\end{equation}
In order for the condition $\ones{\kappa}\Tr\bm{c}_{m}^{\kappa}=0$ to be met, the coupling 
terms on the pseudo mortar need to match the analytical terms in the integration of the 
GCL conditions for example 
\begin{equation*}
  \barones{\kappa}\Tr
  \barEfm{f}{\kappa}{f}{m}\barones{f}=\oint_{\Ghat^{f}}\Jdxildxm{f}{m}\nxil{f}\mr{d}\Ghat.
\end{equation*}
Consider the $\kappa^{\rm{th}}$ element  as having at least one nonconforming interface but a 
conforming interface on the face perpendicular to where $\xil{1}$ is maximum, abutting a fully conforming 
element (this face will be denoted face $2$). The contribution of the coupling elements to the discrete GCL for $m=1,2,3$, listed in that order, are
\begin{equation*}
\begin{split}
&\barones{\kappa}\Tr\barEfm{2}{\kappa}{1}{1}=
\barones{\kappa}\Tr\barRbetal{1}\Tr\barPorthol{1}\barRalphal{1}
\overline{\left[{\color{red}\fnc{J}\frac{\partial\xil{1}}{\partial\xm{1}}}\right]},\qquad
\barones{\kappa}\Tr\barEfm{2}{\kappa}{1}{2}=
\barones{\kappa}\Tr\barRbetal{1}\Tr\barPorthol{1}\barRalphal{1}
\overline{\left[{\color{red}\fnc{J}\frac{\partial\xil{1}}{\partial\xm{2}}}\right]},\\
&\barones{\kappa}\Tr\barEfm{2}{\kappa}{1}{3}=
\barones{\kappa}\Tr\barRbetal{1}\Tr\barPorthol{1}\barRalphal{1}
\overline{\left[{\color{red}\fnc{J}\frac{\partial\xil{1}}{\partial\xm{3}}}\right]},
\end{split}
\end{equation*}
where the metric terms, for example $\overline{\left[{\color{red}\fnc{J}\frac{\partial\xil{1}}{\partial\xm{1}}}\right]}$, 
are those computed in the fully conforming element using the Thomas and Lombard~\cite{Thomas1979} approach. 
What needs to be shown is that each term is exact, \ie,
\begin{equation*}
\begin{split}
 \barones{\kappa}\Tr\barEfm{2}{\kappa}{1}{1}&=\oint_{\Ghat_{2}}\fnc{J}\frac{\partial\xil{1}}{\partial\xm{1}}\mr{d}\Ghat
=\oint_{\Ghat_{2}}\left(\frac{\partial\xm{3}}{\partial\xil{3}}\frac{\partial\xm{2}}{\partial\xil{2}}
       -\frac{\partial\xm{3}}{\partial\xil{2}}\frac{\partial\xm{2}}{\partial\xil{3}}\right)\mr{d}\Ghat
=\oint_{\Gamma_{1}}\left(\frac{\partial}{\partial\xil{3}}\left(\xm{3}\frac{\partial\xm{2}}{\partial\xil{2}}\right)
       -\frac{\partial}{\partial\xil{2}}\left(\xm{3}\frac{\partial\xm{2}}{\partial\xil{3}}\right)\right)\mr{d}\Ghat\\
&=\int_{\xil{2}=-1}^{1}\left(\left.\xm{3}\frac{\partial\xm{2}}{\partial\xil{2}}\right|_{\xil{3}=-1}^{1}\right)\mr{d}\xil{2}-
  \int_{\xil{3}=-1}^{1}\left(\left.\xm{3}\frac{\partial\xm{2}}{\partial\xil{3}}\right|_{\xil{2}=-1}^{1}\right)\mr{d}\xil{3}
,\\
 \barones{\kappa}\Tr\barEfm{2}{\kappa}{1}{2}&=
\oint_{\Ghat_{2}}\fnc{J}\frac{\partial\xil{1}}{\partial\xm{2}}\mr{d}\Ghat=\oint_{\Ghat_{2}}\left(\frac{\partial}{\partial\xil{3}}\left(\xm{1}\frac{\partial\xm{3}}{\partial\xil{2}}\right)
       -\frac{\partial}{\partial\xil{2}}\left(\xm{1}\frac{\partial\xm{3}}{\partial\xil{3}}\right)\right)\mr{d}\Ghat\\
&=\int_{\xil{2}=-1}^{1}\left(\left.\xm{1}\frac{\partial\xm{3}}{\partial\xil{2}}\right|_{\xil{3}=-1}^{1}\right)\mr{d}\xil{2}-
  \int_{\xil{3}=-1}^{1}\left(\left.\xm{1}\frac{\partial\xm{3}}{\partial\xil{3}}\right|_{\xil{2}=-1}^{1}\right)\mr{d}\xil{3}
,\\
 \barones{\kappa}\Tr\barEfm{2}{\kappa}{1}{3}&=
\oint_{\Ghat_{2}}\fnc{J}\frac{\partial\xil{1}}{\partial\xm{3}}\mr{d}\Ghat=\oint_{\Ghat_{2}}\left(\frac{\partial}{\partial\xil{3}}\left(\xm{2}\frac{\partial\xm{1}}{\partial\xil{2}}\right)
       -\frac{\partial}{\partial\xil{2}}\left(\xm{2}\frac{\partial\xm{1}}{\partial\xil{3}}\right)\right)\mr{d}\Ghat\\
&=\int_{\xil{2}=-1}^{1}\left(\left.\xm{2}\frac{\partial\xm{1}}{\partial\xil{2}}\right|_{\xil{3}=-1}^{1}\right)\mr{d}\xil{2}-
  \int_{\xil{3}=-1}^{1}\left(\left.\xm{2}\frac{\partial\xm{1}}{\partial\xil{3}}\right|_{\xil{2}=-1}^{1}\right)\mr{d}\xil{3}.
\end{split}
\end{equation*}
Inserting the Thomas Lombard approximation in the first coupling term
\begin{equation*}
\begin{split}
&\barones{\kappa}\Tr\barEfm{2}{\kappa}{1}{1}=
\barones{\kappa}\Tr\barRbetal{1}\Tr\barPorthol{1}\barRalphal{1}\left(\barDxil{3}
\diag(\bm{x}_{3})\barDxil{2}\bm{x}_{2}
-\barDxil{2}\diag(\bm{x}_{3})\barDxil{3}\bm{x}_{2}\right)\\
&=\barones{\kappa}\Tr\left\{\left(\eNl{1}\eonel{1}\Tr\otimes\PxiloneD{}\otimes\QxiloneD{}\right)\diag(\bm{x}_{3})\barDxil{2}\bm{x}_{2}
-\left(\eNl{1}\eonel{1}\Tr\otimes\QxiloneD{}\otimes\PxiloneD{}\right)\diag(\bm{x}_{3})\barDxil{3}\bm{x}_{2}
\right\}\\
&=\left\{\bm{1}\Tr\left(\eNl{1}\eonel{1}\Tr\otimes\bm{1}\Tr\PxiloneD{}\otimes\bm{1}\Tr\QxiloneD{}\right)\diag(\bm{x}_{3})\barDxil{2}\bm{x}_{2}
-\left(\bm{1}\Tr\eNl{1}\eonel{1}\Tr\otimes\bm{1}\Tr\QxiloneD{}\otimes\bm{1}\Tr\PxiloneD{}\right)\diag(\bm{x}_{3})\barDxil{3}\bm{x}_{2}
\right\}\\
&=\left\{\bm{1}\Tr\left(\eNl{1}\eonel{1}\Tr\otimes\bm{1}\Tr\PxiloneD{}\otimes\bm{1}\Tr\ExiloneD{}\right)\diag(\bm{x}_{3})\barDxil{2}\bm{x}_{2}
-\left(\bm{1}\Tr\eNl{1}\eonel{1}\Tr\otimes\bm{1}\Tr\ExiloneD{}\otimes\bm{1}\Tr\PxiloneD{}\right)\diag(\bm{x}_{3})\barDxil{3}\bm{x}_{2}
\right\}
\end{split}
\end{equation*}
From the last equality, it can be concluded that 
\begin{equation*}
 \barones{\kappa}\Tr\barEfm{2}{\kappa}{1}{1}=\oint_{\Gamma_{1}}\left(\frac{\partial}{\partial\xil{3}}\left(\xm{3}\frac{\partial\xm{2}}{\partial\xil{2}}\right)
       -\frac{\partial}{\partial\xil{2}}\left(\xm{3}\frac{\partial\xm{2}}{\partial\xil{3}}\right)\right)\mr{d}\Ghat,
\end{equation*}
because the projection operators pick off the functions at the boundary and the $\mat{E}$ matrices are at least degree $2p-1$, which is the degree 
of the terms that they are integrating if the curvilinear coordinate transformation is constructed from degree $p$ tensor products. The remaining coupling terms are shown to 
exactly equal the required surface/line integrals.
}{}

\end{document}